\documentclass[11pt]{amsart}

\usepackage{tikz}
\tikzstyle{startstop} = [rectangle, rounded corners, minimum width=1cm, minimum height=1cm,text centered, draw=black]
\tikzstyle{io} = [rectangle, rounded corners, minimum width=1cm, minimum height=1cm,text centered, draw=black]
\tikzstyle{process} = [rectangle, rounded corners, minimum width=1cm, minimum height=1cm,text centered, draw=black]
\tikzstyle{decision} = [rectangle, rounded corners, minimum width=1cm, minimum height=1cm,text centered, draw=black]
\tikzstyle{explain} = [minimum width=1cm, minimum height=0.5cm,text centered, draw=black]
\tikzstyle{equations} = []
\tikzstyle{arrow} = [thick,->,>=stealth]
\tikzstyle{d-arrow} = [thick,->,dashed,>=stealth]
\usepackage{amsfonts}
\usepackage{color}
\usepackage{amssymb,amsthm,amsmath,hyperref}
\usepackage{epsfig}
\usepackage{graphicx,float}
\usepackage{CJK}
\usepackage{setspace}
\usepackage{cancel}
\usepackage{stmaryrd}
\usepackage{verbatim}
\usepackage{ulem}
\usepackage{makecell}
\usepackage{bbm}
\usepackage{enumitem}

\usepackage{scalerel,stackengine}
\stackMath
\newcommand\reallywidehat[1]{%
\savestack{\tmpbox}{\stretchto{%
  \scaleto{%
    \scalerel*[\widthof{\ensuremath{#1}}]{\kern-.6pt\bigwedge\kern-.6pt}%
    {\rule[-\textheight/2]{1ex}{\textheight}}
  }{\textheight}%
}{0.5ex}}%
\stackon[1pt]{#1}{\tmpbox}%
}

\usepackage{amssymb}
\usepackage{scalerel,stackengine}
\usepackage{blindtext}
\usepackage{xcolor}
\usepackage{mathrsfs}
\usepackage{ulem}
\usepackage{stmaryrd}
\usepackage{hyperref}
\usepackage{dsfont}
\allowdisplaybreaks[4]

\newtheorem{theorem}{Theorem}[section]
\numberwithin{equation}{section}
\newtheorem{remark}[theorem]{Remark}
\newtheorem{prop}[theorem]{Proposition}

\newtheorem{lemma}[theorem]{Lemma}

\hypersetup{
    colorlinks=true,
    linkcolor=blue,
    filecolor=magenta,      
    urlcolor=cyan,
}

\def\dl{\delta}
\def\tl{\tilde}

\def\Dl{\Delta}
\def\gt{\gtrsim}
\def\mf{\mathfrak}

\def\sig{\sigma}

\def\nn{\nonumber}

\def\sq{\sqrt}
\def\eps{\epsilon}

\def\fr{\frac}
\def\al{\alpha}

\def\la{\langle}
\def\ra{\rangle}

\def\pr{\partial}
\def\nb{\nabla}

\def\les{\lesssim}
\def\lm{\lambda}

\def\Pe{\mathbb{P}}

\def\l|{\left\|}
\def\r|{\right\|}

 \newtheorem{thm}{Theorem}[section]
 
 \newtheorem{lem}[thm]{Lemma}
 \theoremstyle{definition}
 
 \newtheorem{rem}[thm]{Remark}

\newcommand{\be}{\begin{equation}}
\newcommand{\ee}{\end{equation}}

\usepackage{scalerel,stackengine}
\stackMath

\newcommand\sometext

\newcommand{\JB}[1]{\langle #1 \rangle}

\newcommand{\E}{\textup{E}}

\newcommand{\nr}{\textup{NR}}
\newcommand{\T}{\mathbf{T}}
\newcommand{\R}{\mathbf{R}}
\newcommand{\res}{\textup{R}}

\newcommand{\N}{\textup{N}}
\newcommand{\Z}{\mathbb{Z}}

\newcommand{\norm}[1]{\left\lVert#1\right\rVert}

\title[] {Asymptotic stability of the three-dimensional Couette flow for the Stokes-transport equation}

\author{Daniel Sinambela}
\address[D.S.]{Department of Mathematics, New York University Abu Dhabi, Saadiyat Island, P.O. Box 129188, Abu Dhabi, United Arab Emirates.}
\email{dos2346@nyu.edu}

\author{Weiren Zhao}
\address[W.Z.]{Department of Mathematics, New York University Abu Dhabi, Saadiyat Island, P.O. Box 129188, Abu Dhabi, United Arab Emirates.}
\email{zjzjzwr@126.com, wz19@nyu.edu}

\author{Ruizhao Zi}

\address[R.Z.]{School of Mathematics and Statistics, and Key Laboratory of Nonlinear Analysis \& Applications (Ministry of Education), Central China Normal University, Wuhan,  430079,  P. R. China.}
\email{rzz@ccnu.edu.cn}

\begin{document}

\begin{abstract}
 In this paper, we investigate the asymptotic stability of the three-dimensional Couette flow in a stratified fluid governed by the Stokes-transport equation. We observe that a similar lift-up effect to the three-dimensional Navier-Stokes equation near Couette flow destabilizes the system. We find that the inviscid damping type decay due to the Couette flow together with the damping structure caused by the decreasing background density stabilizes the system. More precisely, we prove that if the initial density is close to a linearly decreasing function in the Gevrey-$\frac{1}{s}$ class with $\frac{1}{2}< s\leq 1$, namely, $\|\varrho_{\rm in}(X,Y,Z)-(-Y)\|_{\mathcal{G}^{s}}\leq \epsilon$, then the perturbed density remains close to $-Y$. Moreover, the associated velocity field converges to Couette flow $(Y, 0, 0)^{\top}$ with a convergence rate of $\frac{1}{\langle t\rangle^3}$.
\end{abstract}
\date{\today}
\maketitle


\allowdisplaybreaks

\section{Introduction}\label{Introduction}
We consider the three-dimensional Stokes-transport equations in  $\Omega:=\mathbb{T}\times \mathbb{R}\times \mathbb{T}\ni (X,Y,Z)$, 
\begin{equation}\label{governing eq}
 \begin{cases}
    \partial_t \varrho +v \cdot \nabla \varrho=0 ,\\
   -\Delta v +\nabla p= -\varrho e_2 ,\\
    \nabla \cdot v=0,\\
    \varrho(t=0)=\varrho_{\rm in}.
\end{cases} 
\end{equation}
Here $v(t,X,Y,Z)=\big(v^1(t,X,Y,Z),v^2(t,X,Y,Z),v^3(t,X,Y,Z)\big)^{\top}$ is the velocity field, $p(t,X,Y,Z)$ represents the pressure and $\varrho(t,X,Y,Z)$ denotes fluid density while $e_2=(0,1,0)$ is the unit vector pointing in the vertical direction.  The system \eqref{governing eq} models the evolution of an incompressible and viscous fluid with inhomogeneous density stratification with  gravity acting in $-e_2$ direction. Unlike the classical Boussinesq equation, this model excludes both velocity self-advection and density diffusion terms. From a physical point of view, the system in \eqref{governing eq} can be viewed as a mesoscopic model that portrays the behavior of inertia-less particles in a Stokes fluid during sedimentation \cite{hofer2018sedimentation, ingham1998transport}. The well-posedness theory of \eqref{governing eq} has been a subject of recent investigations, we refer to \cite{LEBLOND2022120} and references therein. 

On the one hand, the Stokes-transport equation and the incompressible porous medium equation (IPM) both describe the dynamics of fluids driven by density variations. In a two-dimensional fluid at rest ($v_{\rm s}\equiv 0$), it is well-known that a vertically decreasing stratified density profile is asymptotically stable under sufficiently small and smooth perturbations. We refer to \cite{dalibard2023long, leblond2023well} for recent stability results of Stokes-transport equations. Additionally, for recent works on stability/instability of IPM, we refer readers to \cite{bianchini2024relaxation, castro2019global, elgindi2017asymptotic, jo2024quantitative, kiselev2023small} and \cite{Dong2022Asymptotic, Tao2020Stability} for stability of two-dimensional Boussinesq system. 

On the other hand, the Stokes-transport equation belongs to the family of active scalar equations, which means that its velocity $v$ depends on the scalar function $\varrho$ in a nonlocal way. There are many important equations that belong to this family, for instance, the two-dimensional Euler equation in the vorticity form, the surface quasi-geostrophic equation, and IPM. All of these equations share a common class of steady states known as the shear flows whose asymptotic stability has been a major theme of research in recent decades. The stability of shear flows of the two-dimensional Euler equation at linear and nonlinear levels have been done, respectively, by the following authors \cite{IonescuIyerJia2023, Jia2,    WZZ1, WZZ2, WZZ3, Zill} and \cite{BM1, IonescuJia2, IonescuJia,  masmoudizhao2024}. 

For the Stokes-transport equations, the stability mechanisms are intrinsically different for the two-dimensional and three-dimensional regimes, with and without the Couette flow. More precisely, in the two-dimensional case \cite{SZZ2024-2D}, we proved that the Couette flow can stabilize the system under no assumption of a monotonic (decreasing) background density profile. In the present article, however, the mechanisms leading to asymptotic stability require the background density profile to be decreasing. It is now the right time to introduce the steady states concerned here:
\begin{equation}\label{steady states}
v_{\rm s}:=(Y, 0, 0)^{\top},\qquad p_{\rm s}=\frac12 Y^2+C_0, \qquad \varrho_{\rm s}=-Y.
\end{equation}
Here the steady velocity field is the three-dimensional Couette flow and the stratified density profile is monotonically decreasing in the vertical direction. 

As our work concerns the stability of solutions in the perturbation regime of the steady states \eqref{steady states}, let us introduce the perturbations $(U,\Theta, P)$. More precisely, we let $v=U+v_{\rm s},\; p=P+p_{\rm s}, \; \varrho=\Theta+\varrho_{\rm s}$ with $U=(U^{(1)},U^{(2)},U^{(3)})^{\top}$. A direct calculation shows that $(U,\Theta, P)$ satisfies 
\begin{equation}\label{perturbation equation}
\begin{cases}
    \partial_t \Theta  +Y\partial_X \Theta -U^2+U\cdot\nabla\Theta=0,\\
       -\Delta U +\nabla P= -\Theta e_2 ,\\
    \nabla \cdot U=0,\\
    \Theta(t=0)=\Theta_{\rm in}.
\end{cases} 
\end{equation}
We study the long-time behavior of the solutions to \eqref{perturbation equation} when the initial perturbation $\Theta_{\rm in}$ is small in some suitable spaces. 
Our main theorem states as follows:
\begin{theorem}\label{Thm: main}
    For all $\frac{1}{2}<s\leq 1$ and $\lambda_{\mathrm{in}}>\lambda_{\infty}>0$, there is an $\epsilon_0=\epsilon_0(s,\lambda_{\mathrm{in}},\lambda_{\infty})\leq \frac{1}{2}$ such that for all $0<\epsilon<\epsilon_0$ if $\Theta_{\mathrm{in}}$ satisfies
    \begin{align*}
        \int_{\mathbb{T}\times \mathbb{R}\times\mathbb{T}}\Theta_{\mathrm{in}}(X, Y, Z)dXdYdZ=0,
    \end{align*}
    and 
    \begin{align*}
        \int_{\mathbb{R}}|Y\Theta_{\mathrm{in}}(X, Y, Z)|dX dYdZ<\epsilon,\quad
        \sup_{\eta\in\mathbb{R}}\left|e^{\lambda_{\mathrm{in}}|\eta|^s}\widehat{\Theta_{\mathrm{in}}}(0,\eta, 0)\right|&<\epsilon,\\
        \|\Theta_{\mathrm{in}}\|_{\mathcal{G}^{\lambda_{\mathrm{in}}}}=\left(\sum_{l}\sum_{\alpha}\int_{\mathbb{R}}e^{2\lambda_{\mathrm{in}}|l,\eta,\alpha|^s}\left|\widehat{\Theta_{\mathrm{in}}}(l,\eta, \alpha)\right|^2d\eta\right)^{\fr{1}{2}}&<\epsilon,
    \end{align*}
    then there exists a constant $C>1$ independent of $\epsilon$ and $\Theta_{\infty}(X, Y, Z)$ with 
    \begin{align*}
        \int_{\mathbb{T}\times \mathbb{R}\times\mathbb{T}}\Theta_{\mathrm{\infty}}(X, Y, Z)dXdYdZ=0,\quad \text{and} \quad \|\Theta_{\infty}\|_{\mathcal{G}^{\lambda_{\mathrm{\infty}}}}\leq C\epsilon
    \end{align*} 
    such that the solution $(\Theta, U)$ to \eqref{perturbation equation} satisfies 
    \begin{align}
        \left\|\Theta(t, X-tY, Y, Z)-\Theta_{\mathrm{\infty}}(X, Y, Z)\right\|_{\mathcal{G}^{\lambda_{\mathrm{\infty}}}}\leq \frac{C\epsilon}{\JB{t}^2}+\frac{C\epsilon^2}{\JB{t}}.
    \end{align}
    Moreover, it holds that
    \begin{align}
        \left\|\int_{\mathbb{T}}\Theta(t, X, \cdot)dX-\frac{1}{2\pi}\int_{\mathbb{T}^2}\Theta(t, X, \cdot,  Z)dXdZ\right\|_{\mathcal{G}^{\lambda_{\mathrm{\infty}}}}\leq \frac{C\epsilon}{\JB{t}^3},
    \end{align}
    and the velocity field $U(t, X, Y, Z)$ satisfies
    \begin{align}
        &\left\|\Big(U^{(1)},U^{(3)}\Big) (t, \cdot )\right\|_{L^2}+\left\|\fr{1}{2\pi}\int_{\mathbb{T}}U^{(2)}(t,X,\cdot)dX\right\|_{L^2}\leq \frac{C\epsilon}{\JB{t}^3},\\
        &\left\|U^{(2)}(t, \cdot)-\fr{1}{2\pi}\int_{\mathbb{T}}U^{(2)}(t,X,\cdot)dX\right\|_{L^2}\leq \frac{C\epsilon}{\JB{t}^4}.
    \end{align}
\end{theorem}
Below, we list a few remarks related to our main theorem.
\begin{rem}
    In the statement of Theorem \ref{Thm: main}, the density $\Theta(t, X-tY, Y, Z)$ remains small in $\mathcal{G}^{\lambda_{\infty}}$. In the proof, we allow the highest norm of density to grow as $\JB{t}^{\frac{3}{2}}$, see \eqref{En-1} for more details. This growth is from the nonlinear interactions between zero and non-zero modes. 
\end{rem}
\begin{rem}
    The zero mode of the velocity $\fr{1}{2\pi}\int_{\mathbb{T}}\Big(U^{(2)}, U^{(3)}\Big)(t,X, Y, Z)dX$ and the non-zero modes of the velocity 
    \begin{align*}
        \left(U^{(1)},\  
        U^{(2)}(t)-\fr{1}{2\pi}\int_{\mathbb{T}}U^{(2)}(t,X)dX, \ 
        U^{(3)}(t)-\fr{1}{2\pi}\int_{\mathbb{T}}U^{(3)}(t,X)dX\right)^{\top}
    \end{align*}
    have different decay mechanisms. All the decay rates are optimal. 
\end{rem}

\begin{rem}
The lift-up phenomenon is an important destabilizing effect for Couette flow in three-dimensional fluid. In the present work, a linear growth that is similar to such lift-up is present in the system, see \eqref{eq-theta'}, \eqref{exp-u} and discussions in section \ref{sec:mainidea}. 

To stabilize the lift-up effect and attain asymptotic stability, another mechanism known as enhanced dissipation can be exploited. Various works on the three-dimensional fluid models that involve diffusive structures take full advantage of this mechanism, see \cite{BGM, Bedrossian_Germain_Masmoudi_I, Bedrossian_Germain_Masmoudi_II,  Chen_wei_zhang_2020_3d, Wei_Zhang_2020}. We also refer to \cite{BedrossianHeIyerWang2023, Bedrossian_2016, chen_Li_Wei_Zhang2018transition, LiMasmoudiZhao2022asymptotic, MZ2, MZ1} for the enhanced dissipation of Couette flow in two-dimensional fluid models. More precisely, in the aforementioned papers, the transport-diffusion structure in the equation provides an avenue for a trade between smallness and decay.

In the present article, due to the lack of the diffusion term in the system, such a dissipation mechanism is out of our reach. We, instead, exploit the damping structure in the zero-mode of velocity equation to give the time decay we need by sacrificing regularity. In the process, we allow some growth in the top norm to gain decay in the low norm. Additionally, in order to take full advantage of the damping, we have introduced a decreasing background density profile at the outset of the paper. We refer to section \ref{sec:mainidea} for more discussions. 
\end{rem}

\begin{rem}
    It is an interesting problem to consider the finite channel setting as considered in the two-dimensional case \cite{SZZ2024-2D} and study the effect arising from the presence of boundaries. 
\end{rem}

\subsection{Notations}
 In this subsection, we introduce important notations used throughout the present work. First, we define the common $l^1$ norm for frequency $(k,\eta,\alpha)$ which takes the form $|k,\eta,\alpha|=|k|+|\eta|+|\alpha|$. In addition to that, let us explicitly define the underlying function spaces our work is based upon. Namely, the Gevrey$-\frac{1}{s}$ space with Sobolev correction which together with its norm is defined by
\[
\mathcal{G}^{\lambda,\sigma}:=\{f\in C^{\infty}: \norm{f}_{\mathcal{G}^{\lambda,\sigma}}<\infty\}, \text{ where }\norm{f}^2_{\mathcal{G}^{\lambda,\sigma}}=\sum_{k,\alpha \in \mathbb{Z}}\int_{\eta}|\widehat{f}_k(t,\eta,\alpha)|^2\JB{k,\eta,\alpha}^{2\sigma}e^{2\lambda(t)|k,\eta,\alpha|^s}\;d\eta.
\]
For any given scalar $x$, we define the standard Japanese bracket $\JB{x}=\sqrt{1+|x|^2}$. Furthermore, let $f$ and $g$ be any given functions, we define the commonly used notation $f\lesssim g$ to mean that there exists a pure constant (independent of any parameter) $C>1$ such that $f\leq C g$. Throughout the work, we also write $f\approx g$ to mean that there exists $C>1$ (again, independent of any parameter) such that $C^{-1}f\leq g \leq C f$. We also use $c'$, $c$ and ${\bf c}$ to denote constants satisfying $0<c'<c<{\bf c}<1$.

We write $f_0:=\frac{1}{2\pi}\int_{\mathbb{T}}f(x,y,z)\;dx$; this takes the average of the function $f$ in the horizontal direction. The $f_0$ goes by the name ``the zero mode of $f$". We introduce the non-zero mode of $f$, which is simply denoted by  $f_{\neq}:=f-f_0$. One can think of it as the projection off of the zero mode of $f$. Similarly, we use $\mathbb{P}^{z}_0f$ to denote the average of $f$ in the $z$ direction, namely, $\mathbb{P}^{z}_0f:=\frac{1}{2\pi}\int_{\mathbb{T}}f(x,y,z)\;dz$. Moreover, we also use the notation $\mathbb{P}^{z}_{\ne}f:=f-\mathbb{P}^{z}_0f$.

For $\eta \geq 0$, we let $\E(\eta):=\lfloor \eta \rfloor \in \mathbb{Z}$. Fix $\eta \in \mathbb{R}$ and let $1\leq |k| \leq \E(|\eta|^{\frac{1}{2}})$ with $\eta k\geq 0$. Let $t_{k,\eta}:=\frac{|\eta|}{|k|}-\frac{|\eta|}{2|k|(|k|+1)}=\frac{|\eta|}{|k|+1}+\frac{|\eta|}{2|k|(|k|+1)}$ and $t_{0,\eta}=2|\eta|$. With this in mind, we define the following {\it critical intervals}
\[
\mathbb{I}_{k,\eta}
=\begin{cases} 
[t_{|k|,\eta},t_{|k|-1,\eta}],\qquad \text{ for } \eta k \geq 0, \text{ and } 1\leq k \leq \E(|\eta|^{\frac{1}{2}}), \\
\emptyset, \qquad \qquad \qquad \quad\;  \text{ otherwise.}
\end{cases}
\]
The {\it resonant intervals} are defined by
\be\label{res-int}
{\rm{I}}_{k, \eta}:=\begin{cases}
{\mathbb{I}}_{k,\eta},\ \ \mathrm{if}\ \ 2\sqrt{|\eta|}\le t_{|k|,\eta},\\
\emptyset, \ \ \ \ \ \mathrm{otherwise}.
\end{cases}
\ee
\subsection{Main ideas}\label{sec:mainidea}
Let us now outline the main idea in the proof of Theorem \ref{Thm: main}. We first introduce the linear change of coordinate 
\begin{equation}\label{linear change}
(x,y,z)=(X-tY, Y,Z),
\end{equation}
and $\theta(t, x, y, z)=\Theta(t, X, Y, Z)$, $V(t, x, y, z)=U(t, X-tY, Y, Z)$. Via \eqref{perturbation equation}, we obtain the following expression
\begin{align*}
    V(t, x, y, z)=\left(
        -\partial_{xy}^L\Delta_L^{-2}\theta,
        \Delta_{xz}\Delta_L^{-2}\theta,
        -\partial_{yz}^L\Delta_L^{-2}\theta
\right)^{\top},
\end{align*}
where $\Delta_{xz}=\partial_x^2+\partial_z^2,\ \partial_y^L=\partial_y-t\partial_x,\ \partial_{yy}^L=\left(\partial_y-t\partial_x\right)^2, \ \Delta_L=\partial_{xx}+\partial_{yy}^L+\partial_{zz}$. After taking the Fourier transform, we get that 
\begin{align*}
    \widehat{V}(t, k, \eta, \alpha)=\left(
    \frac{k(\eta-kt)}{(k^2+(\eta-kt)^2+\alpha^2)^2}\widehat{\theta}, 
    \frac{-(k^2+\alpha^2)}{(k^2+(\eta-kt)^2+\alpha^2)^2}\widehat{\theta}, 
    \frac{(\eta-kt)\alpha}{k^2+(\eta-kt)^2+\alpha^2)^2}\widehat{\theta}\right)^{\top}. 
\end{align*}
which implies that for $k\neq 0$, 
\begin{align}\label{eq:V-bd}
    |\widehat{V^{(1)}}|\lesssim \frac{\JB{k,\eta,\alpha}^{3}}{\JB{t}^3}|\widehat{\theta}|, \quad
    |\widehat{V^{(2)}}|\lesssim \frac{\JB{k,\eta,\alpha}^{6}}{\JB{t}^4}|\widehat{\theta}|,\quad
    |\widehat{V^{(3)}}|\lesssim \frac{\JB{k,\eta,\alpha}^{4}}{\JB{t}^3}|\widehat{\theta}|. 
\end{align}
The decay of the non-zero modes is analog to the inviscid damping of Couette flow for the two-dimensional Euler equation, which is due to the vorticity mixing. For the Stokes-transport equation, it is due to the density mixing. A similar phenomenon can also be observed in shear flows for stratified fluid \cite{bedrossian2023nonlinear, zelati2023explicit, MSZ, YangLin2018}, ideal inhomogeneous fluid \cite{ChenWeiZhangZhang2023, Zhao2023}, and inviscid plasma \cite{ZhaoZi2023}, demonstrating its occurrence in various systems. 

Now let us introduce the working system. A direct calculation gives us that $\theta$ satisfies a transport equation with a weak damping term $-\Delta_{xz}\Delta_{L}^{-2}\theta$. More precisely, we have 
\begin{align}\label{eq-theta'}
\partial_t\theta-\Delta_{xz}\Delta_{L}^{-2}\theta+u\cdot\nabla_{xyz}\theta=0,
\end{align}
where the new velocity $u$ can be decomposed into the non-zero modes and the zero mode: 
\begin{equation}\label{exp-u}
\begin{aligned}
    u(t,x,y,z)
    &=(-t\Delta_{xz}-\partial_{xy}^L,\Delta_{xz}, -\partial_{yz}^L
)\Delta_L^{-2}\theta_{\neq}+(-t\pr_{zz},\pr_{zz}, -\partial_{yz}
)\Delta_{yz}^{-2}\theta_0:=u_{\neq}+u_{0}. 
\end{aligned}
\end{equation}
It is easy to see that on the Fourier side
\begin{align}\label{u_ne-decay}
&\nonumber|\widehat{u}_{\ne}(t,k,\eta,\al)|=\left|\mathcal{F}\left[(-t\Delta_{xz}-\partial_{xy}^L,\Delta_{xz}, -\partial_{yz}^L
)\Delta_L^{-2}\theta_{\ne}\right](t, k,\eta,\al)\right|\\
&\lesssim \frac{\langle t\rangle (k^2+\al^2)+|k(\eta-kt)|+|\al(\eta-kt)|}{(k^2+(\eta-kt)^2+\al^2)^2}|\widehat{\theta}_{\ne}(t,k,\eta,\al)|\lesssim\frac{\langle k,\eta,\al\rangle^6}{\langle t\rangle^3}|\widehat{\theta}_{\ne}(t,k,\eta,\al)|.
\end{align}

Next, let us focus on the zero mode. A classical way to treat zero mode is to introduce the nonlinear change of coordinate, which works well for the two-dimensional case, as the zero mode of the velocity is a time-dependent shear flow. It is not easy to apply this method directly in our case. Indeed, we prove the decay of average velocity $u_0$, which is not obvious. We now introduce the equation of the zero mode of the density
\begin{align}\label{eq-theta0'}
\partial_t\theta_0-\partial_{zz}\Delta_{yz}^{-2}\theta_0+\tl{u}_0\cdot\nb_{yz}\theta_0+\left(u_{\ne}\cdot\nabla_{xyz}\theta_{\ne}\right)_0=0,
\end{align}
where 
\begin{equation}\label{eq: tlu_0}
    \tl{u}_0(t,y,z)=(\pr_{zz}, -\partial_{yz}
)\Delta_{yz}^{-2}\theta_0
\end{equation}
Compared to the three-dimensional Navier-Stokes equation near Couette flow, the linear growth in the first component of the zero mode of $u_0$ has a destabilizing effect similar to the lift-up phenomenon. Specifically, in the context of the three-dimensional Navier-Stokes equation, the streak solution $u_0$ satisfies the following system:
\begin{align*}
    \left\{\begin{aligned}
        &\partial_t u_0^{(1)}-\nu \Delta u_0^{(1)}=- u_0^{(2)}+{\text{low order terms}},\\
&\partial_t u_0^{(2)}-\nu \Delta u_0^{(2)} ={\text{low order terms}}, \\
&\partial_t u_0^{(3)}-\nu \Delta u_0^{(3)} ={\text{low order terms}}.
    \end{aligned}\right.
\end{align*} 
The term $- u_0^{(2)}$ in the first equation induces a linear transient growth, expressed as $u_0^{(1)} \approx - t u_0^{(2)}$ for $t \ll \nu^{-1}$, mirroring the relationship between $u_0^{(1)}$ and $u_0^{(2)}$ in \eqref{exp-u}. The equation for $\theta_0$ serves a similar role as the streak equation for the first component of the zero mode of the three-dimensional vorticity field $\omega_0^{(1)} = \partial_y u_0^{(3)} - \partial_z u_0^{(2)}$, with the damping term $\partial_z^2 \Delta^{-2} \theta_0$ replacing the dissipation $\nu \Delta \omega_0^{(1)}$ in the three-dimensional Navier-Stokes equation.

Now, let us discuss more on their differences. In the three-dimensional Navier-Stokes equation, both $\omega_0^{(1)}$ and $(u_0^{(2)}, u_0^{(3)})$ follow well-behaved transport-diffusion equations. However, in our model, while $\theta_0$ follows a transport equation with damping, the equation for $u_0^{(2)}$ lacks such favorable characteristics due to the presence of the anisotropic derivative $\partial_{z}^2\Delta^{-2}$, which disrupts the transport structure. Moreover, a significant disparity lies in the approach to counteract the linear growth of $u_0^{(1)}$. In contrast to the traditional reliance on smallness in the Navier-Stokes equation, our model necessitates the use of derivatives, introducing additional challenges in closing the energy estimate.
More precisely, by dropping the nonlinear term and taking the Fourier transform, we arrive at a toy model
\begin{align}\label{eq: toymodel-0mode}
    \partial_t\widehat{\theta}_0+\frac{\alpha^2}{(\alpha^2+\eta^2)^2}\widehat{\theta}_0 =F, 
\end{align}
where the force term $F$ decays as $\frac{1}{\JB{t}^3}$. Thus we get for $\alpha\neq 0$
\begin{align}\label{eq:theta_0-estimate}
    |\widehat{\theta}_0(t)|\leq |\widehat{\theta}_0(0)|e^{-\frac{\alpha^2}{(\alpha^2+\eta^2)^2} t}+\int_0^te^{-\frac{\alpha^2}{(\alpha^2+\eta^2)^2} (t-\tau) }{\JB{\tau}^{-3}} d\tau \lesssim \fr{1}{\JB{t}^3}\Big(\frac{(\alpha^2+\eta^2)^2}{\alpha^2}\Big)^3|\widehat{\theta_0}(0)|. 
\end{align}
Note that the decay rate $\JB{t}^{-3}$ is also optimal. But to obtain the same decay rate as the non-zero modes in \eqref{u_ne-decay}, we have to pay more regularity.

We now consider the nonlinear interactions and treat $u\cdot \nabla_{xyz}\theta$ appearing in \eqref{eq-theta'}. Due to the transport structure, we can expect a good estimate of the low-high interactions, namely, $u_{\mathrm{low}}\cdot \nabla_{xyz}\theta_{\mathrm{high}}$, since both $u_0$ and $u_{\neq}$ decay fast. Now we focus on the high-low interactions. For the interaction with $u_{\neq}$, we have two observations from \eqref{u_ne-decay}. First, if $k\eta>0$ and $\eta$ is large relative to $k$, then the velocity is amplified by a factor $\JB{\frac{\eta^2}{k^2+\alpha^2}}^2$ at a critical time $\frac{\eta}{k}$. Such transient growth is similar to the Orr mechanism \cite{Orr}. Second, by using the fact that $\JB{t}\lesssim \JB{t-\fr{\eta}{k}}+\JB{\fr{\eta}{k}}$, the linear growth $\JB{t}$ can be also regarded as the derivative loss. In section \ref{sec:growth model}, we introduce a toy model that describes the growth mechanism and design a time-dependent Fourier multiplier that captures the loss of regularity in a very precise way adapted to the critical times and the associated nonlinear effect.  We note that by carefully studying the interactions of the non-zero modes, one may determine that the total growth will result in a Gevrey-3 derivative loss. 
The interactions with $u_0$, however, are more complicated, since they stem from the estimate of $\theta_0$ which enjoys a different decay mechanism. Due to a derivative loss of the source term $\left(u_{\ne}\cdot\nabla_{xyz}\theta_{\ne}\right)_0$ in \eqref{eq-theta0'}, we may expect $\theta_0$ to admit one derivative less than $\theta$. Then by taking the Fourier transform, dropping several terms: the non-zero modes interactions $u_{\neq}\cdot\nabla_{xyz}\theta$, linear damping term $-\Delta_{xz}\Delta_L^{-2}\theta$, the low-high zero mode interactions, and the weak high-low zero mode interactions, we then arrive at a toy model
\begin{align*}
    \partial_t\widehat{\theta}_l(t,\eta,\alpha)= \epsilon\fr{t\alpha^2}{(\alpha^2+\eta^2)^2}\widehat{\theta_0}(\eta,\alpha). 
\end{align*}
By taking the one-derivative loss into account and using \eqref{eq:theta_0-estimate}, we get that 
\begin{align*}
    \partial_t\widehat{\theta}_l(t,\eta,\alpha)
    &\lesssim \epsilon^2\fr{t\alpha^2}{(\alpha^2+\eta^2)^2}|\alpha, \eta|e^{-\frac{\alpha^2}{(\alpha^2+\eta^2)^2} t}\\
    &\lesssim \epsilon^2 \sqrt{t}\sqrt{\fr{t\alpha^2}{(\alpha^2+\eta^2)^2}}\fr{|\alpha||\alpha,\eta|}{(\alpha^2+\eta^2)} e^{-\frac{\alpha^2}{(\alpha^2+\eta^2)^2} t}
    \lesssim \epsilon^2\sqrt{t},
\end{align*}
which gives the $\JB{t}^{\fr32}$ growth of $\theta$. Thus we should allow the top energy to grow with rate $\JB{t}^{\fr32}$, see \eqref{En-1}. In order to obtain enough decay for $u_{\neq}$ and $u_0$, we design  three levels of energy functionals for $\theta$, see \eqref{En-1}, \eqref{En-3}, and $\eqref{En-5}$.  Correspondingly, we define four levels of energy functionals for the zero mode $\theta_0$, see \eqref{En0-1}, \eqref{En0-2}, \eqref{En0-4}, and \eqref{En-6}. Each level has an improved growth rate compared with the previous level. We refer to Figure (\ref{map}) for the clue of bounds improvement. 
\begin{figure*}[tbp]
\caption{The map of bounds improvement}
\label{map}

\medskip\medskip

{\small
\begin{tikzpicture}[node distance=4cm]
\node (start) [startstop]{$\|\theta\|_{\sig_1}\lesssim \epsilon \JB{t}^{\fr32}$};
\node (exp1) [explain, below of=start, yshift=2.5cm] {Main energy};
\node (io1) [io, below of=start, yshift=1cm] {$\|B\theta_0\|_{\sigma_1-2}\lesssim \epsilon$};
\node (exp2) [explain, right of=exp1, xshift=-1.8cm] {Improve};
\node (io2) [io, right of=start, xshift=0cm] {$\|\theta\|_{\mathcal{G}^{\lambda,\sigma_3}}\lesssim \epsilon\JB{t}^{\fr12}$};
\node (exp3) [explain, below of=io2, yshift=2.5cm] {Improve};
\node (io3) [io, below of=exp3, yshift=2.5cm] {$\|\theta_0\|_{\mathcal{G}^{\lambda,\sigma_2}}\lesssim \epsilon \JB{t}^{-\fr32}$};
\node (exp4) [explain, right of=exp3, xshift=-2cm] {Improve};
\node (io4) [io, right of=io2, xshift=0cm] {$\|\theta\|_{\mathcal{G}^{\lambda,\sigma_5}}\lesssim \epsilon $};
\node (exp5) [explain, below of=io4, yshift=2.5cm] {Improve};
\node (exp6) [explain, right of=exp5, xshift=-2cm] {Improve};
\node (io5) [io, below of=exp5, yshift=2.5cm] {$\|\theta_0\|_{\mathcal{G}^{\lambda,\sigma_4}}\lesssim \epsilon\JB{t}^{-\fr52}$};
\node (io6) [io, right of=io5, xshift=0cm] {$\|\theta_0\|_{\mathcal{G}^{\lambda,\sigma_6}}\lesssim \epsilon\JB{t}^{-3}$};
\draw (start) -- (exp1);
\draw (exp1) -- (io1);
\draw[arrow] (start) -- (exp2);
\draw[arrow] (exp2) -- (io3);
\draw[arrow] (io3) -- (exp3);
\draw[arrow] (exp3) -- (io2);
\draw[arrow] (io2) -- (exp4);
\draw[arrow] (exp4) -- (io5);
\draw[arrow] (io5) -- (exp5);
\draw[arrow] (exp5)--(io4);
\draw[arrow] (io4) -- (exp6);
\draw[arrow] (exp6)--(io6);
\end{tikzpicture}

}
\end{figure*}
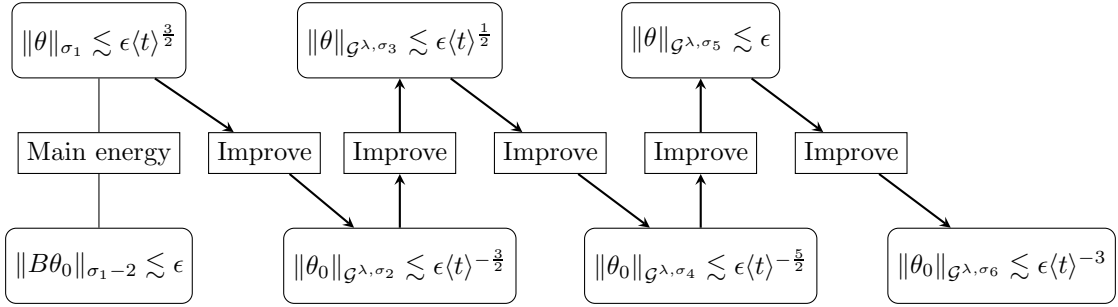

We also want to highlight the new idea that is used to combine the energy method and semi-group estimate. Recall by \eqref{eq: toymodel-0mode}, we have the semi-group $\mathbb{S}(t;\eta,\alpha)=\exp\{\frac{-\alpha^2}{(\alpha^2+\eta^2)^2}t\}$. Although $\mathbb{S}$  decays exponentially, it is not uniform in all frequencies, which suggests we do the estimate frequency separately. That means we need to apply the time norm first then the space norm. This is different from the classical energy method approach where the norm with respect to space is computed first followed by the one with respect to time. To overcome this difficulty, we pay enough regularity to obtain a uniform semi-group estimate. To deal with ${\rm S}_{00z}^{\rm HL}$, see \eqref{Isig2}, we are not allowed to pay derivative. We then introduce the lowest norm \eqref{En-7}, where we can switch and space and time norm freely. 

We conclude this section with observations to enhance the required regularity. The Gevrey-2 regularity requirement is from the interactions between zero mode and non-zero modes, see Lemma \ref{lem: zero-mode-Reaction}, \eqref{G-2 req-1}, and \eqref{up-growth2'}. It might not be optimal. Firstly, one can modify the time-dependent Fourier multipliers as outlined in \cite{SZZ2024-2D}, thereby controlling the Gevrey-3 growth for the interactions of non-zero modes. Secondly, the growth of $t^{\frac{3}{2}}$ is not uniform across all frequencies. By leveraging the advantages gained from derivatives in the zero-mode interactions, one can design a new time-dependent Fourier multiplier to regulate the growth, particularly for $t\lesssim |\alpha, \eta|$. Finally, as the growth from zero-mode interactions and non-zero-mode interactions occurs at different time intervals, one can devise a new time-dependent weight to capture the growth for $t\gg |\alpha, \eta|$, thereby replacing the uniformly time-growing weight $\JB{t}^{\frac{3}{2}}$ with a weaker weight. We conjecture that the required regularity cannot be weaker than Gevrey-3, as observed in the two-dimensional case \cite{SZZ2024-2D}.

\section{Growth Mechanism and  Construction of Weight}\label{sec:growth model}

In this section, we elaborate on the construction of weight $w$ together with the associated multipliers. Recall that the main crack of the issue here is to get control of the regularity of the solution. It is important to note also that one has to pay regularity to get decay in time, say for the velocity field. Similar to the two-dimensional Stokes-transport problem, it reduces to analyze the nonlinear interactions between the velocity field and the density gradient, namely $u_{\ne}\cdot \nabla \theta $. It turns out that such an analysis can be done by looking at a toy model. 

\subsection{Construction of time-dependent weight}
\subsubsection{Toy Model for non-zero mode interactions}
From \eqref{exp-u} and \eqref{eq-theta0'}, on the Fourier side, we derive a toy model to mimic the growth caused by $\fr{\left(t|\beta|+|\xi|+|l,\beta|\right) |l,\beta|}{(l^2+(\xi-l t)^2+\beta^2)^2}$
\begin{align}\label{toy1}
\pr_t\widehat{\theta}_k(t,\eta,\al)=\sum_{l\in\Z\setminus\{0\},\, \beta\in\Z}\int_\xi \fr{\left(t|\beta|+|\xi|+|l,\beta|\right) |l,\beta|}{(l^2+(\xi-l t)^2+\beta^2)^2}\widehat{\theta}_l(t,\xi,\beta)\left(\widehat{\theta}_{k-l}\right)_{\rm low}(t,\eta-\xi,\al-\beta)\;d\xi.
\end{align}
Since the appearance of the low frequency quantity $\left(\widehat{\theta}_{k-l}\right)_{\rm low}(t,\eta-\xi,\al-\beta)$, we fix $\eta=\xi$,  $\beta=\al$ and $l=k\pm1$ in \eqref{toy1}. In addition, $\left(\widehat{\theta}_{k-l}\right)_{\rm low}(t,\eta-\xi,\al-\beta)=O(\kappa)$. Then \eqref{toy1} reduces to 
\begin{align}\label{toy2}
\pr_t\widehat{\theta}_k(t,\eta,\al)=O(\kappa)\sum_{l= k\pm1}{  \fr{\left(t|\al|+|\eta|+|l,\al|\right) |l,\al|}{(l^2+(\eta-l t)^2+\al^2)^2}}\widehat{\theta}_l(t,\eta,\al).
\end{align}
As $\widehat{\theta}$ evolves in time, the equation above undergoes a critical time interval ${\rm I}_{l,\eta}$ at a fixed $k$ around $t=\frac{\eta}{l}$. This interval is of width $\frac{\eta}{l^2}$. When $|\eta|l^{-2}\ll 1$, any growth there in that critical interval is negligible. However, when $|\eta|l^{-2}>1$, one has to do a careful estimate to know and control the growth. Let us commit to the scenario of one critical time in which the resonance may happen, that is to say, $t\in {\rm I}_{l,\eta}$. We only pick one term on the right-hand side of \eqref{toy2}. Before proceeding any further, we note that
\begin{align}\label{up-growth}
{  \fr{\left(t|\al|+|\eta|+|l,\al|\right) |l,\al|}{(l^2+(\eta-l t)^2+\al^2)^2}}\les\fr{t+\fr{|\eta|}{|l|}}{l^2+(\eta-l t)^2+\al^2}.
\end{align}
{\bf (I) From non-resonance to resonance (NR$\rightarrow$ R)}. Since $k$ mode is the resonant mode, we choose $l\neq k$, for simplicity just pick the $l=k+1$. Hence, right hand side of \eqref{toy2} reads
\begin{align}\label{toy3}
\pr_t\widehat{\theta}_k(t,\eta,\al)=O(\kappa){\fr{t+\fr{|\eta|}{|k+1|}}{(k+1)^2+(\eta-(k+1)t)^2+\al^2}}\widehat{\theta}_{k+1}(t,\eta,\al).
\end{align}
For $t\in{\rm I}_{k,\eta}$, we have
\[
t\approx \fr{|\eta|}{|k|},\quad \left|t-\fr{\eta}{k+1}\right|\ge\fr{|\eta|}{2k(k+1)}.
\]
Then a simple calculation leads us to
\begin{align}\label{grow1}
{\fr{t+\fr{|\eta|}{|k+1|}}{(k+1)^2+(\eta-(k+1)t)^2+\al^2}}\les\fr{|k|}{|\eta|}\lesssim\fr{|k|^2}{|\eta|}.
\end{align}
{\bf (II) From resonance to non-resonance (R$\rightarrow$ NR)}. In this case, let us replace $k$ by $k-1$ on the left hand side of \eqref{toy2}, and pick $l=k$ on the right and side of \eqref{toy2}:
\begin{align}\label{toy4}
\pr_t\widehat{\theta}_{k-1}(t,\eta,\al)=O(\kappa)\fr{t+\fr{\eta}{|k|}}{k^2+(\eta-kt)^2+\al^2}\widehat{\theta}_k(t,\eta,\al).
\end{align}
A straightforward calculation yields 
\be\label{grow2}
\fr{t+\fr{\eta}{|k|}}{k^2+(\eta-kt)^2+\al^2}\les \fr{|\eta|}{|k|^2(1+(t-\fr{\eta}{k})^2)}.
\ee

Collecting \eqref{toy3}--\eqref{grow2} and thinking of $k$ as the resonant mode and $l\neq k$ as non-resonant mode, now we arrive at the following toy model in each critical time interval $t\in {\rm I}_{k,\eta}$
\begin{align}\label{toy-theta}
\begin{cases}
\pr_t \theta_{\rm R}\approx \kappa\fr{|k|^2}{|\eta|}\theta_{\rm NR},\\
\pr_t \theta_{\rm NR}\approx \kappa\fr{|\eta|}{|k|^2\left(1+|t-\fr{\eta}{k}|^2\right)}\theta_{\rm R}.
\end{cases}
\end{align}
Note that we enlarge the coefficients to deduce \eqref{toy-theta} which can be modified to a more accurate one. 

\subsubsection{Detailed construction of w}

The above discussion shows that the growth happens when the frequency in $y$ direction dominates the other two directions, namely $|\eta|>\max\{|k|,|\alpha|\}$. However, to make the Fourier multiplier well-adapted to transport structure, we should balance the growth in all three directions. Therefore, we first homogenize the growth direction by introducing  
\begin{equation}\label{eq-def-iota}
  \iota(k,\eta,\alpha)=\left\{
    \begin{array}{ll}
      k,&|k|>|\eta|,\ |k|>|\alpha|\\
      \eta,&|\eta|\geq |k|,\ |\eta|\geq |\alpha|\\
      \alpha,&|\alpha|> |\eta|,\ |\alpha|\geq |k|.
    \end{array}
  \right.
\end{equation}
If $\iota=\eta$, then the Fourier multiplier captures the real growth that happens in the system. If $\iota\neq \eta$, namely frequency in $x$ or frequency in $z$ is larger, then the real growth is less than the growth of the Fourier multiplier. The idea is from \cite{LiMasmoudiZhao2022asymptotic, MZ1}. Throughout the paper, whenever it is clear from the context, we shall use $\iota_1$ to mean $\iota(k,\eta,\alpha)$ and  $\iota_2$ to mean $\iota(l,\xi,\beta)$.

Now we introduce two weights $w_{\rm NR}$  and $w_{\rm R}$ to mimic the growth that the reaction term may create during the time interval $\mathbb{I}_{\ell,\iota}$ for all $\ell=1,2,....,\E(\iota^{\frac{1}{2}})$. 
  Essentially, $w_{\nr}$ and $w_{\res}$ are constructed via an approximate solution to the toy model \eqref{toy-theta} satisfying $w_{\rm NR}(t,-\iota)=w_{\rm NR}(t,\iota)$ and $w_{\rm R}(t,-\iota)=w_{\rm R}(t,\iota)$. It suffices to focus on the definitions of $w_{\rm NR}$ and $w_{\rm R}$ for $\iota\in[0,\infty)$. In light of the toy model \eqref{toy-theta}, we are able to use the weights $w_{\rm NR}$ and $w_{\rm R}$ constructed in \cite{BM1}. For the sake of completeness, let us sketch the process below.  To begin with, for  $0\le\iota\le1$, set
\[
w_{\rm NR}(t,\iota)=w_{\rm R}(t,\iota)=1.
\]
For $\iota>1$, $w_{\rm NR}(t,\iota)$  will be constructed backward in time. First, define
\[
w_{\rm NR}(t,\iota)=w_{\rm R}(t,\iota)=1, \quad\mathrm{if}\quad t>t_{0,\iota}=2\iota.
\]
Next,  for $t\in \mathbb{I}_{\ell,\iota}$, the function $w_{\nr}$ satisfies the following equations
\[
\begin{cases}
w_{\nr}(t,\iota)=\big(\frac{\ell^2}{|\iota|}\big[1+b_{\ell,\iota}|t-\frac{\iota}{\ell}|\big]\big)^{C_*} w_{\nr}(t_{\ell-1,\iota},\iota),\qquad \text{for all } t\in \big[\frac{\iota}{\ell},t_{\ell-1, \iota}\big]=:\mathbb{I}_{\ell,\iota}^{\rm R},\\ \\
w_{\nr}(t,\iota)=\big(\big[1+a_{\ell,\iota}|t-\frac{\iota}{\ell}|\big]\big)^{-1-C_*} w_{\nr}(\frac{\iota}{\ell},\iota), \ \ \, \qquad \text{for all } t\in \big[t_{\ell,\iota},\frac{\iota}{\ell}\big]=:\mathbb{I}_{\ell,\iota}^{\rm L},
\end{cases}
\]
where 
\be\label{def-ab}
b_{\ell,\iota}=
\begin{cases}
\fr{2(\ell-1)}{\ell}\left(1-\fr{\ell^2}{\iota}\right), \quad\mathrm{for}\quad \ell\ge2,\\
1-\fr{1}{\iota},\quad\qquad\ \ \ \ \ \ \ \mathrm{for}\quad \ell=1,
\end{cases}
\quad a_{\ell,\iota}=\fr{2(\ell+1)}{\ell}\left(1-\fr{\ell^2}{\iota}\right),
\ee
and $C_*$ is a positive constant depending on $\kappa$ in \eqref{toy-theta}, see more details in \cite{BM1}.
Finally, set
\[
w_{\rm NR}(t,\iota)=
w_{\rm NR}(t_{\E(\sqrt{\iota}),\iota},\iota),\quad \mathrm{if}\quad t\in\left[0,t_{\E(\sqrt{\iota}),\iota}\right].
\]
Now, on each time interval $\mathbb{I}_{\ell,\iota}$, we define $w_{\res}$ which relies on the behavior of $w_{\nr}$ on  $\mathbb{I}_{\ell,\iota}$. More precisely, we write it as follows
\begin{equation}\label{wR}
\begin{cases}
w_{\res}(t,\iota)=\big(\frac{\ell^2}{|\iota|}\big[1+b_{\ell,\iota}|t-\frac{\iota}{\ell}|\big]\big) w_{\nr}(t,\iota),\qquad \text{for all } t\in \big[\frac{\iota}{\ell},t_{\ell-1,\iota}\big]=:\mathbb{I}_{\ell,\iota}^{\rm R},\\ \\
w_{\res}(t,\iota)=\dfrac{\ell^2}{|\iota|}\big(\big[1+a_{\ell,\iota}|t-\frac{\iota}{\ell}|\big]\big)w_{\nr}(t,\iota),\qquad   \text{for all } t\in \big[t_{\ell,\iota},\frac{\iota}{\ell}\big]=:\mathbb{I}_{\ell,\iota}^{\rm L}.
\end{cases}
\end{equation}

\begin{rem}\label{rem-1}
    One can see from \eqref{def-ab} that $b_{\ell,\iota}$ and $a_{\ell,\iota}$ approach 0 when $\ell$ approaches $\E(\sqrt{\iota})$ which is the reason why the resonant intervals $\textup{I}_{\ell,\iota}$ defined in \eqref{res-int} is introduced. It is worth pointing out that, on one hand, if $t\in\textup{I}_{\ell,\iota}$, there holds $|\ell|\le\fr{1}{2} \sqrt{|\iota|}$, and hence $b_{\ell,\iota}\approx1\approx a_{\ell,\iota}$. On the other hand, if $t\notin\textup{I}_{\ell,\iota}$, then either $t\in{\rm I}_{\ell,\iota}$ with $|\ell|>\fr{\sqrt{|\iota|}}{4}$ or $t\notin{\mathbb I}_{\ell,\iota}$, and hence for both cases, $(\ell,\iota)$ will never be resonant.
\end{rem}

Component-wise $w_k$ is given as follows
\begin{equation}\label{def-wk}
w_k(t, \iota)=
\begin{cases}
   w_k(t_{\E(|\iota|^{\frac{1}{2}}),\iota},\iota), \ \qquad &t < t_{\E(|\iota|^{\frac{1}{2}}),\iota}, \\
   w_\nr(t, \iota), \qquad \qquad \; &t\in \Big[ t_{\E(|\iota|^{\frac{1}{2}}),\iota},2\iota\Big]\setminus \text{I}_{k,\iota},\\
   w_\res(t, \iota)  \qquad \qquad \quad \;&t\in \text{I}_{k,\iota},\\
   1 \ \ \quad \quad \quad \qquad \qquad \; &t\geq 2|\iota|.
\end{cases}
\end{equation}

We now introduce 
\begin{equation}\label{multiplier J}
{\rm J}_k(t,\eta,\al)=\frac{1}{w_k(t,\iota(k,\eta,\alpha))},
\end{equation}
and $w_k$ is as defined in \eqref{def-wk}. In this way, we have three wave numbers $k$, $\eta$, $\alpha$. We first pick up the one that has the largest norm, namely $ \iota(k,\eta,\alpha)$. Then, we introduce the resonance time based on this wave number, although for either $k$ or $\alpha$ larger than $\eta$, the system does not have resonance. Then we choose either $w_{\mathrm{R}}$ or $w_{\mathrm{NR}}$ based on the relationship between $t, k$, and $\iota(k,\eta,\alpha)$. The total growth is enlarged in this way.
\begin{rem}\label{rem-2}
    Note that if $k=0$ or $k=\iota(k,\eta,\alpha)$, the weight $w_k(t, \iota(k,\eta,\alpha))=w_{\nr}(t,\iota(k,\eta,\alpha))$ for all $t\geq 2$. 
\end{rem}

Now we introduce the multiplier that will be used in the construction of the energy functional. For $\sig\ge0,\fr12 <s\leq 1,$ let us define
\[
A^{\sig}_k(t,\eta,\al)=e^{\lm(t)|k,\eta,\al|^s}\la k,\eta,\al\ra^{\sig}{\rm J}_k(t,\eta,\al),
\]
with 
\begin{equation}\label{def-lambda}
\lambda(t):=\lambda_\infty+\fr{\tilde{\delta}}{(1+t)^a},
\end{equation}
for positive constants $\tilde{\delta}$ and $a$ is chosen sufficiently small such that $a<\min\{\fr{s}{4},s-\fr12\}$ and $\lm(0)=\lm_\infty+\tilde{\delta}<\frac{9}{10}\lm_{\rm in}$.

For the top norm of the zero mode $\theta_0$, we also use the following multiplier
\begin{align}\label{B}
B(\eta,\al)=\la |\eta|^{\fr12},\al\ra.
\end{align}
The anisotropic multiplier $B$ still provides us with a good commutator estimate. 

\subsection{Base properties of the main Fourier multiplier}
In this subsection, we present estimates associated with various Fourier multipliers employed in the paper. The majority of them rely on the properties of $w_{\nr}$ and $w_{\res}$ studied in \cite{BM1}. 
Let us begin by introducing a useful property satisfied by the function $\iota$: 
\begin{lemma}[\cite{LiMasmoudiZhao2022asymptotic}]\label{lem-switch}
  It holds that
  \begin{align*}
  \big||\iota(k,\eta,\alpha)|-|\iota(l,\xi,\beta)|\big|\le \left| k-l,\eta-\xi,\alpha-\beta\right|.
\end{align*}
\end{lemma}
\begin{proof}
    Without loss of generality, we let $\alpha=\iota(k,\eta,\alpha)$. Then
    \begin{align*}
\big||\iota(k,\eta,\alpha)|-|\iota(l,\xi,\beta)|\big|
&= \big||\alpha|-\max\{|l|, |\xi|, |\beta|\}\big|. 
    \end{align*}
    If $|\al|\ge \max\{|l|, |\xi|, |\beta|\}$, then $\big||\alpha|-\max\{|l|, |\xi|, |\beta|\}\big|\le |\al-\beta|$. For the case $|\al|< \max\{|l|, |\xi|, |\beta|\}$,
    \[
    \big||\alpha|-\max\{|l|, |\xi|, |\beta|\}\big|\le
    \begin{cases}
    |\beta|-|\al|\le|\beta-\al|,\quad{\rm if}\quad\max\{|l|,|\xi|,|\beta|\}=|\beta|,\\
    |l|-|\al|\le|l-k|,\quad\ \ {\rm if}\quad\max\{|l|,|\xi|,|\beta|\}=|l|,\\
    |\xi|-|\al|\le|\xi-\eta|,\quad \ {\rm if}\quad\max\{|l|,|\xi|,|\beta|\}=|\xi|.
    \end{cases}
    \]\qedhere

\end{proof}

To simplify the presentation, let us denote 
\be\label{iota12}
\iota_1=\iota(k,\eta,\al), \quad \iota_2=\iota(l,\xi,\beta), 
\ee
and
\[
\mathfrak{A}_0=\left\{\big((k,\eta,\alpha),(l,\xi,\beta)\big):  |k-l,\eta-\xi,\beta-\alpha| \leq \frac{3}{16}|l,\xi,\beta| \right\}.
\]
\begin{lemma}\label{lem-ratio-J}
The following inequality holds for $t> 10$
\begin{align}\label{ratio-J}
\dfrac{{\rm J}_k(t,\eta,\alpha)}{{\rm J}_l(t,\xi,\beta)} &\lesssim \bigg({\mathds{1}}_{\textup{A}}\dfrac{|\iota_1|}{|k|^2(1+|t-\frac{\iota_1}{k}|)}
+{\mathds{1}}_{\textup{B}}\frac{|l|^2\Big(1+|t-\frac{\iota_2}{l}|\Big)}{|\iota_2|} 
+{\mathds{1}}_{{\rm C}}\bigg)e^{2\mu|k-l,\eta-\xi,\alpha-\beta|^{\fr12}},
\end{align}
where $\textup{A}=\{t \in \textup{I}_{k,\iota_1}\cap \textup{I}^c_{l,\iota_2}, k \neq l\}$,  $\textup{B}=\{t \in \textup{I}^c_{k,\iota_1}\cap  \textup{I}_{l,\iota_2}\}$ and ${\rm C}=\mathfrak{A}_0\setminus (\textup{A}\cup \textup{B})$.
\end{lemma}
\begin{proof}
If $\big((k,\eta,\al),(l,\xi,\beta)\big)\in{\rm A}\cup{\rm B}$, we can get \eqref{ratio-J} by using the definition of ${\rm J}$, \eqref{wR},  Lemmas \ref{lem-switch} and \ref{lem-r-wNR} immediately. If $t \in \textup{I}_{k,\iota_1}\cap \textup{I}^c_{k,\iota_2}$, we have
\[
\begin{aligned}
\dfrac{{\rm J}_k(t,\eta,\alpha)}{{\rm J}_l(t,\xi,\beta)}\les& \dfrac{|\iota_1|}{|k|^2(1+|t-\frac{\iota_1}{k}|)}e^{\mu|k-l,\eta-\xi,\alpha-\beta|^{\fr12}}\\
=& \dfrac{|\iota_1|-|\iota_2|+|\iota_2|}{|k|^2(1+|t-\frac{\iota_2}{k}|)}\fr{1+|t-\frac{\iota_2}{k}|}{1+|t-\frac{\iota_1}{k}|}e^{\mu|k-l,\eta-\xi,\alpha-\beta|^{\fr12}}\les e^{2\mu|k-l,\eta-\xi,\alpha-\beta|^{\fr12}},
\end{aligned}
\]
since $\fr{|\iota_2|}{|k|^2(1+|t-\fr{\iota_2}{k}|)}\les1$ for $t\in{\rm I}_{k,\iota_2}^c$ with $k\ne0$. Similarly,  if $t \in \textup{I}_{k,\iota_1}\cap \textup{I}_{k,\iota_2}$, we have $|\iota_1|\approx|\iota_2|$, and
\[
\dfrac{{\rm J}_k(t,\eta,\alpha)}{{\rm J}_l(t,\xi,\beta)}\les \dfrac{|\iota_1|}{|\iota_2|}\fr{1+|t-\frac{\iota_2}{k}|}{1+|t-\frac{\iota_1}{k}|}e^{\mu|k-l,\eta-\xi,\alpha-\beta|^{\fr12}}\les e^{2\mu|k-l,\eta-\xi,\alpha-\beta|^{\fr12}}.
\]
For the case $t \in \textup{I}_{k,\iota_1}\cap \textup{I}_{l,\iota_2}$ with $l\ne k$, thanks to the condition $\big((k,\eta,\al), (l,\xi,\beta) \big)\in\mathfrak{A}_0$ and Lemma \ref{lem-switch}, we find that
\[
|\iota_1-\iota_2|\le|k-l, \eta-\xi, \al-\beta|\le\fr{9}{16} |\iota_2|.
\]
Consequently,
\begin{equation}\label{iota1=iota2}
\fr{7}{16}|\iota_2|\le|\iota_1|\le\fr{25}{16}|\iota_2|.
\end{equation}
This enables us to Lemma \ref{Scenarios} to bound $\fr{|\iota_1|}{k^2(1+|t-\fr{\iota_1}{k}|)}$. More precisely, together with Lemma \ref{lem-switch}, if (2) or (3) in Lemma \ref{Scenarios} holds, we are led to
\begin{equation}\label{good-bd}
\fr{|\iota_1|}{k^2(1+|t-\fr{\iota_1}{k}|)}\les\la k-l,\eta-\xi,\al-\beta\ra.
\end{equation}
Then we obtain
\[
\dfrac{{\rm J}_k(t,\eta,\alpha)}{{\rm J}_l(t,\xi,\beta)}\les \dfrac{|\iota_1|}{|k|^2(1+|t-\frac{\iota_1}{k}|)}e^{\mu|k-l,\eta-\xi,\alpha-\beta|^{\fr12}}\les e^{2\mu|k-l,\eta-\xi,\alpha-\beta|^{\fr12}}.
\]
Finally, if $t\in{\rm I}_{k,\iota_1}^c\cap{\rm I}_{k,\iota_2}^c$, by using Lemma \ref{lem-r-wNR}, we get $\dfrac{{\rm J}_k(t,\eta,\alpha)}{{\rm J}_l(t,\xi,\beta)}\les e^{\mu|k-l,\eta-\xi,\alpha-\beta|^{\fr12}}$ immediately.
\end{proof}
\begin{remark}
    The condition $\big((k,\eta,\al), (l,\xi,\beta) \big)\in\mathfrak{A}_0$ is merely used for the case $t \in \textup{I}_{k,\iota_1}\cap \textup{I}_{l,\iota_2}$ with $l\ne k$ in the proof of Lemma \ref{lem-ratio-J}.
\end{remark}
\begin{remark}
    It holds that for $t\leq \fr12\min\{\sqrt{|\iota(k,\eta,\alpha)|}, \sqrt{|\iota(l,\xi,\beta)|}\}$, 
    \begin{align*}
        \left|\frac{\mathrm{J}_{k}(0,\eta,\alpha)}{\mathrm{J}_{l}(0,\xi,\beta)}\right|\lesssim e^{2\mu|k-l,\eta-\xi,\alpha-\beta|^{\fr12}}.
    \end{align*}
\end{remark}

\begin{lemma}\label{short time}
    For $t\leq \fr12\min\{\sqrt{|\iota(k,\eta,\alpha)|}, \sqrt{|\iota(l,\xi,\beta)|}\}$, it holds that 
    \begin{equation}\label{Jk|Jl-1}
    \begin{aligned}
        \left|\frac{\mathrm{J}_{k}(t,\eta,\alpha)}{\mathrm{J}_{l}(t,\xi,\beta)}-1\right|\lesssim \fr{\JB{k-l,\eta-\xi,\alpha-\beta}}{\sqrt{|\iota(k,\eta,\alpha)|}+\sqrt{|\iota(l,\xi,\beta)|}}e^{3\mu|k-l,\eta-\xi,\alpha-\beta|^{\fr12}}.
    \end{aligned}
    \end{equation}
\end{lemma}
 
\begin{proof}
The proof of Lemma \ref{short time} can be obtained by adjusting the proof of Lemma 3.7 in \cite{BM1}, we give the details below for the sake of completeness. Let $\iota_1$ and $\iota_2$ be the same as in \eqref{iota12}.
    Under the assumption of this lemma, we know that by \eqref{def-wk} and \eqref{multiplier J}, ${\rm{J}}_k(t,\eta,\alpha)={\rm{J}}_k(0,\eta,\alpha)$ and similarly, ${\rm{J}}_l(t,\xi,\beta)={\rm{J}}_l(0,\xi,\beta)$. Assume that $\frac{1}{100}\left(\sqrt{|\iota_1|}+\sqrt{|\iota_2|}\right)<|k-l|+|\eta-\xi|+|\alpha-\beta|$. Combining this with the upper-bound of the ratio in \eqref{ratio-J} when the support is $\rm{C}$, one easily infer that 
    
    \begin{align*}
        \left|\frac{\mathrm{J}_{k}(t,\eta,\alpha)}{\mathrm{J}_{l}(t,\xi,\beta)}-1\right|\lesssim\JB{k-l,\eta-\xi,\alpha-\beta}^{\fr12} e^{2\mu|k-l,\eta-\xi,\alpha-\beta|^{\fr12}}\lesssim \fr{\JB{k-l,\eta-\xi,\alpha-\beta}}{\sqrt{|\iota_1|}+\sqrt{|\iota_2|}} e^{3\mu|k-l,\eta-\xi,\alpha-\beta|^{\fr12}},
    \end{align*}
    which yields \eqref{Jk|Jl-1}.

    Observe that when $\sqrt{|\iota_1|}+\sqrt{|\iota_2|} \lesssim 1$, we can follow the above procedure to arrive at $\eqref{Jk|Jl-1}$. Hence, for the remaining part of the proof we focus on the case when \begin{equation}\label{sum iota larger}
        \sqrt{|\iota_1|}+\sqrt{|\iota_2|} \gtrsim 1.
    \end{equation}
    Under this restriction, we assume that 
    \begin{equation}\label{assmptn}
        \frac{1}{100}\left(\sqrt{|\iota_1|}+\sqrt{|\iota_2|}\right)\geq|k-l|+|\eta-\xi|+|\alpha-\beta| \geq \left|\iota_1-\iota_2\right|.
    \end{equation}
    It is easy to see that via \eqref{assmptn} and $\eqref{sum iota larger}$, $|\sqrt{|\iota_1|}-\sqrt{|\iota_2|}|\leq \frac{1}{100}<1$.  From this we can infer that $|\E(\sqrt{|\iota_1|})-\E(\sqrt{|\iota_2|})|\leq 1$. Let us first suppose that $\E(\sqrt{|\iota_1|})=\E(\sqrt{|\iota_2|})$. Upon recalling the definition \eqref{multiplier J}, exploiting the inequality
    \[
    \left(1+\frac{a}{b^2}\right)^{bc}-1 \leq ce^c \frac{a}{b}
    \]
    for positive constants $a,b,c$
    and using the total growth of $w$ in Lemma~\ref{total growth}, we can infer that
        \[
        \left|\fr{w_l(t,\iota_2)}{w_k(t,\iota_1)}-1\right|=\left|\fr{w_l(0,\iota_2)}{w_k(0,\iota_1)}-1\right|=\left| \big(\fr{\iota_1}{\iota_2}\big)^{c\E(\sqrt{|\iota_1|})} -1\right|\lesssim \fr{|\iota_1-\iota_2|}{\sqrt{|\iota_2|}}.
        \]
    Now let us consider the scenario when $\E(\sqrt{|\iota_1|})=\E(\sqrt{|\iota_2|})+1$. In this case it is clear that $\sqrt{|\iota_2|}<\E(\sqrt{|\iota_1|})\leq\sqrt{|\iota_1|}$. Hence,
    \[
    \begin{aligned}
        \left|\fr{w_l(t,\iota_2)}{w_k(t,\iota_1)}-1\right|=\left|\fr{w_l(0,\iota_2)}{w_k(0,\iota_1)}-1\right|&=\left| \big(\fr{\iota_1}{\iota_2}\big)^{c\E(\sqrt{|\iota_2|})}\left(\fr{|\iota_1|}{\E(\sqrt{|\iota_1|})^2}\right)^c -1\right|\\
        &\lesssim \fr{|\iota_1-\iota_2|}{\sqrt{|\iota_2|}}+\left|\left(\fr{|\iota_1|}{\E(\sqrt{|\iota_1|})^2}\right)^c-1\right| \lesssim \fr{\JB{\iota_1-\iota_2}}{\sqrt{|\iota_2|}}.
    \end{aligned}
    \]
    The estimate for remaining scenario when $\E(\sqrt{|\iota_1|})=\E(\sqrt{|\iota_2|})-1$ can be done similarly.
\end{proof}
Next, we introduce a sequence of lemmas presenting estimates on the time-dependent Fourier multipliers which are mainly used in dealing with the nonlinear interactions. We organize them according to the resonant and non-resonant interactions. To that end, we define the following set
\begin{equation}\label{Set A}
    \mathfrak{A}=\left\{((k,\eta,\alpha),(l,\xi,\beta)):  |(k-l,\eta-\xi,\beta-\alpha)| \leq \frac{3}{16}|l,\xi,\beta|\ \text{and}\ |l|\neq 0 \right\}.
\end{equation}
Clearly, any element in $\mathfrak{A}$ satisfies the condition
\begin{equation}\label{comparability}
     \frac{13}{16}|l,\xi,\beta| \leq |k,\eta,\alpha|\leq \frac{19}{16}|l,\xi,\beta|.
\end{equation}

To begin with, let us consider the case where resonance or nonresonance does not matter.
\begin{lemma}\label{l xi nonresonant}
For any $t> 10$, let us denote
\begin{align}
\mathfrak{U}_1&=\{l\xi<0\},\quad
\mathfrak{U}_2=\{l \xi>0, |\xi|\leq \frac{1}{2}|lt|\, \text{or}\, |\xi|\geq 2|lt|\},\\
\mathfrak{U}_3&=\{l \xi>0, \frac{1}{2}|lt|<|\xi|<2|lt|, |l,\xi,\beta|\leq 10^3\},\\
\mathfrak{U}_4&=\{|\xi|<t/2\}, \quad
\mathfrak{U}_5=\{|\beta|\geq\frac{1}{100}|\xi| \},
\end{align}
and
$\mathfrak{U}=\mathfrak{U}_1\cup\mathfrak{U}_2\cup \mathfrak{U}_3\cup\mathfrak{U}_4\cup\mathfrak{U}_5.$
Assume that $\big((k,\eta,\alpha),(l,\xi,\beta)\big)\in \mathfrak{A}$ and $(l,\xi,\beta)\in \mathfrak{U}$, then
\begin{align}\label{approximate}
    l^2+(\xi-lt)^2+\beta^2\approx l^2+\xi^2+l^2t^2+\beta^2,
\end{align}
and
\begin{equation}\label{est when res nonres not considered}
\begin{aligned}
&\dfrac{(t|\beta|+|\xi|+|l,\beta|)|l,\beta|}{(l^2+(\xi-lt)^2+\beta^2)^2}\frac{A_k^{\sig_1}(\eta,\alpha)}{A_l^{\sig_1}(\xi,\beta)}\\
&\lesssim  e^{\mathbf{c}\lambda(t)|k-l,\eta-\xi,\alpha-\beta|^{s}}\bigg(\frac{1}{\JB{t}^2}+\frac{|l,\xi,\beta|^\fr{s}{2} |k,\eta,\alpha|^\fr{s}{2}}{\JB{t}^{1+s}}\\
&\qquad\qquad\qquad\qquad\qquad 
+ \sqrt{\dfrac{\partial_t w_k(t,\iota(k,\eta,\alpha))}{w_k(t,\iota(k,\eta,\alpha))}}\sqrt{\dfrac{\partial_t w_l(t,\iota(l,\xi,\beta))}{w_l(t,\iota(l,\xi,\beta))}}\bigg).
\end{aligned}
\end{equation}
\end{lemma}
\begin{proof}
We begin by proving \eqref{approximate}. By triangle inequality, the `` $\lesssim$" direction of \eqref{approximate} holds. Therefore, it suffices to prove it for the lower bound of the term on the left-hand side of\eqref{approximate}. For $(l,\xi,\beta)\in \mathfrak{U}_1$, clearly, $(\xi-lt)^2=(\xi^2-2\xi lt+l^2t^2)\geq \xi^2+l^2t^2$. When $(l,\xi,\beta)\in \mathfrak{U}_2$, note that $|\xi-lt|^2\geq(|\xi|-|lt|)^2\gtrsim (|\xi|+|lt|)^2\gtrsim \xi^2+l^2t^2$. Similar steps can be taken to show \eqref{approximate} when $(l,\xi,\beta)\in \mathfrak{U}_4$.   For  $(l,\xi,\beta)\in \mathfrak{U}_3$, notice that $|lt|\approx|\xi|\les |l|$ due to $l\ne0$, then \eqref{approximate} holds. Lastly, for $(l,\xi,\beta)\in \mathfrak{U}_5\cap (\mathfrak{U}_1\cup\mathfrak{U}_2\cup\mathfrak{U}_3\cup\mathfrak{U}_4)^c$, we have $|\beta|\ge\fr{|\xi|}{100}$,  $l\xi>0, |\xi|\approx |lt|, |l,\xi,\beta|>1000$, and $2|\xi|\geq t$. With this in mind, we can infer that $l^2+(\xi-lt)^2 +\beta^2 \gtrsim l^2+\xi^2+\beta^2\gtrsim l^2+\xi^2+l^2t^2+\beta^2$ where the last inequality is obtained via $|\xi|\approx |lt|$.

Now, we devote our attention to proving \eqref{est when res nonres not considered}. First let us assume that $(l,\xi,\beta) \in \mathfrak{U}_3$. Notice that by \eqref{comparability}, one can also assert that $|k,\eta,\alpha|\lesssim 1$. Thus,
\[
\begin{aligned}
\dfrac{(t|\beta|+|\xi|+|l,\beta|)|l,\beta|}{(l^2+(\xi-lt)^2+\beta^2)^2}\frac{A^{\sig_1}_k(\eta,\alpha)}{A^{\sig_1}_l(\xi,\beta)}&\lesssim \frac{\la t\ra \JB{\frac{\xi}{l}}^4}{|l|^4\JB{\frac{\xi}{l}-t}^4\JB{\frac{\xi}{l}}^4} \lesssim \frac{1}{\JB{t}^3}.
\end{aligned}
\]

Next we focus on $(l,\xi,\beta)\in \big(\mathfrak{U}_1\cup\mathfrak{U}_2\cup \mathfrak{U}_4\cup \mathfrak{U}_5\big)\cap \mathfrak{U}_3^c$. Note first that the restriction $\big((k,\eta,\alpha),(l,\xi,\beta)\big)\in \mathfrak{A}$ implies that \eqref{iota1=iota2} holds. 
We consider the following two cases: $t\notin \mathrm{I}_{k,\iota_1}$ and $t\in \mathrm{I}_{k,\iota_1}$. 

 We first focus on the case when $t \notin {\rm{I}}_{k,\iota_1}$ or $t\in {\rm I}_{k,\iota_1}\cap {\rm I}_{k,\iota_2}^c$. By Lemma \ref{lem-ratio-J}, \eqref{triangle1} and \eqref{approximate}, we have
\begin{align*}
\dfrac{(t|\beta|+|\xi|+|l,\beta|)|l,\beta|}{(l^2+(\xi-lt)^2+\beta^2)^2}\frac{A_k^{\sigma_1}(\eta,\alpha)}{A_l^{\sigma_1}(\xi,\beta)}&\lesssim \dfrac{t|l,\beta|^2+|\xi||l,\beta|}{\xi^4+(lt)^4+\beta^4} e^{c'\lambda(t)|k-l,\eta-\xi,\alpha-\beta|^{s}}e^{2\mu|k-l,\eta-\xi,\alpha-\beta|^{\fr12}}\\&\lesssim\frac{|l,\beta|^s}{\JB{t}^{1+s}}e^{c\lambda(t)|k-l,\eta-\xi,\alpha-\beta|^{s}}
\\&\lesssim
\frac{\la l,\beta\ra^{\fr{s}{2}}\la k,\alpha\ra^{\fr{s}{2}}}{\JB{t}^{1+s}} e^{\mathbf{c}\lambda(t)|k-l,\eta-\xi,\alpha-\beta|^{s}}.
\end{align*}

Now, we focus on the case when $t \in {\rm{I}}_{k,\iota_1}\cap {\rm I}_{l,\iota_2}^c$ with $k\ne l$.
If $t\le 2\sqrt{|\iota_2|}\approx\sqrt{|\iota_1|}$, then we infer from the fact $t\approx\fr{|\iota_1|}{|k|}$ that $|k|\approx \sqrt{|\iota_1|}$ due to $t \in {\rm{I}}_{k,\iota_1}$. Or if $t\ge2|\iota_2|$, then $|t-\fr{\iota_2}{k}|\gtrsim|\iota_2|$.
Thus, for both cases, \eqref{good-bd} holds, and
\begin{equation}\label{good-bd2}
\begin{aligned}
    \dfrac{(t|\beta|+|\xi|+|l,\beta|)|l,\beta|}{(l^2+(\xi-lt)^2+\beta^2)^2}\frac{A_k^{\sigma_1}(\eta,\alpha)}{A_l^{\sigma_1}(\xi,\beta)} 
    &\lesssim 
    \dfrac{t|l,\beta|^2+|\xi||l,\beta|}{\xi^4+(lt)^4+\beta^4} \dfrac{|\iota_1|}{|k|^2(1+|t-\frac{\iota_1}{k}|)} e^{c\lambda(t)|k-l,\eta-\xi,\alpha-\beta|^{s}}\\
    &\lesssim \frac{|l,\xi,\beta|^{\fr{s}{2}}|k,\eta,\alpha|^{\fr{s}{2}}}{\JB{t}^{1+s}} e^{\mathbf{c}\lambda(t)|k-l,\eta-\xi,\alpha-\beta|^{s}}.
\end{aligned}
\end{equation}
Next we focus on $t\in[2\sqrt{|\iota_2|}, 2|\iota_2|]$. Then there exists $m$ such that $t\in{\rm I}_{m,\iota_2}$. If $m=k$, 
\[
\begin{aligned}
    \dfrac{(t|\beta|+|\xi|+|l,\beta|)|l,\beta|}{(l^2+(\xi-lt)^2+\beta^2)^2}\frac{A_k^{\sigma_1}(\eta,\alpha)}{A_l^{\sigma_1}(\xi,\beta)} 
    &\lesssim 
    \dfrac{t|l,\beta|^2+|\xi||l,\beta|}{\xi^4+(lt)^4+\beta^4} \dfrac{|\iota_1|}{|k|^2(1+|t-\frac{\iota_1}{k}|)} e^{c\lambda(t)|k-l,\eta-\xi,\alpha-\beta|^{s}}\\
    &\lesssim \dfrac{1}{\sqrt{1+|t-\frac{\iota_1}{k}|}}\dfrac{1}{\sqrt{1+|t-\frac{\iota_2}{m}|}}\fr{\sqrt{1+|t-\frac{\iota_2}{m}|}}{\sqrt{1+|t-\frac{\iota_1}{k}|}} e^{{\bf c}\lambda(t)|k-l,\eta-\xi,\alpha-\beta|^{s}}\\
    &\les \sqrt{\fr{\pr_tw_k(t,\iota_1)}{w_k(t,\iota_1)}}\sqrt{\fr{\pr_tw_l(t,\iota_2)}{w_l(t,\iota_2)}}e^{{\bf c}\lambda(t)|k-l,\eta-\xi,\alpha-\beta|^s}.
\end{aligned}
\]
For the case $m\ne k$,  (2) or (3)  in Lemma \ref{Scenarios} holds, then we have \eqref{good-bd}, which gives \eqref{good-bd2}.
\end{proof}

\begin{remark}\label{rem-iota}
Before proceeding any further, we would like to emphasize that under the restrictions $(l,\xi,\beta)\notin \mathfrak{U}$ and $((k,\eta,\alpha),(l,\xi,\beta))\in \mathfrak{A}$ with $t\geq 10$, namely, 
$|k-l,\eta-\xi,\alpha-\beta| \leq \frac{3}{16}|l,\xi,\beta|, \frac{|lt|}{2}<|\xi|<2|lt|, |\beta| < \frac{1}{100}|\xi|, t>10$,  it holds that $|l|<\fr{|\xi|}{5}$, and hence
\[
\begin{aligned}
&|\alpha|\leq \frac{619}{1600}|\xi|, \qquad |k|\leq \frac{683}{1600}|\xi|,\qquad \frac{1237}{1600}|\xi|\leq |\eta|\le\fr{1963}{1600}|\xi|,
\end{aligned}
\]
which yields $|\alpha| < |\eta|$ and $|k|<|\eta|$. Referring back to the definition \eqref{eq-def-iota}, we therefore have $\iota_1=\iota(k,\eta,\alpha)=\eta$ and $\iota_2=\iota(l,\xi,\beta)=\xi$, where this fact shall be heavily used in the sequence of lemmas below.
\end{remark}

The next lemma concerns the interaction between non-resonant frequencies.
\begin{lemma}[Nonresonant-Nonresonant]\label{est:NR-NR}
Let $t \in {\rm{I}}^c_{k,\eta} \cap {\rm{I}}^c_{l,\xi}$ and $\big((k,\eta,\alpha), (l,\xi,\beta)\big) \in \mathfrak{A}$. Then,

\[
\begin{aligned}
&\dfrac{(t|\beta|+|\xi|+|l,\beta|)|l,\beta|}{(l^2+(\xi-lt)^2+\beta^2)^2}\frac{A^{\sig_1}_k(\eta,\alpha)}{A^{\sig_1}_l(\xi,\beta)}\\
&\qquad \qquad\lesssim \bigg(\frac{1}{\JB{t}^2}+{ \frac{|l,\xi,\beta|^\fr{s}{2} |k,\eta,\alpha|^\fr{s}{2}}{\JB{t}^{1+s}}} + \sqrt{\dfrac{\partial_t w_k(t,\iota_1)}{w_k(t,\iota_1)}}\sqrt{\dfrac{\partial_t w_l(t,\iota_2)}{w_l(t,\iota_2)}}\bigg)e^{\mathbf{c}\lambda(t)|k-l,\eta-\xi,\alpha-\beta|^{s}}
\end{aligned}
\]
where $\iota_1=\iota(k,\eta,\alpha)$ and $\iota_2=\iota(l,\xi,\beta)$. 
\end{lemma}
\begin{proof}
If $(l,\xi,\beta) \in \mathfrak{U}$, then Lemma~\ref{l xi nonresonant} gives us the desired inequality.
Here, we just focus on the case when $(l,\xi,\beta)\in\mathfrak{U}^c$, namely, 
  \begin{equation}\label{assumptions on NR-NR}
   l\xi>0, \quad \frac{|lt|}{2}<|\xi|<2|lt|, \quad |\beta|<\frac{1}{100}|\xi|,\quad \text{ and } |l,\xi,\beta|>1000.
  \end{equation}
Thus under the restriction in \eqref{assumptions on NR-NR} and by Remark~\ref{rem-iota}, we know that $\iota_1=\eta$ and $\iota_2=\xi$. Furthermore, it holds that
\begin{align}\label{eq: use it}
    \frac{(t|\beta|+|\xi|+|l,\beta|)|l,\beta|}{(l^2+(\xi-lt)^2+\beta^2)^2}\les \frac{t|l,\beta|^2}{(l^2+(\xi-lt)^2+\beta^2)^2}\lesssim \frac{t}{l^2+(\xi-lt)^2+\beta^2}.
\end{align} 
Therefore, we shall focus on estimating the upper bound above. 
    Due to the fact that both $(k,\eta,\alpha)$ and $(l,\xi,\beta)$ are non-resonant, by \eqref{triangle1} and Lemma~\ref{lem-ratio-J}, we are led to
  \begin{equation}\label{ratio A nrnr}
       \dfrac{A^{\sig_1}_k(\eta,\alpha)}{A^{\sig_1}_l(\xi,\beta)}\lesssim e^{c'\lambda(t)|k-l,\eta-\xi,\alpha-\beta|^{s}+2\mu |k-l,\eta-\xi,\alpha-\beta|^{\fr12}}\les e^{c\lambda(t)|k-l,\eta-\xi,\alpha-\beta|^{s}},
  \end{equation}
  for some $0<c'<c<1$.
  We are left to estimate  $\frac{t}{l^2+(\xi-lt)^2+\beta^2}$.  Indeed, the fact $t\in{\rm I}^c_{l,\xi}$ with $l\ne0$ imply that
  \begin{equation}\label{tnoinIlxi}
      \fr{|\xi|}{l^2(1+|t-\fr{\xi}{l}|)}\les1.
  \end{equation}
Combining this with the fact $t\approx\fr{|\xi|}{|l|}$ yield
\begin{align*}
    \fr{t}{l^2+(\xi-lt)^2+\beta^2}\les\fr{|\xi|^2}{|l|^4(1+|t-\fr{\xi}{l}|)^2}\cdot\fr{|l|}{|\xi|}\les\fr{1}{t}\les \fr{|\xi|^s}{\la t\ra^{1+s}}.
\end{align*}
The proof of Lemma \ref{est:NR-NR} is completed.
\end{proof}
  
  We now proceed and discuss the interactions between resonant and non-resonant frequencies. 

 \begin{lemma}[Resonant-Nonresonant]\label{R NR}
Let  $t\in {\rm{I}}_{k,\eta} \cap {\rm{I}}^c_{l,\xi}$ and $\big((k,\eta,\alpha), (l,\xi,\beta)\big) \in \mathfrak{A}$. Then,
\[
\begin{aligned}
&\dfrac{(t|\beta|+|\xi|+|l,\beta|)|l,\beta|}{(l^2+(\xi-lt)^2+\beta^2)^2}\frac{A^{\sig_1}_k(\eta,\alpha)}{A^{\sig_1}_l(\xi,\beta)}\\
&\qquad \qquad\lesssim \bigg(\frac{1}{\JB{t}^2}+{\frac{|l,\xi,\beta|^\fr{s}{2} |k,\eta,\alpha|^\fr{s}{2}}{\JB{t}^{1+s}}} + \sqrt{\dfrac{\partial_t w_k(t,\iota_1)}{w_k(t,\iota_1)}}\sqrt{\dfrac{\partial_t w_l(t,\iota_2)}{w_l(t,\iota_2)}}\bigg)e^{\mathbf{c}\lambda(t)|k-l,\eta-\xi,\alpha-\beta|^{s}},
\end{aligned}
\]
where $\iota_1=\iota(k,\eta,\alpha)$ and $\iota_2=\iota(l,\xi,\beta)$. 
 \end{lemma}
 \begin{proof}
We split the estimate into two major cases:

\textbf{Case 1: $(l,\xi,\beta) \in \mathfrak{U}$.} Here, Lemma~\ref{l xi nonresonant} gives us the desired inequality.

\textbf{Case 2: $(l,\xi,\beta) \in \mathfrak{U}^c$}. Now the estimates \eqref{assumptions on NR-NR}, \eqref{eq: use it},  and the facts in Remark \ref{rem-iota} hold.  Moreover, here the restriction $t\in \textup{I}_{k,\eta}$ implies that $\eta k>0$, $1\leq |k|\leq E(|\eta|^{\frac{1}{2}})$, $t\approx\fr{|\eta|}{|k|}$ and $t\le2|\eta|$. 
Thus, we have
\be\label{approx}
|kt|\approx|\eta|\approx |\xi|\approx |lt|.
\ee
     By Lemma~\ref{lem-ratio-J} and \eqref{eq: use it}, we are led to
     \[
     \dfrac{(t|\beta|+|\xi|+|l,\beta|)|l,\beta|}{(l^2+(\xi-lt)^2+\beta^2)^2}\frac{A^{\sig_1}_k(\eta,\alpha)}{A^{\sig_1}_l(\xi,\beta)}\lesssim \frac{t}{l^2+(\xi-lt)^2+\beta^2} \dfrac{|\eta|}{|k|^2(1+|t-\frac{\eta}{k}|)} e^{c\lm(t)|k-l,\eta-\xi,\alpha-\beta|^{s}}.
     \]
     Our goal now is to bound the product of two fraction symbols appearing in the above inequality.  We split the analysis into two time sub-intervals:
\begin{enumerate}
     \item $10\leq t<1000 |\xi|^{\frac{1}{2}}$. Here  $|\xi|\les |l|^2$ holds. This, together with  \eqref{approx}, implies that
     \begin{align*}
        &\frac{t}{l^2+(\xi-lt)^2+\beta^2} \dfrac{|\eta|}{|k|^2(1+|t-\frac{\eta}{k}|)} \les \fr{|\xi|^2}{|l|^5}
        \les \fr{1}{|\xi|^{\fr12}}\les\fr{1}{\la t\ra}\les \frac{|\xi|^s}{\JB{t}^{1+2s}}.
     \end{align*}
     \item $1000 |\xi|^\fr12\leq t\leq \min\{2 |\xi|,2|\eta|\}$. Now there exists an $m$ such that $t \in {\rm{I}}_{m,\xi} \cap {\rm{I}}_{k,\eta} $. Hence, we infer from \eqref{tnoinIlxi} and \eqref{approx} that
     \begin{align*}
        \frac{t}{l^2+(\xi-lt)^2+\beta^2} \dfrac{|\eta|}{|k|^2(1+|t-\frac{\eta}{k}|)} \lesssim \frac{1}{|k|(1+|t-\frac{\eta}{k}|)}.
     \end{align*}
     It then remains to bound the term $1/(|k|(1+|t-\frac{\eta}{k}|)$. By Lemma~\ref{Scenarios}, we are led to the following scenarios:
     
     a) If $m=k$, then straightforward computation yields
     \[
     \frac{1}{1+|t-\frac{\eta}{k}|} \lesssim \frac{1}{\sqrt{1+|t-\frac{\eta}{k}|}}\frac{1}{\sqrt{1+|t-\frac{\xi}{m}|}} \JB{\eta-\xi}^\fr12.
     \]
     
     b) If $m \neq k$, then we perform our estimation in two sub-scenarios:
     
          If $|t-\frac{\eta}{k}|\gtrsim \frac{|\eta|}{|k|^2}$, then
         $\frac{1}{|k|(1+|t-\frac{\eta}{k}|)} \lesssim \frac{1}{t}\lesssim \frac{|\xi|^s}{\JB{t}^{1+s}}$;
         
          If $|\eta-\xi|\gtrsim \frac{|\eta|}{|m|}$, then we have by $\frac{|\eta|}{|m|}\approx \frac{|\eta|}{|k|}\approx t$ that
        $\frac{1}{1+|t-\frac{\eta}{k}|} \lesssim 1 \lesssim \frac{|\eta-\xi|^2}{\JB{t}^2}$.
        \end{enumerate}
        The proof of Lemma \ref{R NR} is completed.\qedhere
        \end{proof}

The following lemma focuses on the interaction between non-resonant and resonant frequencies. 
\begin{lemma}[Nonresonant-Resonant]\label{est: NR-R}
 Suppose that $t\in {\rm{I}}^c_{k,\eta} \cap {\rm{I}}_{l,\xi}$ and $\big((k,\eta,\alpha), (l,\xi,\beta) \big)\in \mathfrak{A}$. Then, 
\[
\begin{aligned}
&\dfrac{(t|\beta|+|\xi|+|l,\beta|)|l,\beta|}{(l^2+(\xi-lt)^2+\beta^2)^2}\frac{A^{\sig_1}_k(\eta,\alpha)}{A^{\sig_1}_l(\xi,\beta)}\\
&\qquad \qquad \lesssim \bigg(\frac{1}{\JB{t}^2}+\frac{|l,\xi,\beta|^\fr{s}{2} |k,\eta,\alpha|^\fr{s}{2}}{\JB{t}^{1+s}} + \sqrt{\dfrac{\partial_t w_k(t,\iota_1)}{w_k(t,\iota_1)}}\sqrt{\dfrac{\partial_t w_l(t,\iota_2)}{w_l(t,\iota_2)}}\bigg)e^{\mathbf{c}\lambda(t)|k-l,\eta-\xi,\alpha-\beta|^{s}}.
\end{aligned}
\] 
 \end{lemma}
\begin{proof}
We split the estimate into two major cases as above:

\textbf{Case 1: $(l,\xi,\beta) \in \mathfrak{U}$.} Here, Lemma~\ref{l xi nonresonant} gives us the desired inequality.

\textbf{Case 2: $(l,\xi,\beta) \in \mathfrak{U}^c$}. Observe that the estimates \eqref{assumptions on NR-NR}, \eqref{eq: use it},  and the facts in Remark \ref{rem-iota} still hold here. Also, note that by Lemma~\ref{lem-ratio-J},
    \begin{align}\label{bound on ratio of A (NR-R)}
    \frac{A^{\sig_1}_k(\eta,\alpha)}{A^{\sig_1}_l(\xi,\beta)} \lesssim \frac{|l|^2\Big(1+|t-\frac{\xi}{l}|\Big)}{|\xi|} e^{c\lambda(t)|k-l,\eta-\xi,\alpha-\beta|^{s}}.
    \end{align}
Hence, combining \eqref{eq: use it},\eqref{bound on ratio of A (NR-R)}, using the fact $\fr{|\xi|}{|l|}\approx t$, our goal now is to control
    \[
    \frac{t}{l^2+(\xi-lt)^2+\beta^2}\frac{|l|^2\Big(1+|t-\frac{\xi}{l}|\Big)}{|\xi|} e^{c\lambda(t)|k-l,\eta-\xi,\alpha-\beta|^{s}}\lesssim \fr{e^{c\lambda(t)|k-l,\eta-\xi,\alpha-\beta|^{s}}}{|l|(1+|t-\fr{\xi}{l}|)}.
    \] 
    Let us ignore the exponential term appearing above, and  split the analysis into three time sub-intervals:
    \begin{enumerate}
     \item $1\leq t<1000 |\xi|^{\frac{1}{2}}$. Here, using this restriction and the fact $\fr{|\xi|}{|l|}\approx t$, one easily deduces that
     \[
    \frac{t}{l^2+(\xi-lt)^2+\beta^2}\frac{|l|^2\Big(1+|t-\frac{\xi}{l}|\Big)}{|\xi|}\lesssim \frac{|1|}{|l|} \lesssim \frac{t}{|\xi|}\lesssim \frac{1}{t}\lesssim \frac{|\xi|^s}{t^{1+2s}}.
    \]
     \item $1000 |\xi|^\fr12\leq t\leq \min\{2 |\xi|,2|\eta|\}$. There exists an $m$ such that $t \in {\rm{I}}_{l,\xi} \cap {\rm{I}}_{m,\eta} $. By Lemma~\ref{Scenarios}, we are led to the following scenarios:
     
     a) If $m=l$, then straightforward computation yields
         \[
         \frac{t}{l^2+(\xi-lt)^2+\beta^2}\frac{|l|^2\Big(1+|t-\frac{\xi}{l}|\Big)}{|\xi|}\lesssim \frac{1}{|l|}\frac{1}{1+|\frac{\xi}{l}-t|} \lesssim \frac{1}{\sqrt{1+|\frac{\xi}{l}-t|}}\frac{1}{\sqrt{(1+|\frac{\eta}{m}-t|)}}\JB{\eta-\xi}^\fr12.
        \]
     
     b) If $m \neq l$, then we perform our estimation in two sub-scenarios:
     \begin{itemize}
         \item If $|t-\frac{\xi}{l}|\gtrsim \frac{|\xi|}{|l|^2}$, then
         \[
          \frac{t}{l^2+(\xi-lt)^2+\beta^2}\frac{|l|^2\Big(1+|t- 
          \frac{\xi}{l}|\Big)}{|\xi|}\lesssim \frac{1}{|l|}\frac{1}{(1+|\frac{\xi}{l}-t|)}\lesssim \frac{|l|}{|\xi|}\approx \frac{1}{t}\lesssim \frac{|\xi|^s}{\JB{t}^{1+s}}.
         \]
         
         \item If $|\eta-\xi|\gtrsim \frac{|\eta|}{|m|}$, as before, since $\frac{|\eta|}{|m|}\approx \frac{|\xi|}{|m|}\approx \frac{|\xi|}{|l|}\approx t$, we then have
         \[
          \frac{t}{l^2+(\xi-lt)^2+\beta^2}\frac{|l|^2\Big(1+|t- 
          \frac{\xi}{l}|\Big)}{|\xi|}\lesssim \frac{1}{|l|}\frac{1}{(1+|\frac{\xi}{l}-t|)}\lesssim 1 \lesssim \frac{|\eta-\xi|^2}{\JB{t}^2}.
          \]
     \end{itemize} 
    \item  $2|\xi|\geq t\geq 2|\eta|$.  In this regime, $|t-\fr{\eta}{l}|\geq t-|\fr{\eta}{l}|\geq t-|\eta|\geq t/2$. Thus,
          \[
          \begin{aligned}
          \frac{t}{l^2+(\xi-lt)^2+\beta^2}\frac{|l|^2\Big(1+|t- 
          \frac{\xi}{l}|\Big)}{|\xi|}&\lesssim \frac{1}{1+|\frac{\xi}{l}-t|}= \frac{1}{1+|\frac{\eta}{l}-t|} \frac{1+|\frac{\eta}{l}-t|}{1+|\frac{\xi}{l}-t|}\\&\lesssim \frac{\JB{\eta-\xi}}{t} \lesssim \frac{|\xi|^s \JB{\eta-\xi}}{\JB{t}^{1+s}}.
          \end{aligned}
          \]
\end{enumerate}
This completes the proof of Lemma~\ref{est: NR-R}.
\end{proof}

Now, let us investigate the interactions between the resonant modes.
\begin{lemma}[Resonant-Resonant]\label{R R}
    Let $t\in {\rm{I}}_{k,\eta} \cap {\rm{I}}_{l,\xi}$ and $(k,\eta,\alpha), (l,\xi,\beta) \in \mathfrak{A}$. Then for $k \neq l,$
\[
\begin{aligned}
&\dfrac{(t|\beta|+|\xi|+|l,\beta|)|l,\beta|}{(l^2+(\xi-lt)^2+\beta^2)^2}\frac{A^{\sig_1}_k(\eta,\alpha)}{A^{\sig_1}_l(\xi,\beta)}\\
&\qquad \qquad \lesssim \bigg(\frac{1}{\JB{t}^2}+{\frac{|l,\xi,\beta|^\fr{s}{2} |k,\eta,\alpha|^\fr{s}{2}}{\JB{t}^{1+s}}} + \sqrt{\dfrac{\partial_t w_k(t,\iota_1)}{w_k(t,\iota_1)}}\sqrt{\dfrac{\partial_t w_l(t,\iota_2)}{w_l(t,\iota_2)}}\bigg)e^{\mathbf{c}\lambda(t)|k-l,\eta-\xi,\alpha-\beta|^{s}}
\end{aligned}
\]
     and for $k=l,$
    \begin{align*}
     &\frac{(l^2+\beta^2+|\xi-lt||\beta|)}{(l^2+(\xi-lt)^2+\beta^2)^2}\frac{A^{\sig_1}_k(\eta,\alpha)}{A^{\sig_1}_l(\xi,\beta)}\\
     &\qquad \qquad \lesssim  \bigg(\frac{1}{\JB{t}^2}+\frac{|l,\xi,\beta|^\fr{s}{2} |k,\eta,\alpha|^\fr{s}{2}}{\JB{t}^{1+s}} + \sqrt{\dfrac{\partial_t w_k(t,\iota_1)}{w_k(t,\iota_1)}}\sqrt{\dfrac{\partial_t w_l(t,\iota_2)}{w_l(t,\iota_2)}}\bigg)e^{\mathbf{c}\lambda(t)|\eta-\xi,\alpha-\beta|^{s}}.
     \end{align*}
\end{lemma}
\begin{proof}
As before, when $(l,\xi,\beta)\in\mathfrak{U}$, Lemma~\ref{l xi nonresonant} gives us the above inequality. Here we only focus on the case when $(l,\xi,\beta)\in\mathfrak{U}^c$. Hence, the estimates \eqref{assumptions on NR-NR}, \eqref{eq: use it}, and the facts in Remark \ref{rem-iota} hold here. 
Moreover, since $t\in {\rm{I}}_{k,\eta} \cap {\rm{I}}_{l,\xi},$ by \eqref{ratio-J}, we find that \eqref{ratio A nrnr} holds as well. 

If $k=l$, it suffices to estimate 
   \begin{align*}
    \frac{(l^2+\beta^2+|\xi-lt||\beta|)}{(l^2+(\xi-lt)^2+\beta^2)^2}\lesssim& \frac{1}{(l^2+(\xi-lt)^2+\beta^2)}\lesssim \frac{1}{\sqrt{1+|\frac{\xi}{l}-t|}}\frac{1}{\sqrt{1+|\frac{\eta}{k}-t|}}\JB{\eta-\xi}^\fr12\\
    \les&\sqrt{\fr{\pr_tw_k(t,\iota_1)}{w_k(t,\iota_1)}}\sqrt{\fr{\pr_tw_k(t,\iota_1)}{w_k(t,\iota_1)}}\la\eta-\xi\ra^{\fr12}.
  \end{align*}

If $k\ne l$, then we bound $\dfrac{(t|\beta|+|\xi|+|l,\beta|)|l,\beta|}{(l^2+(\xi-lt)^2+\beta^2)^2}$ from above by $\frac{t}{l^2+(\xi-lt)^2+\beta^2}$ as that in \eqref{eq: use it}, and the following two scenarios should be taken into consideration  according to   Lemma~\ref{Scenarios}:
\begin{enumerate}
    \item $|t-\frac{\eta}{k}|\gtrsim \frac{|\eta|}{k^2}$ and $|t-\frac{\xi}{l}|\gtrsim \frac{|\xi|}{l^2}$. Thanks to the lower bound of $|t-\fr{\xi}{l}|$ and the fact $\fr{|\xi|}{|l|}\approx t$, one deduces that
    \[
    \frac{t}{l^2+(\xi-lt)^2+\beta^2}\les \fr{|l|}{|\xi|}\approx \fr{1}{t}\lesssim \frac{|\xi|^s}{\JB{t}^{1+s}}.
    \]
    
    \item $|\xi-\eta|\gtrsim \frac{|\eta|}{|l|}$. Note that $\frac{|\eta|}{|l|}\approx \frac{|\xi|}{|l|}\approx t$. Therefore, we can infer that
    \[
    \frac{t}{l^2+(\xi-lt)^2+\beta^2} \lesssim t\lesssim \frac{|\xi-\eta|^3}{\JB{t}^2}.
    \]
\end{enumerate}
The proof of Lemma \ref{R R} is completed.
\end{proof}

Finally, we give a lemma that will be used to treat the reaction term stemming from $u_0$ when the derivative in  $z$ direction dominates. 
\begin{lemma}\label{lem: zero-mode-Reaction}
    Let  $t>10$,  $100| \beta|> |\xi|$, and 
    \begin{equation}\label{frequency1}
|k,\eta-\xi, \al-\beta|\le \fr{3}{16}|\xi,\beta|.
\end{equation} 
Then
     \begin{align*}
     &t^{-3} \dfrac{(t|\beta|^2+|\xi||\beta|)}{\xi^2+\beta^2}\fr{1}{B(\xi,\beta)}\frac{A^{\sig_1}_k(t,\eta,\alpha)}{A^{\sig_1}_0(t,\xi,\beta)} \\
     \lesssim& \bigg(\frac{|\beta|}{\xi^2+\beta^2}t^{-2} +\frac{|\beta|^\fr{s}{2}}{\JB{t}^{\fr{1+s}{2}}} t^{-\fr32}\sqrt{\fr{\partial_tw_k(t,\iota(k,\eta,\alpha))}{w_k(t,\iota(k,\eta,\alpha))}} \bigg)e^{\mathbf{c}\lambda(t)|k,\eta-\xi,\alpha-\beta|^{s}}.
     \end{align*}
\end{lemma}

\begin{proof}
By \eqref{frequency1} and \eqref{triangle1}, we have
\begin{align}\label{ratio-A-1}
\fr{A^{\sig_1}_k(t,\eta,\al)}{ A^{\sig_1}_0(t,\xi,\beta)}=&\fr{e^{\lm(t)| k,\eta,\al|^s}| k,\eta,\al|^{\sig_1}{\rm J}_k(t,\eta,\al)}{e^{\lm(t)| \xi,\beta|^s}|\xi,\beta|^{\sig_1}{\rm J}_0(t,\xi,\beta)}
\les e^{c\lm(t)| k, \eta-\xi,\al-\beta|^s}\fr{{\rm J}_k(t,\eta,\al)}{{\rm J}_0(t,\xi,\beta)}.
\end{align}
We split our estimates into two cases:

    \textbf{Case 1:} Let $t \in \mathrm{I}_{k,\iota_1}$. By \eqref{ratio-J}, now we have
    \[
    \frac{{\rm J}_k(\eta,\alpha)}{\rm{J}_0(\xi,\beta)} \lesssim \frac{|\iota_1|}{k^2(1+|t-\fr{\iota_1}{k}|)}e^{2\mu|k,\eta-\xi,\alpha-\beta|^\fr12}.
    \]
    On the other hand, the restrictions $|\xi|<100|\beta|$ and \eqref{frequency1} show that $|\iota_1|\approx|k,\eta,\al|\approx |\xi,\beta|\approx|\beta|$. This, together with the fact $t\in{I}_{k,\iota_1}$, implies that
    $t\approx \fr{|\iota_1|}{|k|}\approx \frac{|\beta|}{|k|}$. 
    Therefore,
    \[
    \begin{aligned}
     &t^{-3} \dfrac{(t|\beta|^2+|\xi||\beta|)}{\xi^2+\beta^2}\fr{1}{B(\xi,\beta)}\frac{A^{\sigma_1}_k(\eta,\alpha)}{A^{\sigma_1}_0(\xi,\beta)}\\
      \lesssim& \fr{1}{t^{\fr12}}\fr{e^{{\bf c}\lm|k,\eta-\xi,\alpha-\beta|^s}}{1+|t-\fr{\iota_1}{k}|}t^{-\fr32}\lesssim \fr{|\beta|^\fr{s}{2}t^{-\fr32}}{t^{\fr{1+s}{2}}(1+|t-\fr{\iota_1}{k}|)}e^{{\bf c}\lambda|k,\eta-\xi,\alpha-\beta|^s}\\
     \lesssim& \frac{|\beta|^\fr{s}{2}}{\JB{t}^{\fr{1+s}{2}}}t^{-\fr32}\sqrt{\fr{\partial_tw_k(t,\iota(k,\eta,\alpha))}{w_k(t,\iota(k,\eta,\alpha))}} e^{\mathbf{c}\lambda(t)|k,\eta-\xi,\alpha-\beta|^{s}},
     \end{aligned}
     \]
     for some constants $c<{\bf c}<1$.

     \textbf{Case 2:} Let $t \notin \mathrm{I}_{k,\iota_1}$. By  \eqref{ratio-A-1} and Lemma \ref{lem-ratio-J},  one easily deduces that 
     \[
    \begin{aligned}
     t^{-3} \dfrac{(t|\beta|^2+|\xi||\beta|)}{\xi^2+\beta^2}\fr{1}{B(\xi,\beta)}\frac{A^{\sigma_1}_k(\eta,\alpha)}{A^{\sigma_1}_0(\xi,\beta)} \lesssim \fr{1}{t^{2}}\fr{|\beta|}{\xi^2+\beta^2} e^{\mathbf{c}\lambda(t)|k,\eta-\xi,\alpha-\beta|^{s}}.
     \end{aligned}
     \]
    This completes the proof of Lemma \ref{lem: zero-mode-Reaction}.
\end{proof}

\section{Main Energy Estimates and Bootstrap Proposition}
This section is devoted to presenting the main energy used throughout the work as well as to proving some bootstrap assumptions that play a crucial role here.
For fixed $\sigma_1>\sigma_2>\cdots>\sigma_6>\sigma_7\geq 2$ with
$\sigma_i-\sigma_{i+1}\ge 30$, for $ i=1,2,\cdots, 6$, we define the norms
\begin{equation}\label{main norm}
\begin{aligned}
&\|f(t)\|^2_{\sig_1}=\sum_{k,\al\in\Z}\int_\eta \left|A^{\sig_1}_k(t,\eta,\al)\widehat{f}_k(t,\eta,\al)\right|^2d\eta,\\
&\|Bf_0(t)\|^2_{\sigma-2}=\sum_{\al\in \Z}\int_\eta \left|A^{\sig_1}_0(t,\eta,\al)B(\eta,\alpha)\widehat{f}_0(t,\eta,\al)\right|^2d\eta
\end{aligned}
\end{equation}

For later use, we define the following terms for $t> 10$:
\begin{align*}
&{\rm CK}^{\sig_1}_{\lambda,\theta}(t)=-\dot{\lambda}(t)\norm{t^{-\fr32}|\nabla|^\fr{s}{2}\theta}^2_{\sigma_1},\\&
{\rm CK}^{\sig_1}_{w,\theta}(t)=t^{-3}\sum_{k,\alpha}\int_{\eta}\frac{\partial_t w_k(t,\iota(k,\eta,\al))}{w_k(t,\iota(k,\eta,\al))}  \big|A^{\sig_1}_k(t,\eta,\al)   \widehat{\theta}\big|^2 \;d\eta,\\&
{\rm CK}^{\sig_1-2}_{\lambda, B\theta_0}(t)=-\dot{\lambda}(t)\norm{|\nabla|^\fr{s}{2}B\theta_0}^2_{\sigma_1-2},\\&
{\rm CK}^{\sig_1-2}_{w,B\theta_0}(t)=\sum_{\alpha}\int_{\eta}\frac{\partial_t w_0(t,\iota(0,\eta,\al))}{w_0(t,\iota(0,\eta,\al))}\big|A^{\sig_1-2}_0(t,\eta,\al) B(\eta,\alpha)\widehat{\theta_0}\big|^2 \;d\eta.
\end{align*}

The norms associated with $\sigma_2,...,\sigma_7$ can be found in the bootstrap hypotheses below. 
\begin{rem}
    Note that for all $k$ and $t\in \left[\frac{2|\iota(k,\eta,\alpha)|}{2\E\big(\sqrt{|\iota(k,\eta,\alpha)|}\big)+1},2|\iota(k,\eta,\alpha)| \right]$, there is a unique $\ell$ with $\iota(k,\eta,\alpha)\ell >0$, such that $t\in \mathrm{I}_{\ell, \iota(k,\eta,\alpha)}$ and 
\begin{align*}
    \frac{\partial_t w_k(t,\iota(k,\eta,\al))}{w_k(t,\iota(k,\eta,\al))}
    \approx \frac{1}{1+\left|t-\fr{\iota(k,\eta,\al)}{\ell}\right|}. 
\end{align*}
\end{rem}

 Without loss of generality, assume that $t>10$. We can assume this because well-posedness theory holds for the three-dimensional Stokes-transport in Gevrey Spaces (even less-smoother spaces). Hence, by choosing initial data to be sufficiently small, we may ignore the time interval $[0,10]$). This is precisely the content of Lemma~\ref{local wp}.

Next, we need to show that the main energy is uniformly bounded for all times if the constant $\epsilon$ is chosen sufficiently small. To do that, we start by assuming a number of quantities, known as the bootstrap hypotheses. More precisely, we assume that for $t> 10$,
\begin{enumerate}[start=1,label={\bfseries B\arabic*:}]
    
\item  {\bf High norm bounds:}
 \begin{align}\label{En-1}
&\| t^{-\frac32}\theta(t)\|_{\sigma_1}\leq 4C \epsilon.\\
\label{En0-1}
&\left\|B\theta_0(t)\right\|_{\sigma_1-2}\leq 4C \epsilon.
\end{align}

\item {\bf Intermediate norm bounds:}
\begin{align}\label{En0-2}
&\left(\sum_{\alpha\neq 0}\int_{\mathbb{R}}\left|\sup_{t\in[10,T^{*}]}\big|\JB{\eta,\alpha}^{\sigma_2}e^{\lambda(t)|\eta,\alpha|^s}{\la t\ra^{\fr{3}{2}}}\widehat{\theta}_0(t,\eta,\alpha)\big|\right|^2d\eta\right)^{\fr{1}{2}}\leq {4C\eps}.\\
\label{En-3}
&\|t^{-\fr12}\theta(t)\|_{\mathcal{G}^{\lm,\sig_3}}\leq 4C  \eps.\\
\label{En0-4}
& \left(\sum_{\alpha\neq 0}\int_{\mathbb{R}}\left|\sup_{t\in[10,T^{*}]}\big|\JB{\eta,\alpha}^{\sigma_4}e^{\lambda(t)|\eta,\alpha|^s}{\la t\ra^{\fr52}}\widehat{\theta}_0(t,\eta,\alpha)\big|\right|^2d\eta\right)^{\fr{1}{2}}\leq {4C\eps}.
\end{align}

\item {\bf Lower norm bounds:}
\begin{align}\label{En-5}
&\|\theta(t)\|_{\mathcal{G}^{\lambda,\sig_5}}\leq 4C \eps.\\
\label{En-6}
&\left(\sum_{\alpha\neq 0}\int_{\mathbb{R}}\left|\sup_{t\in[10,T^{*}]}\big|\JB{\eta,\alpha}^{\sigma_6}e^{\lambda(t)|\eta,\alpha|^s}{\la t\ra^{3}}\widehat{\theta}_0(t,\eta,\alpha)\big|\right|^2d\eta\right)^{\fr{1}{2}}\leq {4C\eps}.\\
\label{En-7}
&\sup_{t\in [10, T^*]}\sup_{\eta\in \mathbb{R}}\left|e^{\lambda(t)|\eta|^s}\JB{\eta}^{\sigma_7}\widehat{\mathbb{P}^z_{0}\theta_0}(t,\eta)\right|\leq 4C\epsilon. 
\end{align}

\item ${\rm CK}$ and Damping integral estimates:
\[
\int_{10}^t \bigg[{\rm CK}^{\sig_1}_{\lambda,\theta}+{\rm CK}^{\sig_1}_{w,\theta}+{\rm CK}^{\sig_1-2}_{\lambda,B\theta_0}+{\rm CK}^{\sig_1-2}_{w,B\theta_0}+\|t^{-2}\theta\|_{\sigma_1}^2+\left\|\pr_z{\Dl_{yz}^{-1}} B\theta_0 \right\|_{\sig_1-2}^2\bigg]\;ds\leq 4C\epsilon^2.
\]
\end{enumerate}

\begin{rem}
One can see from \eqref{eq-theta0'} that  $\theta_0$ can take one less derivative than $\theta$ due to the appearance of $\nb_{xyz}\theta_{\ne}$ in the non-linearity. Also note that $|\alpha|\leq B(\alpha, \eta)\leq |\alpha|+|\eta|$ can be formally regarded as one derivative. That is why we assign $\la \eta,\al\ra^{\sig_1-2}$ instead of $\la \eta, \al\ra^{\sig_1}$, as the Sobolev correction in \eqref{En0-1}.
\end{rem}
\begin{rem}\label{rmk: Minkowski}
    The Minkowski inequality implies that $L^{2}L^{\infty}\subset L^{\infty}L^2$, which gives us that for all $\sigma\ge0, m\ge0$,
    \begin{align*}
    \sup_{t\in [10,T^{*}]}\|\JB{t}^m\mathbb{P}_{\neq}^z\theta_0\|_{\mathcal{G}^{\lambda,\sigma}}
    \leq \left(\sum_{\alpha\neq 0}\int_{\mathbb{R}}\left|\sup_{t\in[10,T^{*}]}\big|\JB{\eta,\alpha}^{\sigma}e^{\lambda(t)|\eta,\alpha|^s}{\la t\ra^{m}}\widehat{\theta}_0(t,\eta,\alpha)\big|\right|^2d\eta\right)^{\fr{1}{2}}.
    \end{align*}
\end{rem}
\subsection{Bootstrap}
Here, we present the bootstrap proposition which is an important ingredient in proving the main theorem. 
\begin{prop}[Bootstrap]\label{bootstrap} There exists $\epsilon_0 \in(0,\frac{1}{2})$ which depends on $\lambda',\lambda_0,s,$ and $\sigma$ such that provided $\epsilon < \epsilon_0$ and the bootstrap hypotheses \textup{\textbf{B1}}-\textup{\textbf{B4}} hold for all time $t\in[10,T^*]$, then for all $t\in [10,T^*]$, the estimates from \textup{\textbf{B1}}-\textup{\textbf{B4}} can be improved by replacing the constant $4$ by $2$. 
\end{prop}
The rest of the paper is to prove Proposition \ref{bootstrap}. The improvements of \eqref{En-1} and \eqref{En0-1} follow directly from the classical energy method (see \eqref{e-theta}) with well-designed time-dependent Fourier multipliers and Propositions \ref{Prop: T_0R_0R}, \ref{Prop: T_neqR_neqR}, \ref{prop: sor}, and \ref{prop: tildeT-R-R}. 
The improvement of \eqref{En0-2} can be found in section \ref{sec-En02}. The improvements of  \eqref{En0-4} and \eqref{En-6} are similar and proved in sections \ref{sec-im-En0-4} and \ref{sec:im-En6} separately. The improvement of \eqref{En-3} and \eqref{En-5} are similar and can be found in sections \ref{sec-im-En-3} and \ref{sec-im-En-5}. The improvement of \eqref{En-7} is proved in section \ref{sec:im-En7}. 

Let us now focus on the key estimates, namely the improvement of \eqref{En-1}. The first natural step is to compute the time derivative of the top energy. 
To this end, from equation \eqref{eq-theta'}, we arrive at
\begin{align}\label{e-theta}
\nn&\fr12\fr{d}{dt}\|t^{-\fr32}\theta(t)\|_{\sig_1}^2
+\left\|t^{-\fr32}\sqrt{-\Dl_{xz}}\Delta^{-1}_L A^{\sig_1}\theta\right\|_{L^2}^2
+\fr{3}{2}\|t^{-2}\theta\|_{\sigma_1}^2
+{\rm CK}^{\sig_1}_{\lambda,\theta}(t)+{\rm CK}^{\sig_1}_{w,\theta}(t)\\
\nn&=-t^{-3}\left\la A^{\sig_1}(u\cdot\nb \theta), A^{\sig_1}\theta\right\ra\\
&=\fr12\int_{\mathbb{T}\times\mathbb{R}\times\mathbb{T}} t^{-3}\nb\cdot u |A^{\sig_1}\theta|^2dxdydz+ t^{-3} \left\la u\cdot\nb A^{\sig_1}\theta-A^{\sig_1}(u\cdot\nb \theta) , A^{\sig_1}\theta \right\ra.
\end{align}

The term on the right-hand side of \eqref{e-theta} consists of two parts: the divergence and commutator terms. Recalling \eqref{exp-u},  the divergence term can be given as follows
\[
\nb\cdot u=\nb \cdot u_{\ne}.
\]
This enables us to use the inequality \eqref{u_ne-decay} to bound
\begin{align}
\fr12\int_{\mathbb{T}\times\mathbb{R}\times\mathbb{T}}t^{-3}\nb\cdot u |A^{\sig_1}\theta|^2dxdydz\nn\les& t^{-3}\|u_{\ne}\|_{H^3}\|\theta\|_{\sig_1}^2\les \fr{\|\theta_{\ne}\|_{H^9}}{\la t\ra^3}\|t^{-\frac{3}{2}}\theta\|_{\sig_1}^2.
\end{align}

To bound the commutator term of \eqref{e-theta}, we split it into two parts:
\[
t^{-3}\left\la u\cdot\nb A^{\sig_1}\theta-A^{\sig_1}(u\cdot\nb \theta) , A^{\sig_1}\theta \right\ra={\bf com}_{\sig_1;0}+{\bf com}_{\sig_1;\ne},
\]
where
\begin{equation}\label{Decomposition of Commutator}
\begin{aligned}
{\bf com}_{\sig_1;0}=& t^{-3}\left\la u_0\cdot\nb A^{\sig_1}\theta-A^{\sig_1}(u_0\cdot\nb \theta) , A^{\sig_1}\theta \right\ra,\\
{\bf com}_{\sig_1;\ne}=& t^{-3}\left\la u_{\ne}\cdot\nb A^{\sig_1}\theta-A^{\sig_1}(u_{\ne}\cdot\nb \theta) , A^{\sig_1}\theta \right\ra.
\end{aligned}
\end{equation}
Via a paraproduct decomposition, each commutator term above can be expressed in terms of three main contributions: transport, reaction, and reminder. We write such decomposition for the non-zero mode interaction, the one for the zero mode can be done similarly. More precisely, ${\bf com}_{\sig_1;\ne}$ can be decomposed as follows
\[
t^{-3}\left\la u_{\ne}\cdot\nb A^{\sig_1}\theta-A^{\sig_1}(u_{\ne}\cdot\nb \theta) , A^{\sig_1}\theta \right\ra=\frac{1}{2\pi}\sum_{\rm{N}\geq 8} {\T_{{\neq}_\mathrm{N}}} +\frac{1}{2\pi}\sum_{\rm{N}\geq 8} \R_{{\neq}_\mathrm{N}} +\frac{1}{2\pi} \mathcal{R}_{\neq},
\]
where
\begin{equation}\label{paraproduct decomposition nonzero mode}
\begin{aligned}
    &\T_{{\neq}_\mathrm{N}}=2\pi t^{-3} \left\la u_{{\ne}_{<\N/8}}\cdot\nb A^{\sig_1}\theta_\mathrm{N}-A^{\sig_1}(u_{{\ne}_{<\N/8}}\cdot\nb \theta_\mathrm{N}) , A^{\sig_1}\theta \right\ra,\\&
    \R_{{\neq}_\mathrm{N}}=2\pi t^{-3}\left\la u_{{\ne}_\mathrm{N}}\cdot\nb A^{\sig_1}\theta_{<\N/8}-A^{\sig_1}(u_{{\ne}_\mathrm{N}}\cdot\nb \theta_{<\N/8}) , A^{\sig_1}\theta \right\ra, \\&
    \mathcal{R}_{\neq}= 2\pi t^{-3}\sum_{\N\in \mathbb{D}}\sum_{\fr{\rm{N}}{8}\leq \rm{N}' \leq 8\rm{N}} \left\la u_{{\ne}_\mathrm{N}}\cdot\nb A^{\sig_1}\theta_{\N'}-A^{\sig_1}(u_{{\ne}_\mathrm{N}}\cdot\nb \theta_{\N'}) , A^{\sig_1}\theta \right\ra .
\end{aligned}
\end{equation}
We use the dyadic domain $\N \in \mathbb{D}=\{\frac{1}{2}, 1,2,4,...,2^m,...\}$ and denote $f_{\N}$ as the Littlewood--Paley projection onto the $\N$-th frequency and $f_{<\N}$ as the Littlewood--Paley projection onto the frequencies less than $\N$. 

By simply replacing $\neq$ by $0$ in the above decomposition, one can obtain the associated paraproduct decomposition for ${\bf com}_{\sig_1;0}$. This process results in the appearance of ${\bf T}_{0_\N}, {\bf R}_{0_\N}, \mathcal{R}_0$.

The following propositions record corresponding estimates of ${\bf T}_{0_\N}, {\bf R}_{0_\N}, \mathcal{R}_0$ and ${\bf T}_{\neq_\N}, {\bf R}_{\neq_\N}, \mathcal{R}_{\neq}$.

\begin{prop}\label{Prop: T_0R_0R}
    Under the bootstrap hypotheses, it holds that
    \begin{align*}
    &\sum_{\N\geq 8}|{\bf R}_{0_\N}|\lesssim \epsilon \bigg(\fr{\eps^2}{\JB{t}^2}+{\rm CK}^{\sig_1}_{\lambda,\theta}+{\rm CK}^{\sig_1}_{w,\theta}+{\rm CK}^{\sig_1-2}_{\lambda,B\theta_0}+\|t^{-2}\theta\|_{\sigma_1}^2
    +\left\|\pr_z{\Dl_{yz}^{-1}} B\theta_0 \right\|_{\sig_1-2}^2\bigg),\\
        &\sum_{\N\geq 8}|{\bf T}_{0_\N}|\lesssim \fr{\epsilon^3}{\JB{t}^2}+\epsilon {\rm CK}^{\sig_1}_{\lambda,\theta},\quad\text{and}\quad 
        |\mathcal{R}_0|\lesssim  \fr{\epsilon^3}{\JB{t}^2}.
    \end{align*}
\end{prop}

\begin{prop}\label{Prop: T_neqR_neqR}
    Under the bootstrap hypotheses, it holds that
     \begin{align*}
    &\sum_{\N\geq 8}|{\bf R}_{\neq_\N}|\lesssim \epsilon \left(\fr{\eps^2}{\JB{t}^2}+{\rm CK}^{\sig_1}_{\lambda,\theta}
    +{\rm CK}^{\sig_1}_{w,\theta}\right),\\
        &\sum_{\N\geq 8}|{\bf T}_{\neq_\N}|\lesssim \fr{\epsilon^3}{\JB{t}^2}+\epsilon {\rm CK}^{\sig_1}_{\lambda,\theta},\quad\text{and}\quad 
        |\mathcal{R}_{\neq}|\lesssim  \fr{\epsilon^3}{\JB{t}^2}.
    \end{align*}
\end{prop}
Both Propositions~\ref{Prop: T_0R_0R} and~\ref{Prop: T_neqR_neqR} are proved in section \ref{Sec: main est1}. 

Similarly, in order to obtain the improvement of \eqref{En0-1}, we study the time derivative of $\left\|  B\theta_0(t)\right\|_{\sigma_1-2}^2$, from \eqref{eq-theta0'}, using the fact $\nb_{yz}\cdot\tl u_0=0$, we find that 
\begin{equation}\label{commutator and source}
\begin{aligned}
&\fr12\fr{d}{dt} \left\|B\theta_0(t)\right\|_{\sigma_1-2}^2
+\left\|{\pr_z}{\Dl_{yz}^{-1}} B\theta_0 \right\|_{\sig_1-2}^2
+{{\rm CK}^{\sig_1-2}_{\lambda, B\theta_0}(t)}+{{\rm CK}^{\sig_1-2}_{w,B\theta_0}(t)}\\& =\left\la \tl u_0\cdot\nb_{yz}A^{\sig_1-2}_0 B\theta_0 -A^{\sig_1-2}_0B\left(\tl u_0\cdot\nb_{yz}\theta_0 \right), A^{\sig_1-2}_0B\theta_0 \right\ra\\
& \qquad -\left\la A^{\sig_1-2}_0 B\left( u_{\ne}\cdot\nb_{xyz}\theta_{\ne} \right)_0, A^{\sig_1-2}_0B\theta_0 \right\ra\\&={  {\bf com}_{\sig_1-1}}+{\bf S}_{\sig_1-1}.
\end{aligned}
\end{equation}
Again, by applying the paraproduct decomposition, we get 
\[
{\bf com}_{\sig_1-1}=\frac{1}{2\pi}\sum_{\rm{N}\geq 8} {\tilde{\T}_{\N}} +\frac{1}{2\pi}\sum_{\rm{N}\geq 8} \tilde{\R}_{\N} +\frac{1}{2\pi} \tilde{\mathcal{R}},
\]
where
\begin{equation}\label{paraproduct}
\begin{aligned}
    &\tilde{\T}_{\N}=2\pi  \left\la \tl{u}_{{0}_{<\N/8}}\cdot\nb A_0^{\sig_1-2}B\theta_{0_\N}-A_0^{\sig_1-2}B(\tl{u}_{{0}_{<\N/8}}\cdot\nb \theta_{0_\N}) , A_0^{\sig_1-2}B\theta_0 \right\ra,\\&
    \tilde{\R}_{\N}=2\pi \left\la \tl{u}_{{0}_{\N}}\cdot\nb A_0^{\sig_1-2}B\theta_{0_{<\N/8}}-A_0^{\sig_1-2}B(\tl{u}_{{0}_{\N}}\cdot\nb \theta_{0_{<\N/8}}) , A_0^{\sig_1-2}B\theta_0 \right\ra, \\&
    \tilde{\mathcal{R}}= 2\pi \sum_{\N\in \mathbb{D}}\sum_{\fr{\rm{N}}{8}\leq \rm{N}' \leq 8\rm{N}} \left\la \tl{u}_{{0}_{\N}}\cdot\nb A_0^{\sig_1-2}B\theta_{0_{\N'}}-A_0^{\sig_1-2}B(\tl{u}_{{0}_{\N}}\cdot\nb \theta_{0_{\N'}}) , A_0^{\sig_1-2}B\theta_0 \right\ra .
\end{aligned}
\end{equation}

We have the following propositions for the estimate of ${\bf S}_{\sig_1-1}$ and ${\tilde{\T}_\mathrm{N}}, \tilde{\R}_\mathrm{N}, \tilde{\mathcal{R}}$. 
\begin{prop}\label{prop: sor}
    Under the bootstrap hypotheses, it holds that
    \begin{align*}
     |{\bf S}_{\sig_1-1}|\lesssim \epsilon \bigg(\fr{\epsilon^2}{\JB{t}^{3/2}}+{\rm CK}^{\sig_1}_{\lambda,\theta}+{\rm CK}^{\sig_1}_{w,\theta}+{\rm CK}^{\sig_1-2}_{\lambda,B\theta_0}+{\rm CK}^{\sig_1-2}_{w,B\theta_0}\bigg)
    \end{align*}
\end{prop}
\begin{prop}\label{prop: tildeT-R-R}
    Under the bootstrap hypotheses, it holds that
    \begin{align*}
        &\sum_{\N\geq 8}|{\tilde{\T}_{\N}}|\lesssim \fr{\epsilon^3}{\JB{t}^2}+\epsilon {\rm CK}^{\sig_1-2}_{\lambda,B\theta_0},\\
        &\sum_{\N\geq 8}|{\tilde{\R}_{\N}}|\lesssim  \fr{\epsilon^3}{\JB{t}^2}
        +\epsilon \left\|\pr_z{\Dl_{yz}^{-1}} B\theta_0 \right\|_{\sig_1-2}^2,\quad\text{and}\quad 
        |{\tilde{\mathcal{R}}}|\lesssim \fr{\epsilon^3}{\JB{t}^2}.
    \end{align*}
\end{prop}
Both Propositions~\ref{prop: sor} and~\ref{prop: tildeT-R-R} are proved in section \ref{sec: main-2}. 

\subsection{Proof of Theorem \ref{Thm: main}}
We complete the proof of Theorem \ref{Thm: main} by proving the scattering result. Let us present the proof in the new coordinate system.  We start by integrating in time the first equation in \eqref{eq-theta'}, since $\theta$ is uniformly bounded and both $\Delta_{xz}\Delta^{-2}_L \theta$ and $u$ are time-integrable. We define $\theta_\infty$ as follows
\[
\theta_\infty=\theta(10)- \int_{10}^\infty \left (u \cdot \nabla\theta +\Delta_{xz}\Delta^{-2}_L \theta \right) dt.
\]
Via the algebra property of Gevrey space for $\lambda_\infty<\lambda_{\rm{in}}$, Minkowski inequality, and Proposition \ref{bootstrap}, we can infer that 
\begin{align*}
&\norm{\theta(t)-\theta{_\infty}}_{\mathcal{G}^{\lambda_\infty}}= \norm{\int_{t}^\infty u\cdot \nabla\theta+ \Delta_{xz}\Delta^{-2}_L \theta}_{\mathcal{G}^{\lambda_\infty}}d\tau\\
&\lesssim \int_{t}^\infty \norm{u}_{\mathcal{G}^{\lambda_\infty}}\norm{\nabla\theta}_{\mathcal{G}^{\lambda_\infty}}d\tau+\int_{t}^\infty \norm{\Delta_{xz}\Delta^{-2}_L \theta}_{\mathcal{G}^{\lambda_\infty}}d\tau\\&
\lesssim  \int_{t}^\infty \norm{u_0}_{\mathcal{G}^{\lambda_\infty}}\left(\norm{\nabla\theta_{\ne}}_{\mathcal{G}^{\lambda_\infty}}+\norm{\nabla_{yz}\theta_{0}}_{\mathcal{G}^{\lambda_\infty}}\right)+\norm{u_{\ne}}_{\mathcal{G}^{\lambda_\infty}}\left(\norm{\nabla\theta_{\ne}}_{\mathcal{G}^{\lambda_\infty}}+\norm{\nabla_{yz}\theta_{0}}_{\mathcal{G}^{\lambda_\infty}}\right)d\tau \\&\qquad+\int_{t}^\infty \norm{\Delta_{xz}\Delta^{-2}_L \theta_{\ne}+\partial_{zz}\Delta^{-2}_{y,z} \theta_{0}}_{\mathcal{G}^{\lambda_\infty}}d\tau
\\&\lesssim\epsilon \int_{t}^\infty \frac{\epsilon}{\JB{\tau}^2}d\tau+\int_t^\infty \frac{\epsilon}{\tau^3}d\tau\leq C\left(\frac{\epsilon^2}{t}+\frac{\epsilon}{t^2}\right).
\end{align*}
The other decay estimates in Theorem \ref{Thm: main} follow directly from Proposition \ref{bootstrap}, \eqref{eq:V-bd}, and the fact that $\|V_{\neq}^{(j)}\|_{L^2}=\|U_{\neq}^{(j)}\|_{L^2}$ for $j=1,2,3$.


\section{Main energy estimates}\label{Sec: main est1}
Now let us investigate the evolution of $\|\theta(t)\|_{\sig_1}^2$. We first consider the zero-mode equation given by  ${\bf com}_{\sigma_1;0}$. In light of \eqref{paraproduct decomposition nonzero mode} (with $\neq$ replaced by $0$), 
we write
\begin{align}
\T_{{0}_\N}\nn=& t^{-3}i\sum_{k,\al,\beta\in\Z}\int_{\eta,\xi} \left(1-\fr{A_k^{\sig_1}(t,\eta,\al)}{A_k^{\sig_1}(t,\xi,\beta)}\right)\widehat{u}_{0}(t,\eta-\xi,\al-\beta)_{<\N/8}\cdot (k,\xi,\beta)\\
\nn&\qquad \qquad \qquad \qquad \times \left(\widehat{A^{\sig_1}\theta}_k(t,\xi,\beta)_\N\right) \left(\overline{\widehat{A^{\sig_1}\theta}}_k(t,\eta,\al)\right) d\xi d\eta,
\end{align}
and
\begin{align}
\R_{0_{\N}}\nn=&-t^{-3} i\sum_{k,\al,\beta\in\Z}\int_{\eta,\xi} A_k^{\sig_1}(t,\eta,\al)\widehat{u}_0(t,\xi, \beta)_\N\cdot (k,\eta-\xi,\al-\beta)\\
\nn&\qquad \qquad \qquad \qquad \times \widehat{\theta}_k(t,\eta-\xi,\al-\beta)_{\N/8} \left(\overline{\widehat{A^{\sig_1}\theta}}_k(t,\eta,\al)\right) d\xi d\eta\\
\nn&+ t^{-3} i\sum_{k,\al,\beta\in\Z}\int_{\eta,\xi} \widehat{u}_0(t, \xi,\beta)_\N \cdot (k,\eta-\xi,\al-\beta)\\
\nn&\qquad \qquad \qquad \qquad \times \left(\widehat{A^{\sig_1}\theta}_k(t,\eta-\xi,\al-\beta)_{\N/8}\right) \left(\overline{\widehat{A^{\sig_1}\theta}}_k(t,\eta,\al)\right) d\xi d\eta\\
\nn=&\R^{(1)}_{0_{\N}}+\R^{(2)}_{0_{\N}}.
\end{align}
Before we move on to discuss estimates for the reaction terms above, we display the following useful inequality to pass the norm from $\sigma_1$ to $\sigma_1-2$. Recalling \eqref{exp-u}, it is true that
\begin{equation}
\begin{aligned}\label{up-u0-1}
\left|\widehat{A^{\sig_1}u}_0(t,\xi, \beta) \right|
\les&A_0^{\sig_1-2}(t,\xi,\beta)\fr{\la t\ra|\beta|^2+|\xi\beta|}{\xi^2+\beta^2}\big|\widehat{\theta}_0(t,\xi,\beta)\big|\\
\les&A_0^{\sig_1-2}(t,\xi,\beta)\left(\fr{\la t\ra|\beta|}{\xi^2+\beta^2}+1 \right)|\beta|\big|\widehat{\theta}_0(t,\xi,\beta)\big|.
\end{aligned}
\end{equation}
Next, we then investigate the reaction terms specified above.
\subsection{Reaction term (zero-mode interactions)}  On the support of the integrand of $\R_{0_{\N}}$, via the paraproduct decomposition, we have the following frequency localizations:
\[
\fr{\N}{2}\le|\xi,\beta|\le\fr{3\N}{2},\quad |k,\eta-\xi,\al-\beta|\le\fr{3}{4}\fr{\N}{8}.
\]
Consequently, \eqref{frequency1} holds.

\subsubsection{Treatment of $\R^{(1)}_{0_{\N}}$} 
\label{Nonresonant}

We first perform the following splitting:
\begin{align*}
\R^{(1)}_{0_{\N}}=&- it^{-3}\sum_{k,\al,\beta\in\Z}\int_{\eta,\xi} \big[{\mathds{1}}_{|\xi|< 100|\beta|}+{\mathds{1}}_{|\xi|>100|\beta|}\big]\fr{A_k^{\sig_1}(t,\eta,\al)}{A_0^{\sig_1}(t,\xi,\beta)}(\widehat{A^{\sig_1}u}_0)(t,\xi, \beta)_\N\cdot (k,\eta-\xi,\al-\beta)\\
&\qquad \qquad \qquad \qquad  \times \widehat{\theta}_k(t,\eta-\xi,\al-\beta)_{<\N/8} \Big(\overline{\widehat{A^{\sig_1}\theta}}_k(t,\eta,\al)\Big)\; d\xi d\eta\\
=&\R^{(1),z}_{0_{\N}}+\R^{(1),y}_{0_{\N}}.
\end{align*}
The first term in the summation above can be estimated via Lemma~\ref{lem: zero-mode-Reaction} along with the first inequality in \eqref{up-u0-1}
which then yields
\begin{align*}
   |\R^{(1),z}_{0_{\N}}|\lesssim \epsilon\|\pr_z\Dl_{yz}^{-1}B\theta_{0_{\N}}\|_{\sig_1-2}\|t^{-2}\theta_{\sim \N}\|_{\sig_1}
   +\epsilon\left\|\fr{|\nabla|^{\fr{s}{2}}}{\JB{t}^{\fr{s+1}{2}}}B\theta_{0_{\N}}\right\|_{\sig_1-2}\left\|\sqrt{\fr{\partial_tw}{w}}\fr{\theta_{\sim \N}}{t^{\fr{3}{2}}}\right\|_{\sig_1}.
\end{align*}

Before turning to the estimates of $\R^{(1),y}_{0_{\N}}$, we would like to remark that since now $l=0$ and \eqref{frequency1} holds, if $|\xi|>100|\beta|$, by elementary calculation as in Remark \ref{rem-iota}, we obtain that
\[
\fr{1297}{1600}|\xi|\le|\eta|\le\fr{1903}{1600}|\xi|,\quad |k|\le\fr{303}{1297}|\eta|,\quad |\al|\le\fr{319}{1297}|\eta|,
\]
and hence $\xi=\iota(0,\xi,\beta)$ and $\eta=\iota(k,\eta,\alpha)$.

The upper-bound of $\R^{(1),y}_{0_{\N}}$ varies depending on whether the modes $(k,\eta,\alpha)$ is resonant or non-resonant.  More precisely,
\begin{align*}
\R^{(1),y}_{0_{\N}}=&- it^{-3}\sum_{k,\al,\beta\in\Z}\int_{\eta,\xi} {\mathds{1}}_{|\xi|>100|\beta|}\big[{\mathds{1}}_{t\notin\rm{I}_{k,\eta}}+{\mathds{1}}_{t\in\rm{I}_{k,\eta}}\big]\fr{A_k^{\sig_1}(t,\eta,\al)}{A_0^{\sig_1}(t,\xi,\beta)}(\widehat{A^{\sig_1}u}_0)(t,\xi, \beta)_\N\cdot (k,\eta-\xi,\al-\beta)\\
&\qquad \qquad \qquad \qquad  \times \widehat{\theta}_k(t,\eta-\xi,\al-\beta)_{<\N/8} \overline{\widehat{A^{\sig_1}\theta}}_k(t,\eta,\al) d\xi d\eta\\
=&\R^{(1),\nr,\nr}_{0_{\N}}+\R^{(1),\res,\nr}_{0_{\N}}.
\end{align*}

 We present the estimate of the first term in the decomposition of $\R^{(1),y}_{0_{\N}}$ below. Since here the interaction is between two non-resonant modes, Lemma~\ref{lem-ratio-J} shows that
\[
\begin{aligned}
{\mathds{1}}_{t\notin{\rm I}_{k,\eta}}\fr{{\rm J}_k(t,\eta,\al)}{{\rm J}_0(t,\xi,\beta)}
\lesssim e^{2 \mu |k,\eta-\beta,\alpha-\beta|^{\fr{1}{2}}}.
\end{aligned}
\]
Using this estimate together with \eqref{up-u0-1}, we obtain
\begin{align}\label{e-HL101}
\left| \R^{(1), \nr,\nr}_{0_{\N}}\right| 
\nn\les&t^{-3}\sum_{k,\al,\beta\in\Z}\int_{\eta,\xi}  \left(\fr{\la t\ra|\beta|}{\xi^2+\beta^2}+1 \right)\left|\left(\reallywidehat{A^{\sig_1-2}B\theta}_0\right)(t,\xi,\beta)_\N\right|\\
\nn&\qquad \qquad \qquad\times \left|e^{{\bf c}\lm(t)|k, \eta-\xi,\al-\beta|^s}\widehat{\theta}_k(t,\eta-\xi,\al-\beta)_{<\N/8}\right| \left| {\widehat{A^{\sig_1}\theta}}_k(t,\eta,\al)\right|d\xi d\eta\\
\les& \left(\left\|{\pr_z}{\Dl_{yz}^{-1}}B\theta_{0_{\N}}\right\|_{\sig_1-2}\|t^{-2}\theta_{\sim \N}\|_{\sig_1}+\fr{1}{t^{\fr32}} \left\|B\theta_{0_{\N}}\right\|_{\sig_1-2}\left\|t^{-\fr32}\theta_{\sim \N}\right\|_{\sig_1}\right) \|\theta\|_{\mathcal{G}^{\lambda,\sig_5}}.
\end{align}

Next, we estimate the second term in the decomposition of $\R^{(1),y}_{0_{\N}}$. Again, by Lemma~\ref{lem-ratio-J},  we have
\begin{equation}\label{ratio-Jk0}
\begin{aligned} 
{\mathds{1}}_{t\in {\rm I}_{k,\eta}}\fr{{\rm J}_k(t,\eta,\al)}{{\rm J}_0(t,\xi,\beta)}
\lesssim \frac{|\eta|}{k^2(1+|t-\frac{\eta}{k}|)} e^{2\mu|k,\eta-\xi,\alpha-\beta|^{s}}.
\end{aligned}
\end{equation}
 Let us consider the following two cases:

{\bf Case 1: $t\in{\rm I}_{k,\eta}\cap{\rm I}_{k,\xi}^c$}. In this case, the fact that $t\notin {\rm I}_{k,\xi}$ implies  $\fr{|\xi|}{k^2(1+|t-\fr{\xi}{k}|)}\les 1 $. Hence, we have
\be
\fr{|\eta|}{|k|^2}\fr{1}{1+\left|t-\fr{\eta}{k}\right|}=\fr{|\eta|-|\xi|+|\xi|}{|k|^2}\fr{1}{1+|t-\fr{\xi}{k}|}\fr{1+|t-\fr{\xi}{k}|}{1+|t-\fr{\eta}{k}|}\les \la\eta-\xi\ra.
\ee
Then the estimate of $\R^{(1),\res,\nr}_{0_{\N}}$ for the corresponding case is the same as the one for $\R^{(1),\nr,\nr}_{0_{\N}}$.

{\bf Case 2: $t\in{\rm I}_{k,\eta}\cap{\rm I}_{k,\xi}$}. 
Due to the fact that $t\approx \fr{|\xi|}{|k|}\approx \fr{|\eta|}{|k|} $, one easily deduces that
\[
 \fr{\la t\ra|\beta|^2+|\xi\beta|}{(\xi^2+\beta^2)^{2}} \fr{|\xi|}{|k|^2}{\mathds{1}}_{t\in{\rm I}_{k,\xi}}\les1.
\]
Combining the above two inequalities with \eqref{ratio-A-1} and \eqref{ratio-Jk0}, and in view of  \eqref{exp-u} again, we  infer that 
\begin{align*}
\left|\R^{(1),\res,\nr}_{0_{\N}}\right|
\lesssim& t^{-3} \sum_{k,\al,\beta\in\Z}\int_{\eta,\xi} \left|(\widehat{A^{\sig_1}\theta}_0)(t,\xi,\beta)_{\N}\right|\fr{\la t\ra|\beta|^2+|\xi\beta|}{(\xi^2+\beta^2)^2}\frac{{\mathds{1}}_{t\in{\rm I}_{k,\eta}\cap{\rm I}_{k,\xi}}|\eta|}{k^2(1+|t-\frac{\eta}{k}|)} e^{2\mu|k,\eta-\xi,\alpha-\beta|^{\fr12}}\\
\nn&\qquad \qquad \qquad \times \left|e^{c\lm(t)| k, \eta-\xi,\al-\beta|^{s}}\widehat{\nb \theta}_k(t,\eta-\xi,\al-\beta)_{<\N/8}\right| \left|{\widehat{A^{\sig_1}\theta}}_k(t,\eta,\al)\right|d\xi d\eta\\
\lesssim& t^{-3}\sum_{k,\al,\beta\in\Z}\int_{\eta,\xi}  \left|(\widehat{A^{\sig_1}\theta}_0)(t,\xi,\beta)_{\N}\right|\frac{{\mathds{1}}_{t\in{\rm I}_{k,\eta}}}{\sqrt{1+|t-\frac{\eta}{k}|}}\frac{{\mathds{1}}_{t\in{\rm I}_{k,\xi}}}{ \sqrt{1+|t-\frac{\xi}{k}|}}\\
\nn&\qquad \qquad \qquad  \times \left|e^{\mathbf{c}\lm(t)| k, \eta-\xi,\al-\beta|^s}\widehat{ \theta}_k(t,\eta-\xi,\al-\beta)_{<\N/8}\right| \left|{\widehat{A^{\sig_1}\theta}}_k(t,\eta,\al)\right|d\xi d\eta\\
\lesssim& t^{-3}\sum_{k,\al,\beta\in\Z}\int_{\eta,\xi}  \left|(\widehat{A^{\sig_1}\theta}_0)(t,\xi,\beta)_{\N}\right| \sqrt{\frac{\partial_t w_0(t,\iota(0,\xi,\beta))}{w_0(t,\iota(0,\xi,\beta))}}\sqrt{\frac{\partial_t w_k(t,\iota(k,\eta,\alpha))}{w_k(t,\iota(k,\eta,\alpha))}}\\
\nn&\qquad \qquad \qquad \times \left|e^{\mathbf{c}\lm(t)| k, \eta-\xi,\al-\beta|^s}\widehat{ \theta}_k(t,\eta-\xi,\al-\beta)_{<\N/8}\right| \left|{\widehat{A^{\sig_1}\theta}}_k(t,\eta,\al)\right|d\xi d\eta\\
\lesssim& \norm{\theta}_{\mathcal{G}^{\lm,\sigma_5}} \norm{t^{-\fr32}\sqrt{\fr{\partial_tw}{w}}A^{\sigma_1}_0 {\theta_{0_{\N}}}}_{L^2} \norm{t^{-\fr32}\sqrt{\fr{\partial_tw}{w}}A^{\sigma_1} \theta_{\sim \N}}_{L^2}.
\end{align*}

\subsubsection{Treatment of $\R^{(2)}_{0_{\N}}$} \label{resonant} 
Now, we move on to discussing $\R^{(2)}_{0_{\N}}$. Noting that now it holds that $|k,\eta-\xi,\al-\beta|\le\fr{3}{16}|\xi,\beta|$, then by Lemma \ref{total growth} we have
\[
A^{\sig_1}_k(t,\eta-\xi,\al-\beta)\les e^{\lm(t)|k,\eta-\xi,\al-\beta|^s}\la \xi,\beta \ra^{\sig_1}e^{\fr{\mu}{2}|\iota(k,\eta-\xi,\al-\beta)|^\fr12}\les e^{c\lm(t)|\xi,\beta|^s},
\]
for some $c\in(0,1)$. Combining this with \eqref{exp-u} yields that
\begin{align*}
\big|\R^{(2)}_{0_{\N}}\big|
&\lesssim t^{-3}\sum_{k,\al,\beta\in\Z}\int_{\eta,\xi}  \left| e^{c\lm(t)|\xi,\beta|^s}|\xi,\beta|  \fr{\la t\ra|\beta|+1}{\xi^2+\beta^2}\la|\xi|^{\fr12},\beta\ra \widehat{\theta}_0(t,\xi,\beta)_\N \right| \\& \qquad \qquad \qquad \qquad \times    \left| \widehat{\theta}_k(t,\eta-\xi,\al-\beta)_{<\N/8}\right|\left|\overline{\widehat{A^{\sig_1}\theta}}_k(t,\eta,\al)\right| \;d\xi d\eta\\&
\lesssim \epsilon  \left(\left\|{\pr_z}{\Dl_{yz}^{-1}}B\theta_{0_{\N}}\right\|_{\sig_1-2}\norm{t^{-2}\theta_{\sim \N}}_{\sigma_1}+\fr{1}{t^{\fr32}}\left\|B\theta_{0_{\N}}\right\|_{\sig_1-2}  \norm{t^{-\fr32}\theta_{\sim \N}}_{\sigma_1}\right).
\end{align*}


\subsection{Transport term (zero-mode interactions)} We now proceed and focus on the transport term $\T_{0_\N}$.
Before we begin any estimation, a direct calculation yields
\begin{align}
\fr{A_k^{\sig_1}(t,\eta,\al)}{A_k^{\sig_1}(t,\xi,\beta)}-1\nn=&\left(e^{\lm(t)(| k,\eta,\al|^s-|k,\xi,\beta|^s)}-1\right)+e^{\lm(t)(| k,\eta,\al|^s-|k,\xi,\beta|^s)}\left(\fr{\mathrm{J}_k(t,\eta,\al)}{\mathrm{J}_k(t,\xi,\beta)}-1\right)\fr{\la k,\eta,\al\ra ^{\sigma_1}}{\la k,\xi,\beta\ra^{\sigma_1}}\\
\nn&+e^{\lm(t)(| k,\eta,\al|^s-| k,\xi,\beta|^s)}\left(\fr{\la k,\eta,\al\ra ^{\sigma_1}}{\la k,\xi,\beta\ra^{\sigma_1}}-1\right).
\end{align}
In light of the equation above, we decompose $\T_{0_\N}$ as follows
\begin{equation}\label{decomp T_0}
\T_{0_\N}=\T^{(1)}_{0_\N}+\T^{(2)}_{0_\N}+\T^{(3)}_{0_\N}.
\end{equation}
It is important to note that here, the following conditions still hold, namely,
\[
\frac{\N}{2} \leq |l,\xi,\beta|\leq \frac{3\N}{2}, \quad |k-l,\eta-\xi,\alpha-\beta|\leq \frac{3\N}{32}.
\]
This implies $|k-l,\eta-\xi,\alpha-\beta|\leq \frac{3}{16} |l,\xi,\beta|$ which yields \eqref{comparability}.

\subsubsection{Treatment of $\T^{(1)}_{0_\N}$}
Let us start by analyzing the first term in the above summation \eqref{decomp T_0}. Notice that, on the support of its integrand \eqref{comparability} still holds. Again, using such fact together with the inequality $|e^x-1|\leq |x|e^{|x|}$, the first inequality in Lemma \ref{exponent ineq} and definition of $u_0$ in \eqref{exp-u}, we have
\begin{align}
\left|\T^{(1)}_{0_\N}\right|\nn\les&t^{-3}\lm(t)\sum_{k,\al,\beta\in\Z}\int_{\eta,\xi} \fr{|\eta-\xi,\al-\beta|}{|k,\eta,\al|^{1-s}+|k,\xi,\beta|^{1-s}}|\widehat{u}_0(t,\eta-\xi,\al-\beta)_{<\N/8}||k,\xi,\beta|\\
\nn&\qquad \qquad \qquad \times \left|\widehat{A^{\sig_1}\theta}_k(t,\xi,\beta)_{\N}\right| \left|\overline{\widehat{A^{\sig_1}\theta}}_k(t,\eta,\al)\right|d\xi d\eta
\\\nn\lesssim&  \lambda(t)\JB{t} \sum_{k,\al,\beta\in\Z}\int_{\eta,\xi} \left|\overline{\widehat{A^{\sig_1}\theta}}_k(t,\eta,\al)\right|\fr{|\al-\beta|}{(\eta-\xi)^2+(\al-\beta)^2}|\widehat{\Pe_{\ne}^z\theta_0}(t,\eta-\xi,\al-\beta)_{<\N/8}|  \\
\nn&\qquad \qquad \qquad \times e^{c \lm(t)|\eta-\xi,\al-\beta|^s} t^{-\fr32} |k,\xi,\beta|^{\frac{s}{2}} t^{-\fr32}|k,\eta,\alpha|^{\frac{s}{2}}  \left|\widehat{A^{\sig_1}\theta}_k(t,\xi,\beta)_{\N}\right| d\xi d\eta
\\\nn\lesssim& \lambda(t) \JB{t} \|t^{-\fr32}|\nabla|^{\fr{s}{2}}\theta_{\sim\N}\|_{\sig_1} \|t^{-\fr32}|\nabla|^{\fr{s}{2}}\theta_{\N}\|_{\sig_1}   {  \|\Pe_{\ne}^z\theta_0\|_{\mathcal{G}^{\lm,\sig_6}}} \lesssim \frac{\epsilon}{\JB{t}^2} \|t^{-\fr32}|\nabla|^{\fr{s}{2}}\theta_{\sim\N}\|_{\sig_1} \|t^{-\fr32}|\nabla|^{\fr{s}{2}}\theta_{\N}\|_{\sig_1} .
\end{align}

We now move on to $\T^{(2)}_{0_\N}$. 
\subsubsection{Treatment of $\T^{(2)}_{0_\mathrm{N}}$}
Similar to the work of \cite{BM1}, we start by decomposing $\T^{(2)}_{0_\mathrm{N}}$ in the following manner
\begin{equation}\label{transport0-2-decomp}
\begin{aligned}
\T^{(2)}_{0_\N}&= t^{-3}\sum_{k,\al,\beta\in\Z}\int_{\eta,\xi} [{\mathds{1}}^{\mf{S}}+{\mathds{1}}^{\mf{L}}]e^{\lm(t)( |k,\eta,\al|^s-| k,\xi,\beta|^s)}\left(\fr{\mathrm{J}_k(t,\eta,\al)}{\mathrm{J}_k(t,\xi,\beta)}-1\right)\widehat{u}_0(t,\eta-\xi,\al-\beta)_{<\N/8}\\
\nn& \qquad \qquad \cdot i(k,\xi,\beta) \fr{\la k,\eta,\al\ra^{\sigma_1}}{\la  k,\xi,\beta\ra^{\sigma_1}}\left(\widehat{A^{\sig_1}\theta}_k(t,\xi,\beta)_\N\right) \left(\overline{\widehat{A^{\sig_1}\theta}}_k(t,\eta,\al)\right)d\xi d\eta\\&
=\T^{(2),\mf{S}}_{0_\N}+\T^{(2),\mf{L}}_{0_\N},
\end{aligned}
\end{equation}
where ${\mathds{1}}^{\mf{S}}={\mathds{1}}_{t\leq \frac{1}{2}\min{(\sqrt{|\iota(k,\eta,\alpha)|},\sqrt{|\iota(k,\xi,\beta)|})}}$ and ${\mathds{1}}^{\mf{L}}=1-{\mathds{1}}^{\mf{S}}$.
This portion requires an estimate like the one in \cite[Lemma 3.7]{BM1} which in the present work is precisely the content of Lemma~\ref{short time}. We use it to absorb the $\frac{1}{2}$ derivative. 
By using the bootstrap hypotheses, Lemma~\ref{lem-ratio-J}, \eqref{Jk|Jl-1},  \eqref{exp-u} together with the fact that $| k,\eta,\alpha | \approx | k,\xi,\beta |$ and $|e^x-1|\leq |x| e^{|x|}$, we obtain that
\begin{align}
\left|\T^{(2),\mf{S}}_{0_\N}\right|\nn\les& t^{-3} \sum_{k,\al,\beta\in\Z}\int_{\eta,\xi} {\mathds{1}}^\mf{S} e^{c\lambda(t)|\eta-\xi,\alpha-\beta|^s}  \la\eta-\xi, \alpha-\beta\ra  |\widehat{u}_0(t,\eta-\xi,\al-\beta)_{<\N/8}| e^{3\mu|\eta-\xi,\beta-\alpha|^\fr12} \\
\nn&\qquad \qquad \qquad \qquad \times |k,\xi,\beta|^{\frac{1}{2}}\left|\widehat{A^{\sig_1}\theta}_k(t,\xi,\beta)_{\N}\right| \left|\overline{\widehat{A^{\sig_1}\theta}}_k(t,\eta,\al)\right|d\xi d\eta\\
\nn\les&\fr{\la t \ra t^{-3}}{t^{s-\fr12}} \sum_{\substack{k,\al,\beta\in\Z,\\ \al\ne\beta}}\int_{\eta,\xi} {\mathds{1}}^\mf{S} \left|\overline{\reallywidehat{|\nabla|^{\fr{s}{2}}A^{\sig_1}\theta}}_k(t,\eta,\al)\right|\left|\reallywidehat{|\nabla|^{\fr{s}{2}}A^{\sig_1}\theta}_k(t,\xi,\beta)_\N\right|e^{\mathbf{c}\lambda|\eta-\xi,\alpha-\beta|^s}  \\
\nn&\qquad \qquad \qquad \qquad \times \fr{|\al-\beta|}{(\eta-\xi)^2+(\al-\beta)^2}|\reallywidehat{\Pe_{\ne}^z\theta_0}(t,\eta-\xi,\al-\beta)_{<\N/8}| d\xi d\eta\\
\nn\les& \fr{\la t \ra}{t^{s-\fr12}}\Big(\|\Pe_{\ne}^z\theta_0\|_{\mathcal{G}^{\lm,\sig_6}}\|t^{-\fr32}|\nabla|^{\frac{s}{2}}\theta_{\N}\|_{\sig_1} \|t^{-\fr32}|\nabla|^{\frac{s}{2}}\theta_{\sim \N} \|_{\sig_1}\Big)\\
\lesssim& \frac{\epsilon}{\JB{t}^{\fr32 +s}} \|t^{-\fr32}|\nabla|^{\frac{s}{2}}\theta_{\N}\|_{\sig_1} \|t^{-\fr32}|\nabla|^{\frac{s}{2}}\theta_{\sim \N} \|_{\sig_1}.
\end{align}
Let us now consider its counterpart (long-time regime), namely  $\T^{(2),\mf{L}}_{0_\N}$. The integrand in this regime is supported on $t >\frac{1}{2}\min{(\sqrt{|\iota_1|},\sqrt{|\iota_2|})}\approx \sqrt{|\iota_1|}\approx \sqrt{|\iota_2|}$. 
In this regime, we have that $|k,\xi,\beta|\lesssim |k,\xi,\beta|^{\frac{s}{2}}|k,\eta,\alpha|^{\frac{s}{2}}t^{2-2s}$. 
Hence, via Lemma \ref{lem-ratio-J} and \eqref{exp-u}, we can say that
\begin{align*}
\left| \T^{(2),\mf{L}}_{0_\N}\right| &\lesssim t^{-3} \sum_{k,\al,\beta\in\Z}\int_{\eta,\xi} {\mathds{1}}^{\mf{L}}e^{\lm(t)(| k,\eta,\al|^s-\la k,\xi,\beta|^s)}e^{3\mu|\eta-\xi,\alpha-\beta|^{\fr12}} |\widehat{u}_0(t,\eta-\xi,\al-\beta)_{<\N/8}|  \\
\nn& \qquad \qquad \qquad \qquad \times |(k,\xi,\beta)| \left|\widehat{A^{\sig_1}\theta}_k(t,\xi,\beta)_{\N}\right| \left|\overline{\widehat{A^{\sig_1}\theta}}_k(t,\eta,\al)\right| d\xi d\eta\\
&\lesssim
t^{-3}\sum_{k,\al,\beta\in\Z}\int_{\eta,\xi} e^{\mathbf{c}\lm(t)|\eta-\xi,\al-\beta|^s}\fr{\la t\ra(\al-\beta)^2+|(\eta-\xi)(\al-\beta)|}{\big((\eta-\xi)^2+(\al-\beta)^2\big)^2}|\widehat{\mathbb{P}_{\ne}^z\theta_0}(t,\eta-\xi,\al-\beta)_{<\N/8}|\\
\nn& \qquad \qquad \qquad \qquad \times t^{2-2s} \left||k,\xi,\beta|^{\frac{s}{2}}\widehat{A^{\sig_1}\theta}_k(t,\xi,\beta)_{\N}\right| \left||k,\eta,\alpha|^{\frac{s}{2}}\overline{\widehat{A^{\sig_1}\theta}}_k(t,\eta,\al)\right|d\xi d\eta\\
&\lesssim \frac{\epsilon}{\JB{t}^{2s}} \|t^{-\fr32}|\nabla|^{\frac{s}{2}}\theta_{\N}\|_{\sig_1} \|t^{-\fr32}|\nabla|^{\frac{s}{2}}\theta_{\sim \N} \|_{\sig_1}.
\end{align*}

\subsubsection{Treatment of $\T^{(3)}_{0_{\N}}$}
Next, we present an estimate associated with the third term in the decomposition: $\T^{(3)}_{0_{\N}}$.
Observe that by using the support of the integrand $|\eta-\xi,\alpha-\beta|\leq \fr{3}{16} |k,\xi,\beta|$ and the mean value theorem, we obtain
\begin{equation}\label{useful est T3}
\begin{aligned}
\nn&e^{\lm(t)( |k,\eta,\al|^s-|k,\xi,\beta|^s)}\left|\fr{\la k,\eta,\al\ra^{\sigma_1}}{\la k,\xi,\beta\ra^{\sigma_1}}-1\right|\les e^{c\lm(t)| \eta-\xi,\al-\beta|^s}\fr{|\eta-\xi,\al-\beta|}{\la k,\xi,\beta\ra}.
\end{aligned}
\end{equation}
Combining this inequality with the definition of $u_0$ in \eqref{exp-u}, we are led to
\begin{align}
\left|\T^{(3)}_{0_\mathrm{N}}\right|\nn\les& t^{-3}\sum_{k,\al,\beta\in\Z}\int_{\eta,\xi} e^{c\lm(t)|\eta-\xi,\al-\beta|^s}|\eta-\xi,\al-\beta||\widehat{u}_0(t,\eta-\xi,\al-\beta)_{<\N/8}|\\
\nn&\qquad \qquad \qquad \qquad \times \left|\widehat{A^{\sig_1}\theta}_k(t,\xi,\beta)_{\N}\right| \left|\overline{\widehat{A^{\sig_1}\theta}}_k(t,\eta,\al)\right|d\xi d\eta\\
\nn\les&\la t\ra t^{-3} \sum_{k,\al,\beta\in\Z}\int_{\eta,\xi} e^{{ c}\lm(t)  |\eta-\xi,\al-\beta|^s}\fr{|\al-\beta|}{(\eta-\xi)^2+(\al-\beta)^2}|\widehat{\Pe_{\ne}^z\theta_0}(t,\eta-\xi,\al-\beta)_{<\N/8}|\\
\nn&\qquad \qquad \qquad \qquad \times \left|\widehat{A^{\sig_1}\theta}_k(t,\xi,\beta)_{\N}\right| \left|\overline{\widehat{A^{\sig_1}\theta}}_k(t,\eta,\al)\right|d\xi d\eta\\
\nn\les&\la t\ra \|\Pe_{\ne}^z\theta_0\|_{\mathcal{G}^{\lm,\sig_6}} \| t^{-\fr32} \theta_{\N}\|_{\sig_1} \| t^{-\fr32}\theta_{\sim \N}\|_{\sig_1} \les\fr{\eps}{\la t\ra^2}\|t^{-\fr32}\theta_{\N}\|_{\sig_1} \|t^{-\fr32}\theta_{\sim \N}\|_{\sig_1}.
\end{align}

\subsection{Reaction term (non-zero-mode interactions)}
Next, we focus on the high-low interaction.
\begin{align}
\R_{{\neq}_\N}\nn=&-t^{-3} i \sum_{\substack{k, l, \al,\beta\in\Z \\ l \neq 0}}\int_{\eta,\xi} A_k^{\sig_1}(t,\eta,\al)\widehat{u}_{\ne}(t, l,\xi, \beta)_\N \cdot (k-l,\eta-\xi,\al-\beta)\\
\nn&\qquad \qquad \qquad \qquad \times \widehat{\theta}_{k-l}(t,\eta-\xi,\al-\beta)_{<\N/8} \left(\overline{\widehat{A^{\sig_1}\theta}}_k(t,\eta,\al)\right) d\xi d\eta\\
\nn&\; +t^{-3} i \sum_{\substack{k, l, \al,\beta\in\Z \\  l \neq 0}} \int_{\eta,\xi}  \widehat{u}_{\ne}(t, l, \xi,\beta)_\N \cdot (k-l,\eta-\xi,\al-\beta)\\
\nn& \qquad \qquad \qquad \qquad \times \left(\widehat{A^{\sig_1}\theta}_{k-l}(t,\eta-\xi,\al-\beta)_{<\N/8}\right) \left(\overline{\widehat{A^{\sig_1}\theta}}_k(t,\eta,\al)\right) d\xi d\eta\\
\nn=&\R^{(1)}_{{\neq}_\N}+\R^{(2)}_{{\neq}_\N}.
\end{align}
We narrow down our analysis of the Reaction term above by considering several sub-cases depending on whether the frequency $(k,\eta,\alpha)$ and $(l,\xi,\beta)$ are resonant or non-resonant. Following the notation used in \cite{Bedrossian_Germain_Masmoudi_I}, we define

\begin{equation}
\begin{aligned}
1&=\mathds{1}_{t\notin \textup{I}_{k,\eta},t\notin \textup{I}_{l,\xi}}+\mathds{1}_{t\notin \textup{I}_{k,\eta},t\in \textup{I}_{l,\xi}}+\mathds{1}_{t\in \textup{I}_{k,\eta},t\notin \textup{I}_{l,\xi}}+\mathds{1}_{t\in \textup{I}_{k,\eta},t\in \textup{I}_{l,\xi}}
\\&=: \chi^{\nr,\nr}+ \chi^{\nr,\res}+ \chi^{\res,\nr}+ \chi^{\res,\res}.
\end{aligned}
\end{equation}
As a result, we get the following decomposition
\begin{align}
\nn\left|\R^{(1)}_{{\neq}_\N}\right|&=t^{-3}\sum_{\substack{k, l, \al,\beta\in\Z \\  l \neq 0}} \bigg| \int_{\eta,\xi} \Big[\chi^{\nr,\nr}+ \chi^{\nr,\res}+ \chi^{\res,\nr}+ \chi^{\res,\res} \Big] A_k^{\sig_1}(t,\eta,\al)\widehat{u}_{\ne}(t, l,\xi, \beta)_\N \\&\nn \qquad \qquad \qquad \qquad \cdot \widehat{\nabla\theta}_{k-l}(t,\eta-\xi,\al-\beta)_{<\N/8} \left(\overline{\widehat{A^{\sig_1}\theta}}_k(t,\eta,\al)\right) d\xi d\eta\bigg|
\\\label{Decom R1neq}&
=:\R^{(1),\nr,\nr}_{{\neq}_\N}+\R^{(1),\nr,\res}_{{\neq}_\N}+\R^{(1),\res,\nr}_{{\neq}_\N}+\R^{(1),\res,\res}_{{\neq}_\N}.
\end{align}
Recalling \eqref{exp-u}, we know that  
\begin{align}\label{u-hat}
|\widehat{u}_{\ne}(t, l,\xi, \beta)|\nn\les&\fr{\left|(t\beta^2+l\xi, -(l^2+\beta^2), \beta(\xi-lt)\right|}{(l^2+(\xi-l t)^2+\beta^2)^2} |\widehat{\theta}_{\ne}(t, l,\xi,\beta)|\\
\les&\fr{\left(t|\beta|+|\xi|+|l,\beta|\right) |l,\beta|}{(l^2+(\xi-l t)^2+\beta^2)^2} |\widehat{\theta}_{\ne}(t, l,\xi,\beta)|,
\end{align}
which allows us to say the following for the first three terms in \eqref{Decom R1neq} 
\begin{align*}
&\R^{(1),\nr,\nr}_{{\neq}_\N}+\R^{(1),\nr,\res}_{{\neq}_\N}+\R^{(1),\res,\nr}_{{\neq}_\N}\\&\les t^{-3}\sum_{\substack{k, l, \al,\beta\in\Z \\ l \neq 0}} \bigg| \int_{\eta,\xi} \Big[\chi^{\nr,\nr}+ \chi^{\nr,\res}+ \chi^{\res,\nr}\Big]|\widehat{A^{\sig_1}\theta}_{\ne}(t, l,\xi,\beta)_\N||{\widehat{A^{\sig_1}\theta}}_k(t,\eta,\al)|\\&\qquad \qquad \qquad \qquad \times \fr{A_k^{\sig_1}(t,\eta,\al)}{A_l^{\sig_1}(t,\xi,\beta)}\fr{\left(t|\beta|+|\xi|+|l,\beta|\right) |l,\beta|}{(l^2+(\xi-l t)^2+\beta^2)^2} \left| \widehat{\nabla\theta}_{k-l}(t,\eta-\xi,\al-\beta)_{<\N/8}\right| d\xi d\eta\bigg|.
\end{align*}

Furthermore, for the last term in \eqref{Decom R1neq}, a direct calculation yields
\begin{align*}
    \R^{(1),\res,\res}_{{\neq}_\N}&\les t^{-3}\sum_{\substack{k, l, \al,\beta\in\Z \\ l \neq 0, \ k\neq l}} \bigg| \int_{\eta,\xi} \chi^{\res,\res}|\widehat{A^{\sig_1}\theta}_{\ne}(t, l,\xi,\beta)_\N||{\widehat{A^{\sig_1}\theta}}_k(t,\eta,\al)|\\&\qquad \qquad \qquad \times \fr{A_k^{\sig_1}(t,\eta,\al)}{A_l^{\sig_1}(t,\xi,\beta)}\fr{\left(t|\beta|+|\xi|+|l,\beta|\right) |l,\beta|}{(l^2+(\xi-l t)^2+\beta^2)^2}  \left|\widehat{\nabla\theta}_{k-l}(t,\eta-\xi,\al-\beta)_{<\N/8}\right| d\xi d\eta\bigg|\\&
    \quad+t^{-3}\sum_{\substack{k, \al,\beta\in\Z \\ k \neq 0}} \bigg|
    \int_{\eta,\xi} \chi^{\res,\res}|\widehat{A^{\sig_1}\theta}_{\ne}(t, k,\xi,\beta)_\N||{\widehat{A^{\sig_1}\theta}}_k(t,\eta,\al)|\\&\qquad \qquad \qquad \times \fr{A_k^{\sig_1}(t,\eta,\al)}{A_k^{\sig_1}(t,\xi,\beta)} \frac{(k^2+\beta^2+|\xi-kt||\beta|)}{(k^2+(\xi-kt)^2+\beta^2)^2} \left|\widehat{\nabla\theta}_{0}(t,\eta-\xi,\al-\beta)_{<\N/8}\right| d\xi d\eta\bigg|.
\end{align*}

As a consequence, keeping both aforementioned inequalities in mind, via Lemma~\ref{est:NR-NR},  Lemma~\ref{R NR}, Lemma~\ref{est: NR-R}, and Lemma~\ref{R R},  we manage to treat the terms $\R^{(1),\nr,\nr}_{{\neq}_\N}, \R^{(1),\res,\nr}_{{\neq}_\N}, \R^{(1),\nr, \res}_{{\neq}_\N}$ and $\R^{(1),\res,\res}_{{\neq}_\N}$, respectively and obtain
\begin{align*}
    \left| \R^{(1)}_{{\neq}_\N} \right|\les \epsilon \Bigg[ &\frac{1}{\JB{t}^2} \norm{\fr{\theta_{{\neq}_\N}}{t^{\fr32}}}_{\sigma_1} \norm{\fr{\theta_{{\neq}_{\sim\N}}}{t^{\fr32}}}_{\sigma_1}+ \fr{1}{ \JB{t}^{1+s}}  \norm{\frac{|\nabla|^{\fr{s}{2}}}{t^\fr32} \theta_{{\neq}_{\N}}}_{\sigma_1} \norm{\frac{|\nabla|^{\fr{s}{2}}}{t^\fr32} \theta_{{\neq}_{\sim \N}}}_{\sigma_1}\\&  +  \norm{\frac{1}{t^{\fr32}}\sqrt{\fr{\partial_t w}{w}} \theta_{{\neq}_{\N}}}_{\sigma_1} \norm{\frac{1}{t^{\fr32}}\sqrt{\fr{\partial_t w}{w}} \theta_{{\neq}_{\sim\N}}}_{\sigma_1}\Bigg].
\end{align*}

Now, let us proceed and estimate $\R^{(2)}_{{\neq}_\N}$. Here, we use the fact that $|k-l,\eta-\xi, \alpha-\beta|\leq |l,\xi,\beta|$.
\[
\begin{aligned}
|\R^{(2)}_{{\neq}_\N}| &\lesssim t^{-3} \sum_{k,l,\al,\beta\in\Z}\int_{\eta,\xi} |\widehat{u}_{\ne}(t, l, \xi,\beta)_\N| |l,\xi,\beta| \left|\widehat{A^{\sig_1}\theta}_{k-l}(t,\eta-\xi,\al-\beta)_{<\N/8}\right| \left|\overline{\widehat{A^{\sig_1}\theta}}_k(t,\eta,\al)\right| d\xi d\eta\\&
\lesssim \frac{1}{t^3} \norm{t^{-\fr32}\theta}_{\sigma_1} \norm{\theta_{\N}}_{\mathcal{G}^{\lambda,\sigma_5}} \norm{t^{-\fr32}\theta_{\sim\N}}_{\sigma_1} \lesssim \frac{\epsilon}{t^3} \norm{t^{-\fr32}\theta}_{\sigma_1}  \norm{t^{-\fr32}\theta_{\sim\N}}_{\sigma_1} .
\end{aligned}
\]
\subsection{Transport term (non-zero-mode interactions)}

First of all, observe that
\begin{align*}
\fr{A_k^{\sig_1}(t,\eta,\al)}{A_l^{\sig_1}(t,\xi,\beta)}-1
=&\left(e^{\lm(t)( |k,\eta,\al|^s-|l,\xi,\beta|^s)}-1\right)+e^{\lm(t)(|k,\eta,\al|^s-| l,\xi,\beta|^s)}\left(\fr{\mathrm{J}_k(t,\eta,\al)}{\mathrm{J}_l(t,\xi,\beta)}-1\right)\fr{\la k,\eta,\al\ra^{\sigma_1}}{ \la l,\xi,\beta\ra^{\sigma_1}}\\
&+e^{\lm(t)(|k,\eta,\al|^s- |l,\xi,\beta|^s)}\left(\fr{\la k,\eta,\al\ra^{\sigma_1}}{\la l,\xi,\beta\ra^{\sigma_1}}-1\right).
\end{align*}
In light of that, we can decompose $\T_{\neq_\N}$ as follows
\[
\T_{\neq_\N}=\T^{(1)}_{\neq_\N}+\T^{(2)}_{\neq_\N}+\T^{(3)}_{\neq_\N}.
\] It is important to note that, throughout, we are dealing with non-zero modes interactions here $k \neq l$ and  $k,l\neq 0$. Additionally, we would like to note that we will be taking advantage of the estimate of $\widehat{u}_{\neq}$ \eqref{u_ne-decay}.

\subsubsection{Treatment of $\T^{(1)}_{\neq_\N}$}
On the support of its integral \eqref{comparability} still holds: $|k,\eta,\alpha|\approx |l,\xi,\beta|.$ Hence, exploiting the facts $|e^x-1|\leq |x|e^{|x|}$, \eqref{u_ne-decay} and then applying the bootstrap hypotheses, allow us to infer
\begin{align}
\left|\T^{(1)}_{\neq_\N}\right|\nn\les& t^{-3}\sum_{k,l,\al,\beta\in\Z}\int_{\eta,\xi} \Big(e^{\lm(t)( |k,\eta,\al|^s-|l,\xi,\beta|^s)}-1\Big)|\widehat{u}_{\neq}(t,k-l,\eta-\xi,\al-\beta)_{<\N/8}||l,\xi,\beta|\\
\nn&\qquad \qquad \qquad \qquad \times \left|\widehat{A^{\sig_1}\theta}_l(t,\xi,\beta)_{\N}\right| \left|\overline{\widehat{A^{\sig_1}\theta}}_k(t,\eta,\al)\right|d\xi d\eta
\\\nn\lesssim&  \lambda(t)t^{-3} \sum_{k,l,\al,\beta\in\Z}\int_{\eta,\xi}  e^{c\lm(t)( |k-l,\eta-k,\al-\beta|^s}\frac{\langle k-l,\eta-\xi,\al-\beta \rangle^6}{\langle t\rangle^3}|\widehat{\theta}_{\ne}(t,k-l,\eta-\beta,\al-\beta)_{<\N/8}|\\
\nn&\qquad \qquad \qquad \qquad\times |l,\xi,\beta|^{\fr{s}{2}}\left|\widehat{A^{\sig_1}\theta}_l(t,\xi,\beta)_{\N}\right| |k,\eta,\alpha|^\fr{s}{2}\left|\overline{\widehat{A^{\sig_1}\theta}}_k(t,\eta,\al)\right|d\xi d\eta
\\\nn\lesssim& \lambda(t) \frac{\norm{\theta}_{\mathcal{G}^{\lambda,\sigma_5}}}{\JB{t}^3} \norm{t^{-\fr32}|\nb|^\fr{s}{2} \theta_{\sim \N}}_{\sigma_1} \norm{t^{-\fr32}|\nb|^\fr{s}{2} \theta_{\N}}_{\sigma_1}\lesssim \frac{\epsilon}{\JB{t}^3} \norm{t^{-\fr32}|\nb|^\fr{s}{2} \theta_{\sim \N}}_{\sigma_1} \norm{t^{-\fr32}|\nb|^\fr{s}{2} \theta_{\N}}_{\sigma_1}.
\end{align}

\subsubsection{Treatment of $\T^{(2)}_{\neq_\N}$} Similar to the treatment of $\T^{(2)}_{0_\N}$, we begin by considering two distinct time regimes,
${\mathds{1}}^{\mf{S}}={\mathds{1}}_{t\leq \frac{1}{2}\min{(\sqrt{|\iota_1|},\sqrt{|\iota_2|})}}$ and ${\mathds{1}}^{\mf{L}}=1-{\mathds{1}}^{\mf{S}}$. We then can recast $\T^{(2)}_{\neq_\N}$ in the following manner

\begin{equation}\label{transport_ne-2-decomp}
\begin{aligned}
\T^{(2)}_{\neq_\N}&= t^{-3}\sum_{k,l,\al,\beta\in\Z}\int_{\eta,\xi} [{\mathds{1}}^{\mf{S}}+{\mathds{1}}^{\mf{L}}]e^{\lm(t)(|k,\eta,\al|^s-|l,\xi,\beta|^s)}\left(\fr{\mathrm{J}_k(t,\eta,\al)}{\mathrm{J}_l(t,\xi,\beta)}-1\right)\fr{\la k,\eta,\al\ra^{\sigma_1}}{\la l,\xi,\beta\ra^{\sigma_1}}\\
\nn& \qquad \qquad   \times \widehat{u}_{\ne}(t,k-l,\eta-\xi,\al-\beta)_{<\N/8}\cdot i(l,\xi,\beta)\left(\widehat{A^{\sig_1}\theta}_l(t,\xi,\beta)_\N\right) \left(\overline{\widehat{A^{\sig_1}\theta}}_k(t,\eta,\al)\right)d\xi d\eta\\&
=\T^{(2),\mf{S}}_{\ne_\N}+\T^{(2),\mf{L}}_{\ne_\N}.
\end{aligned}
\end{equation}
As before, let us concentrate on the short-time regime first. Using the fact that $ |k,\eta,\alpha| \approx | k,\xi,\beta|,$ and \eqref{Jk|Jl-1}, we get that
\begin{align}
\left|\T^{(2),\mf{S}}_{\ne_\N}\right|\nn&\les t^{-3}\sum_{k,l,\al,\beta\in\Z}\int_{\eta,\xi} {\mathds{1}}^\mf{S} e^{c\lambda(t)|k-l,\eta-\xi,\alpha-\beta|^s} \la k-l, \eta-\xi, \alpha-\beta \ra  |\widehat{u}_{\ne}(t,k-l,\eta-\xi,\al-\beta)_{<\N/8}| \\
\nn& \qquad \qquad \qquad \qquad \times |l,\xi,\beta|^{\frac{1}{2}} e^{3\mu|k-l,\eta-\xi,\beta-\alpha|^\frac{1}{2}}\left|\widehat{A^{\sig_1}\theta}_l(t,\xi,\beta)_{\N}\right| \left|\overline{\widehat{A^{\sig_1}\theta}}_k(t,\eta,\al)\right|d\xi d\eta\\
\nn&\les\frac{\norm{\theta}_{\mathcal{G}^{\lambda,\sigma_5}}}{\JB{t}^{3}} \norm{t^{-\fr32}|\nb|^\fr{s}{2} \theta_{\sim \N}}_{\sigma_1} \norm{t^{-\fr32}|\nb|^\fr{s}{2} \theta_{\N}}_{\sigma_1}\lesssim \frac{\epsilon}{\JB{t}^{3}} \norm{t^{-\fr32}|\nb|^\fr{s}{2} \theta_{\sim \N}}_{\sigma_1} \norm{t^{-\fr32}|\nb|^\fr{s}{2} \theta_{\N}}_{\sigma_1}.
\end{align}
Having established the estimate of $\T^{(2)}_{\ne_\N}$ for the short-time regime, we now continue and derive an estimate for $\T^{(2),\mf{L}}_{\ne_\N}$. Here, we have that $t\gtrsim \sqrt{|\iota_1|}\approx \sqrt{|\iota_2|}$ and $|l,\xi,\beta|\lesssim |l,\xi,\beta|^{\frac{s}{2}}|k,\eta,\alpha|^{\frac{s}{2}}t^{2-2s}$. This together with the following inequality (from Lemma \ref{lem-ratio-J})
\begin{align*}
\dfrac{{\rm J}_k(t,\eta,\alpha)}{{\rm J}_l(t,\xi,\beta)} &\lesssim t e^{2\mu|k-l,\eta-\xi,\alpha-\beta|^{\fr12}}
\end{align*}
and \eqref{u_ne-decay} yield
\begin{align*}
\left| \T^{(2),\mf{L}}_{\ne_\N}\right| &\lesssim t^{-3}\sum_{k,l,\al,\beta\in\Z}\int_{\eta,\xi} {\mathds{1}}^{\mf{L}}e^{\lm(t)(| k,\eta,\al|^s-| l,\xi,\beta|^s)}e^{3\mu|k-l,\eta-\xi,\alpha-\beta|^{\fr12}} |\widehat{u}_{\ne}(t,k-l,\eta-\xi,\al-\beta)_{<\N/8}|  \\
\nn&  \qquad \qquad \qquad \qquad \times |(l,\xi,\beta)| \JB{t}\left|\widehat{A^{\sig_1}\theta}_k(t,\xi,\beta)_{\N}\right| \left|\overline{\widehat{A^{\sig_1}\theta}}_k(t,\eta,\al)\right| d\xi d\eta\\&
\lesssim t^{-3}
\sum_{k,l,\al,\beta\in\Z}\int_{\eta,\xi} {\mathds{1}}^{\mf{L}} e^{\mathbf{c}\lm(t)|k-l,\eta-\xi,\al-\beta|^s}\frac{\JB{k-l,\eta-\xi,\alpha-\beta}^6}{\JB{t}^3} t^{3-2s}\left||k,\eta,\alpha|^{\frac{s}{2}}\overline{\widehat{A^{\sig_1}\theta}}_k(t,\eta,\al)\right|\\
\nn& \qquad \qquad \qquad \qquad  \times  |\widehat{\theta}_{\ne}(t,k-l,\eta-\xi,\al-\beta)_{<\N/8}| \left||l,\xi,\beta|^{\frac{s}{2}}\widehat{A^{\sig_1}\theta}_l(t,\xi,\beta)_{\N}\right| d\xi d\eta\\&
\les\frac{\norm{\theta}_{\mathcal{G}^{\lambda,\sigma_5}}}{\JB{t}^{2s}} \norm{t^{-\fr32}|\nb|^\fr{s}{2} \theta_{\sim \N}}_{\sigma_1} \norm{t^{-\fr32}|\nb|^\fr{s}{2} \theta_{\N}}_{\sigma_1}\lesssim \frac{\epsilon}{\JB{t}^{2s}} \norm{t^{-\fr32}||\nb|^\fr{s}{2}\theta_{\sim \N}}_{\sigma_1} \norm{t^{-\fr32}||\nb|^\fr{s}{2} \theta_{\N}}_{\sigma_1},
\end{align*}
where the last inequality is obtained via the bootstrap hypotheses.

Let us now continue and estimate $\T^{(3)}_{\neq_\N}$.
\subsubsection{Treatment of $\T^{(3)}_{\neq_\N}$} Observe that using the support of the integrand $|k-l,\eta-\xi,\alpha-\beta|\leq \fr{3}{16} |l,\xi,\beta|$ and the mean value theorem,
\begin{align}
\nn&e^{\lm( |k,\eta,\al|^s- |l,\xi,\beta|^s)}\left|\fr{ \la k,\eta,\al\ra^{\sigma_1}}{\la  l,\xi,\beta\ra^{\sigma_1}}-1\right|\les e^{c\lm(t) |k-l, \eta-\xi,\al-\beta|^s}\fr{| k-l,\eta-\xi,\al-\beta|}{ \la l,\xi,\beta\ra }.
\end{align}
Upon applying this inequality, \eqref{u_ne-decay}, and the bootstrap hypotheses, we are then led to say
\begin{align}
\left|\T^{(3)}_{\ne_\mathrm{N}}\right|\nn\les& t^{-3}\sum_{k,l,\al,\beta\in\Z}\int_{\eta,\xi} e^{c\lm(t) |k-l, \eta-\xi,\al-\beta|^s}|k-l,\eta-\xi,\al-\beta||\widehat{u}_{\ne}(t,k-l,\eta-\xi,\al-\beta)_{<\N/8}|\\
\nn&\qquad \qquad \qquad \qquad\times \left|A^{\sig_1}\widehat{\theta}_l(t,\xi,\beta)_{\N}\right| \left|A^{\sig_1}\overline{\widehat{\theta}}_k(t,\eta,\al)\right|d\xi d\eta\\
\nn\les& t^{-3}\sum_{k,l,\al,\beta\in\Z}\int_{\eta,\xi} e^{c\lm(t) |k-l,\eta-\xi,\al-\beta|^s}\frac{\la k-l,\eta-\xi,\alpha-\beta\ra ^7}{\JB{t}^3}|\widehat{\theta}_{\ne}(t,k-l,\eta-\xi,\al-\beta)_{<\N/8}|\\
\nn&\qquad \qquad \qquad \qquad\times \left|\widehat{A^{\sig_1}\theta}_l(t,\xi,\beta)_{\N}\right| \left|\overline{\widehat{A^{\sig_1}\theta}}_k(t,\eta,\al)\right|d\xi d\eta\\
\nn\les&\frac{\norm{\theta}_{\mathcal{G}^{\lambda,\sigma_5}}}{t^3} \norm{ t^{-\fr32}\theta_{\sim \N}}_{\sigma_1} \norm{ t^{-\fr32}\theta_{\N}}_{\sigma_1}\lesssim \frac{\epsilon}{\JB{t}^3} \norm{ t^{-\fr32}\theta_{\sim \N}}_{\sigma_1} \norm{ t^{-\fr32}\theta_{\N}}_{\sigma_1}.
\end{align}
\subsection{Remainders} In this subsection, we shall provide estimates corresponding to $\mathcal{R}_{0}$ and $\mathcal{R}_{\neq}$. Recall that from \eqref{paraproduct decomposition nonzero mode}, upon replacing $\neq$ by $0$, we obtain

\begin{align*}
\mathcal{R}_{0}&= 2\pi t^{-3}\sum_{\N\in \mathbb{D}}\sum_{\fr{\rm{N}}{8}\leq \rm{N}' \leq 8\rm{N}} \left\la u_{{0}_{\N}}\cdot\nb A^{\sig_1}\theta_{N'}-A^{\sig_1}(u_{{0}_\mathrm{N}}\cdot\nb \theta_{\N'}) , A^{\sig_1}\theta \right\ra\\&
=\mathcal{R}^1_{0}+\mathcal{R}^2_{0}.
\end{align*}

For the time being, let us focus on $\mathcal{R}^2_{0}$. The first term $\mathcal{R}^1_{0}$ can be treated in a similar manner and even simpler. Notice that on the Fourier side, we can write it as
\begin{align*}
\mathcal{R}^2_{0}&= -\fr{2\pi}{ t^{3}}\sum_{\N\in \mathbb{D}}\sum_{\fr{\rm{N}}{8}\leq \rm{N}' \leq 8\rm{N}} \sum_{k,\al,\beta \in \mathbb{Z}} \int_{\eta,\xi} \overline{\widehat{A^{\sigma_1}\theta_k}}(t,\eta,\alpha) {A^{\sigma_1}_k}(t,\eta,\alpha)\widehat{u_0}(t,\xi,\beta)_\N \cdot\widehat{\nabla \theta}_{k}(t,\eta-\xi,\al-\beta)_{\N'}d\xi d\eta.
\end{align*}
On the support of the integrand of  $\mathcal{R}^2_0$, it is true that $\fr{\N}{2}\le|\xi,\beta|\le\fr{3}{2} \N,\quad{\rm and}\quad \fr{\N'}{2}\le|k,\eta-\xi,\al-\beta|\le\fr{3}{2} \N'\quad{\rm with}\quad \fr{\N}{8}\le \N'\le8\N$.
Hence, 
\begin{equation}\label{comp xi beta}
\fr{1}{24}|k,\eta-\xi,\al-\beta|\le|\xi,\beta|\le24|k,\eta-\xi,\al-\beta|.
\end{equation}
Also, by  \eqref{triangle2}, we know that there exists $0<c<1$ for $ 0<s\leq 1$, so that \begin{align*}
    |k,\eta,\alpha|^s \le  c|\xi,\beta|^s +c|k,\eta-\xi,\alpha-\beta|^s,
\end{align*}
which gives that  
\[
A^{\sigma_1}_k(t,\eta,\alpha)\lesssim \textup{J}_k(t,\eta,\alpha)e^{c\lambda(t)|\xi,\beta|^s}e^{c\lambda(t)|k,\eta-\xi,\al-\beta|^s}|\xi,\beta||k,\eta-\xi,\al-\beta|^{\sigma_1-1}.
\]
But, via the definition of \textup{J} \eqref{multiplier J} and the total growth in Lemma~\ref{total growth}, we can infer that 
\[
\textup{J}_k(t,\eta,\alpha) \lesssim e^{\mu\sqrt{|\iota(k,\eta,\alpha)|}}\lesssim e^{\mu\sqrt{|k,\eta,\alpha|}}\lesssim e^{\mu(\sqrt{|k,\eta-\xi,\alpha-\beta|}+\sqrt{|\xi,\beta|})}. 
\]
Hence, via \eqref{exp-u} and the bootstrap hypotheses
\begin{align*}
|\mathcal{R}^2_{0}|&\lesssim t^{-3}\sum_{\N\in \mathbb{D}}\sum_{\fr{\rm{N}}{8}\leq \rm{N}' \leq 8\rm{N}} \sum_{\substack{k,\al,\beta \in \mathbb{Z}\\ l=0}} \int_{\eta,\xi} |\overline{\widehat{A^{\sigma_1}\theta_k}}(t,\eta,\alpha)|
e^{\lambda(t)|\xi,\beta|^s}|\xi,\beta||\widehat{u_0}(t,\xi,\beta)_\N| \\&
\qquad \qquad \qquad \times e^{\lambda(t)|k,\eta-\xi,\al-\beta|^s} |k,\eta-\xi,\al-\beta|^{\sigma_1-1}|\widehat{\nabla \theta}_{k}(t,\eta-\xi,\al-\beta)_{\N'}|d\xi d\eta\\&
\lesssim \sum_{\N\in \mathbb{D}}\sum_{\fr{\rm{N}}{8}\leq \rm{N}' \leq 8\rm{N}} \norm{t^{-\fr32}\theta}_{\sigma_1} {  \|\Pe_{\ne}^z\theta_{{0}_\N}\|_{\mathcal{G}^{\lm,\sig_6}}} \norm{t^{-\fr12} \theta_{\N'}}_{\mathcal{G}^{\lambda,\sigma_3}}\lesssim \fr{\epsilon^3}{\JB{t}^2}.
\end{align*}
Let us now consider the remainder arising from the non-zero modes
\begin{align*}
\mathcal{R}_{\ne}&= 2\pi t^{-3}\sum_{\N\in \mathbb{D}}\sum_{\fr{\rm{N}}{8}\leq \rm{N}' \leq 8\rm{N}} \left\la u_{{\ne}_{\N}}\cdot\nb A^{\sig_1}\theta_{\N'}-A^{\sig_1}(u_{{\ne}_{\N}}\cdot\nb \theta_{\N'}) , A^{\sig_1}\theta \right\ra\\&
=\mathcal{R}^1_{\ne}+\mathcal{R}^2_{\ne}.
\end{align*}
As before, we focus on $\mathcal{R}^2_{\ne}$. $\mathcal{R}^1_{\ne}$ can be estimated in a similar way. Recall first that, 
\begin{align*}
\mathcal{R}^2_{\ne}&= -2\pi t^{-3}\sum_{\N\in \mathbb{D}}\sum_{\fr{\rm{N}}{8}\leq \rm{N}' \leq 8\rm{N}} \sum_{k,l,\al,\beta \in \mathbb{Z}}\int_{\eta,\xi} \overline{\widehat{A^{\sigma_1}\theta_k}}(t,\eta,\alpha) A^{\sigma_1}_k(t,\eta,\alpha)\widehat{u_{\ne}}(t,l,\xi,\beta)_\N\\
&\qquad\qquad\qquad\qquad
\cdot\widehat{\nabla \theta}_{k-l}(t,\eta-\xi,\al-\beta)_{\N'}d\xi d\eta.
\end{align*}
Undergoing the same estimates as before  (now with $l\neq 0$) and utilizing decay rate of $\widehat{u}_{\neq}$ \eqref{u_ne-decay}, we arrive at 
\[
|\mathcal{R}^2_{\ne}|\lesssim \frac{\epsilon^3}{\JB{t}^3}.
\]

\section{Energy estimates of the zero mode}\label{sec: main-2}

\subsection{Estimates of the source term}
Let us start by analyzing the term ${\bf S}_{\sig_1-1}$. Recall from \eqref{commutator and source} that 
\[
{\bf S}_{\sig_1-1}=
-\left\la A^{\sig_1-2}_0B\left( u_{\ne}\cdot\nb_{xyz}\theta_{\ne} \right)_0, A^{\sig_1-2}_0B\theta_0 \right\ra.
\]
We split it into two parts
\[
{\bf S}_{\sig_1-1}={\bf S}_{\sig_1-1}^{\rm HL}+{\bf S}_{\sig_1-1}^{\rm LH},
\]
where 
\begin{align}
{\bf S}_{\sig_1-1}^{\rm HL}\nn=&-\sum_{\substack{\N\in \mathbb{D},\\ \N\ge8}}\sum_{\substack{k\ne0\\ \al,\beta\in\Z}}\int_{\eta,\xi} \left(\widehat{A^{\sig_1}u}_{\ne}\right)(t,k,\xi,\beta)_{\N}\cdot \widehat{\nb\theta}_{\ne}(t,-k,\eta-\xi,\al-\beta)_{<\N/8}\\
\nn&\qquad \qquad \qquad \qquad \times \fr{A_0^{\sig_1-2}(t,\eta,\al)}{A^{\sig_1}_k(t,\xi,\beta)}\la |\eta|^{\fr12},\al\ra\left(\overline{\reallywidehat{A^{\sig_1-2}B\theta}}_0\right)(t,\eta,\al) d\eta d\xi,\\
{\bf S}_{\sig_1-1}^{\rm LH}\nn=&-\sum_{\N\in \mathbb{D}}\sum_{\substack{k\ne0\\ \al,\beta\in\Z}}\int_{\eta,\xi} \widehat{u}_{\ne}(t,-k,\eta-\xi,\al-\beta)_{<16\N}\cdot \left(\reallywidehat{A^{\sig_1-1}\nb\theta}_{\ne}\right)(t, k, \xi, \beta)_{\N}\\
\nn&\qquad \qquad \qquad \qquad\times \fr{A_0^{\sig_1-2}(t,\eta,\al)}{A^{\sig_1-1}_k(t,\xi,\beta)}\la |\eta|^{\fr12},\al\ra\left(\overline{\reallywidehat{A^{\sig_1-2}B\theta}}_0\right)(t,\eta,\al) d\eta d\xi,
\end{align}
 and the paraproduct decomposition given above is only for $(y,z)$ variables.
\subsubsection{Treatment of high-low interactions}
On the support of the integrand of ${\bf S}_{\sig_1-1}^{\rm HL}$, we have the following frequency localizations 
\[
\fr{\N}{2}\le |\xi,\beta|\le \fr32\N,\quad |\eta-\xi, \al-\beta|\le\fr{3}{4}\fr{\N}{8},
\]
and hence
\[
|\eta-\xi,\al-\beta|\le\fr{3}{16}|\xi,\beta|, \quad \fr{13}{16}|\xi,\beta|\le|\eta,\al|\le\fr{19}{16}|\xi,\beta|.
\]
Noting that
\begin{equation}\label{ratio of B}
\fr{\la |\eta|^\fr12,\alpha\ra}{\la |\xi|^{\fr12},\beta\ra}=\fr{\sqrt{1+|\eta|+\al^2}}{\sqrt{1+|\xi|+\beta^2}}\les\la \eta-\xi, \al-\beta\ra,
\end{equation}
then we have
\begin{align}
\nn\fr{A_0^{\sig_1-2}(t,\eta,\al)}{A^{\sig_1}_k(t,\xi,\beta)}\la |\eta|^{\fr12}, \al\ra=&\fr{e^{\lm(t) |\eta,\al|^s} |\eta,\al|^{\sig_1-2}{\rm J}_0(t,\eta,\al)}{e^{\lm(t) |k,\xi,\beta|^s} |k,\xi,\beta|^{\sig_1}{\rm J}_k(t,\xi,\beta)}\la |\eta|^{\fr12}, \al\ra\\
\nn\les&e^{c\lm(t)|-k,\eta-\xi,\al-\beta|^s}{  \fr{\la |\xi|^{\fr12}, \beta\ra}{\la \xi,\beta\ra^2}}\fr{{\rm J}_0(t,\eta,\al)}{{\rm J}_k(t,\xi,\beta)}\la \eta-\xi, \al-\beta\ra.
\end{align}
Combining this with the upper bound of $|\widehat{u}_{\neq}|$ in \eqref{u-hat}, we can say that
\begin{align}
{\bf S}_{\sig_1-1}^{\rm HL}\nn&\les\sum_{\N\in \mathbb{D}}\sum_{\substack{\al, \beta\in\Z\\ k\neq 0}}\int_{\eta,\xi} ({\mathds{1}}_{t\in{\textup{I}}_{k,\xi}}+{\mathds{1}}_{t\notin{\textup{I}}_{k,\xi}})\left|t^{-\fr32}\reallywidehat{A^{\sig_1}\theta}_{\ne}(t, k,\xi, \beta)_\N\right|\\
\nn&\qquad \qquad \times t^{\fr32}{  \fr{\la |\xi|^{\fr12},\beta\ra}{\la\xi,\beta\ra^2}}\fr{\left(t|\beta|+|\xi|+|k,\beta|\right) |k,\beta|}{(k^2+(\xi-k t)^2+\beta^2)^2}   {   \fr{{\rm J}_0(t,\eta,\al)}{{\rm J}_k(t,\xi,\beta)}} {\mathds{1}}_{|\eta-\xi,\al-\beta|\le \fr{3}{16}|\xi,\beta|}\\
\nn&\qquad \qquad \times e^{\mathbf{c}\lm(t)|\eta-\xi,\al-\beta|^s} {|\reallywidehat{\la\nb\ra^2\theta}_{\ne}(t,-k,\eta-\xi,\al-\beta)_{<\N/8}|}\left|\reallywidehat{A^{\sig_1-2}B\theta}_0(t,\eta,\al)\right|d\eta d\xi
\\&={\bf S}_{\sig_1-1}^{\rm HL;R}+{  {\bf S}_{\sig_1-1}^{\rm HL;NR}}.
\end{align}

If $(k,\xi,\beta)\in \mathfrak{U}$, then \eqref{approximate} holds. This, together with Lemma \ref{lem-ratio-J} and Remark \ref{rem-2}, lead us to the following inequality  
\begin{align*}
    &t^{\fr32}\fr{\la |\xi|^{\fr12}, \beta\ra}{\la\xi,\beta\ra^2}\fr{\left(t|\beta|+|\xi|+|k,\beta|\right) |k,\beta|}{(k^2+(\xi-k t)^2+\beta^2)^2}\fr{{\rm J}_0(t,\eta,\al)}{{\rm J}_k(t,\xi,\beta)}\\
    &\qquad \qquad    \lesssim  t^{\fr32}\fr{\la |\xi|^{\fr12}, \beta\ra}{\la\xi,\beta\ra^2}\fr{\left(t|\beta|+|\xi|+|k,\beta|\right) |k,\beta|}{k^4+\xi^4+k^4 t^4+\beta^4}e^{2\mu|-k,\eta-\xi,\al-\beta|^{\fr12}}
    \lesssim \fr{|\beta|^{s/2}|\alpha|^{s/2}\JB{k}^2}{\JB{t}^{\fr12+s}}e^{2\mu|-k,\eta-\xi,\al-\beta|^{\fr12}}.
\end{align*}
Thus, it suffices to consider the case where $(k,\xi,\beta)\in \mathfrak{U}^c$ and $t>10$. Firstly, we consider the case when $t\notin\textup{I}_{k,\xi}$. Note that $(k,\xi,\beta)\in \mathfrak{U}^c$ implies that  $k\xi>0, t\approx \fr{\xi}{k}$, and $t\le2|\xi|$. Meanwhile,  Remark \ref{rem-1} and $t\notin {\textup {I}}_{k,\xi}$ imply that one of the following holds:
\begin{enumerate}
    \item $t\in{\rm I}_{k,\xi}\ne\emptyset$, but $|k|>\fr{\sqrt{|\xi|}}{4}$;
    \item ${\rm I}_{k,\xi}=\emptyset$;
    \item $t\notin{\rm I}_{k,\xi}\ne\emptyset$.
\end{enumerate}
Thus, for $k\ne0$, if  $(k,\xi,\beta)\in \mathfrak{U}^c$ and $t\notin {\textup{I}}_{k,\xi}$, then either $|k|\gtrsim\sqrt{|\xi|}\gtrsim 
\sqrt{t}$ or $|t-\fr{\xi}{k}|\gtrsim\fr{|\xi|}{k^2}$ holds. In either case, it is always true that
\[
k^4\left(1+\Big|t-\fr{\xi}{k}\Big|^2\right)\gtrsim t^2.
\]
Consequently,
\begin{align}\label{G-2 req-1}
    t^{\fr32}\fr{\la |\xi|^{\fr12}, \beta\ra}{\la\xi,\beta\ra^2}\fr{\left(t|\beta|+|\xi|+|k,\beta|\right) |k,\beta|}{(k^2+(\xi-k t)^2+\beta^2)^2}
    \lesssim  |k|^4 t^{\fr52}\fr{\la |\xi|^{\fr12}, \beta\ra}{\la\xi,\beta\ra^2}\fr{(|\xi|+|k,\beta|)|k,\beta|}{t^4+\beta^4}\lesssim \fr{|\beta|^{s}}{\JB{t}^{\fr12+s}}\JB{k}^6. 
\end{align}
Hence, we get
\[
{\bf S}_{\sig_1-1}^{\rm HL;NR}\lesssim \fr{\epsilon}{t^{\fr12+s}}\left\||\nb |^{\fr{s}{2}}\fr{A^{\sig_1}}{t^{\fr32}}\theta_{\ne}\right\|_{L^2}\left\||\nb |^{\fr{s}{2}}A^{\sig_1-2} _0B\theta_0\right\|_{L^2}.
\]

Let us  now focus on the case $t\in{\textup{I}}_{k,\xi}$. In particular, now we have $|k|\le|\xi|$ and $|\beta|\le\fr{1}{100}|\xi|$, and hence $\iota(k,\xi,\beta)=\xi$. Furthermore, combining the fact $|\beta|\le\fr{1}{100}|\xi|$  with the frequency restriction $|\eta-\xi,\al-\beta|\le \fr{3}{16}|\xi,\beta|$ yields
\begin{equation}\label{eta aprx xi}
    \fr{1297}{1600}|\xi|\le|\eta|\le\fr{1903}{1600}|\xi|, \quad {\rm and } \quad  |\al|\le \fr{319}{1297}|\eta|.
\end{equation}
Thus $\iota(0,\eta,\al)=\eta$. Then by Lemma \ref{lem-ratio-J}, one easily deduces that
\be\label{good1}
\fr{{\rm J}_0(t,\eta,\al)}{{\rm J}_k(t,\xi,\beta)}{\mathds{1}}_{t\in\textup{I}_{k,\xi}}\les \fr{k^2(1+|t-\fr{\xi}{k}|)}{|\xi|}e^{2\mu|-k,\eta-\xi,\al-\beta|^{\fr12}}.
\ee

\par
\noindent{\bf Case 1: $|\xi|^{\fr12}\le 10 (1+ |\beta|)$.} Knowing $t\in\textup{I}_{k,\xi}$, we can infer that $t\approx\fr{|\xi|}{|k|}$. Thanks to \eqref{up-growth}, we have
\begin{align}\label{up-growth2}
t^{\fr32}\fr{\la |\xi|^{\fr12}, \beta\ra}{\la\xi,\beta\ra^2}&\fr{\left(t|\beta|+|\xi|+|k,\beta|\right) |k,\beta|}{(k^2+(\xi-k t)^2+\beta^2)^2}\nn\les \fr{|\fr{\xi}{k}|^{\fr32}\la\beta\ra}{\la\xi,\beta\ra^2}\fr{|\fr{\xi}{k}|}{k^2+(\xi-k t)^2+\beta^2}\\
\les&\fr{{  |\xi|^{\fr12}}\la\beta\ra}{|k|^{\fr{5}{2}}(k^2+(\xi-k t)^2+\beta^2)}
\les\fr{{   |\xi|^{\fr12}}(|\fr{\beta}{k}|+1)}{|k|^{\fr{7}{2}}(1+(t-\fr{\xi}{k})^2+|\fr{\beta}{k}|^2)},
\end{align}
and
\begin{align}\label{up-growth2'}
\fr{|\xi|^{\fr12}(|\fr{\beta}{k}|+1)}{|k|^{\fr{7}{2}}\Big(1+(t-\fr{\xi}{k})^2+|\fr{\beta}{k}|^2\Big)} \nn&{\fr{|k|^2}{|\xi|}\left(1+\left|t-\fr{\xi}{k}\right|\right)}\\
\nn\les&\fr{1}{|k|^\fr32|\xi|^\fr12}\fr{(|\fr{\beta}{k}|+1)}{\sq{1+(t-\fr{\xi}{k})^2+|\fr{\beta}{k}|^2}}\les\fr{1}{|k|^\fr32|\xi|^\fr12}=\fr{1}{|k|^\fr32}\fr{|\xi|^s}{|\xi|^{\fr12+s}}\\
\les &\fr{|\xi|^{\fr{s}{2}}|\eta|^{\fr{s}{2}}}{t^{\fr12+s}}.
\end{align}

\noindent{\bf Case 2: $(1+|\beta|)<\fr{1}{10}|\xi|^{\fr12}$.} 
Now, in this case, we obtain an analog of \eqref{up-growth2}, namely
\begin{align}\label{up-growth3}
t^{\fr32}\fr{\la |\xi|^{\fr12}, \beta\ra}{\la\xi,\beta\ra^2}\fr{\left(t|\beta|+|\xi|+|k,\beta|\right) |k,\beta|}{(k^2+(\xi-k t)^2+\beta^2)^2}\nn\les& \fr{|\fr{\xi}{k}|^{\fr32}|\xi|^{\fr12}}{\la\xi,\beta\ra^2}\fr{|\fr{\xi}{k}|}{k^2+(\xi-k t)^2+\beta^2}\\
\les&\fr{{   |\xi|}}{|k|^{\fr{9}{2}}(1+(t-\fr{\xi}{k})^2+|\fr{\beta}{k}|^2)}.
\end{align}
If $t\in {\rm I}_{k,\xi}\cap{\rm I}_{k,\eta}$, then the right hand side of \eqref{up-growth3} multiplied the ratio in \eqref{good1}  can be estimated as follows
\begin{align}\label{up-growth3'}
\fr{{   |\xi|}}{|k|^{\fr{9}{2}}\Big(1+(t-\fr{\xi}{k})^2+|\fr{\beta}{k}|^2\Big)}{  \fr{|k|^2}{|\xi|}\left(1+\left|t-\fr{\xi}{k}\right|\right)}\les&\fr{1}{|k|^\fr52}\fr{1}{\sq{1+|t-\fr{\xi}{k}|}}\fr{1}{\sq{1+|t-\fr{\eta}{k}|}}\fr{\sq{1+|t-\fr{\eta}{k}|}}{\sq{1+|t-\fr{\xi}{k}|}}\\
\nn\les &\sqrt{\fr{\pr_tw_k(t,\xi)}{w_k(t,\xi)}}\sqrt{\fr{\pr_tw_0(t,\eta)}{w_0(t,\eta)}}\la \eta-\xi\ra^{\fr{1}{2}}.
\end{align}
If $t\in {\rm I}_{k,\xi}\cap{\rm I}^c_{k,\eta}$, then $\fr{|\eta|}{k^2}\fr{1}{1+|t-\fr{\eta}{k}|}\les1$, and hence using \eqref{eta aprx xi} we get
\begin{align}\label{up-growth3''}
\fr{{   |\xi|}}{|k|^{\fr{9}{2}}\Big(1+(t-\fr{\xi}{k})^2+|\fr{\beta}{k}|^2\Big)}{  \fr{|k|^2}{|\xi|}\left(1+\left|t-\fr{\xi}{k}\right|\right)}&\les\fr{1}{|k|^\fr52}\fr{1}{1+|t-\fr{\eta}{k}|}\fr{1+|t-\fr{\eta}{k}|}{1+|t-\fr{\xi}{k}|}\les\fr{1}{|k|^{\fr32}}\fr{|k|}{|\xi|}\fr{|\xi|}{|\eta|}\la \eta-\xi\ra\\
\nn&\les \fr{1}{|k|^{\fr32}} \fr{1}{t}\la\eta-\xi\ra\les \fr{1}{|k|^{\fr32+s}}\fr{|\xi|^{\fr{s}{2}}|\eta|^{\fr{s}{2}}}{t^{1+s}}.
\end{align}

After combining estimates in both cases together with the bootstrap hypotheses, we may infer
\begin{align}
{\bf S}_{\sig_1-1}^{\rm HL;R}\nn\les&
\fr{\epsilon}{t^{\fr12+s}}\left\||\nb |^{\fr{s}{2}}\fr{A^{\sig_1}}{t^{\fr32}}\theta_{\ne}\right\|_{L^2}\left\||\nb |^{\fr{s}{2}}A^{\sig_1-2} _0B\theta_0\right\|_{L^2}\\&
\qquad + \epsilon\left\| \sqrt{\frac{\partial_t w}{w}}\fr{A^{\sig_1}}{t^{\fr32}}\theta_{\ne}\right\|_{L^2}\left\|\sqrt{\frac{\partial_t w}{w}}A^{\sig_1-2} _0B\theta_0\right\|_{L^2}.
\end{align}

\subsubsection{Treatment of low-high interactions}
On the support of the integrand of ${\bf S}^{\rm LH}_{\sig_1-1}$, the following frequency localization holds
\[
\fr{\N}{2}\leq |\xi,\beta|\le\fr32\N,\quad |\eta-\xi,\al-\beta|\le\fr34\times 16\N,
\]
and hence
\[
|\eta-\xi,\al-\beta|\le 24|\xi,\beta|.
\]
Thanks to \eqref{triangle1} and
\eqref{ratio-J}, we are led to
\[
\fr{A_0^{\sig_1-2}(t,\eta,\al)}{A^{\sig_1-1}_k(t,\xi,\beta)}\la |\eta|^{\fr12},\al\ra\les e^{c\lm(t) |\eta-\xi,\al-\beta|^s}.
\]
which in combination with \eqref{u_ne-decay} and the bootstrap hypotheses yields
\begin{equation}
    \begin{aligned}
        {\bf S}_{\sig_1-1}^{\rm LH}\les&\sum_{\substack{k\ne0\\ \al,\beta\in\Z}}\int_{\eta,\xi} \frac{\langle k,\eta-\xi,\al-\beta\rangle^6}{\langle t\rangle^3}|\reallywidehat{\theta}_{\ne}(t,k,\eta-\xi,\al-\beta)| \left|\left(\widehat{A^{\sig_1}\theta}_{\ne}\right)(t, k, \xi, \beta)\right|\\
&\qquad \qquad \times e^{\mathbf{c}\lm(t) |\eta-\xi,\al-\beta|^s}\left|\left(\reallywidehat{A^{\sig_1-2}B\theta}_0\right)(t,\eta,\al)\right| d\eta d\xi\\
\les&\fr{1}{t^{\fr32}}\|\theta\|_{\mathcal{G}^{\lm,\sig_5}}\|B\theta_0\|_{\sig_1-2}\|t^{-\fr32}\theta\|_{\sig_1}\les\fr{\epsilon}{t^{\fr32}}\|B\theta_0\|_{\sig_1-2}\|t^{-\fr32}\theta\|_{\sig_1}.
    \end{aligned}
\end{equation}

\subsection{Estimates of the commutator}

\subsubsection{Treatment of the transport term}
Note that from ${  {\bf com}_{\sig_1-1}}$ in \eqref{commutator and source}, we derive the following useful equation
\begin{align*}
 \frac{A^{\sig_1-2}_0(t,\eta,\alpha)B(\eta,\alpha)}{A^{\sig_1-2}_0(t,\xi,\beta)B(\xi,\beta)}-1&
 =\frac{B(\eta,\al)}{B(\xi,\beta)}\bigg[\left(e^{\lm(t)(|\eta,\al|^s-| \xi,\beta|^s)}-1\right)\\&\qquad \qquad +e^{\lm(t)(|\eta,\al|^s-| \xi,\beta|^s)}\left(\fr{\mathrm{J}_0(t,\eta,\al)}{\mathrm{J}_0(t,\xi,\beta)}-1\right)\fr{\la \eta,\al \ra^{\sigma_1-2}}{\la \xi,\beta\ra^{\sigma_1-2}}\\
\nn&\qquad\qquad +
e^{\lm(t)(|\eta,\al|^s- |\xi,\beta|^s)}\left(\fr{\la\eta,\al\ra^{\sigma_1-2}}{\la\xi,\beta\ra^{\sigma_1-2}}-1\right)\bigg]\\
& \quad
+\frac{B(\eta,\alpha)-B(\xi,\beta)}{B(\xi,\beta)}.
\end{align*}
With this in mind, we can decompose $\tilde{\T}_{\N}$ in \eqref{paraproduct} as follows,
\[
\tilde{\T}_{\N}=\tilde{\T}^{(1)}_{\N}+\tilde{\T}^{(2)}_{\N}+\tilde{\T}^{(3)}_{\N}+\tilde{\T}^{(4)}_{\N}.
\]
Below, we present estimates associated with each $\tilde{\T}^{(i)}_{\N}, i=1,2,3,4.$ 

\underline{Treatment of $\tl{\T}^{(1)}_\N$.}
First of all, notice that on support of the integrand bellow, it is still true that $|\eta,\alpha|\approx|\xi,\beta|.$ Using $|e^x-1|\leq |x|e^{|x|}$, Lemma \ref{exponent ineq}, \eqref{ratio of B}, \eqref{eq: tlu_0} and the bootstrap hypotheses, we can infer that 
\begin{align*}
|\tilde{\T}^{(1)}_{\N}|&
\lesssim \sum_{\al,\beta\in\Z}\int_{\eta,\xi} e^{c\lm(t)( |\eta-\xi,\al-\beta|^s)}\fr{|\al-\beta|}{(\eta-\xi)^2+(\al-\beta)^2}|\widehat{\Pe_{\ne}^z\theta_0}(t,\eta-\xi,\al-\beta)_{<\N/8}||\xi,\beta|^{\fr{s}{2}} |\eta,\alpha|^{\fr{s}{2}}\\
\nn&\qquad \qquad \times| \reallywidehat{A^{\sig_1-2}B\theta}_0(t,\xi,\beta)_{\N}| |A^{\sig_1-2}_0(t,\eta,\alpha)B(\eta,\alpha) \widehat{\theta}_0(t,\eta,\alpha)|\;d\xi d\eta\\&
\lesssim \|\Pe_{\ne}^z \theta_0\|_{\mathcal{G}^{\lm,\sig_6}} \||\nabla|^{\fr{s}{2}}B\theta_{{0}_{\N}}\|_{\sigma_1-2} \||\nabla|^{\fr{s}{2}}B\theta_{{0}_{\sim \N}}\|_{\sigma_1-2}\lesssim \frac{\epsilon}{\JB{t}^3} \||\nabla|^{\fr{s}{2}}B\theta_{{0}_{\N}}\|_{\sigma_1-2} \||\nabla|^{\fr{s}{2}}B\theta_{{0}_{\sim \N}}\|_{\sigma_1-2}.
\end{align*}

\underline{Treatment of $\tl{\T}^{(2)}_\N$.} Here, we are using the estimates in \eqref{ratio of B}, \eqref{ratio-J}, and the fact that $|\eta,\alpha|\approx |\xi,\beta|$. Similar to the approach performed in \eqref{transport0-2-decomp}, we decompose $\tl{\T}^{(2)}_\N$ as follows,
\begin{align*}
\tilde{\T}^{(2)}_{\N}&= \sum_{\al,\beta\in\Z}\int_{\eta,\xi} [{\mathds{1}}^{\mf{S}}+{\mathds{1}}^{\mf{L}}]  \frac{B(\eta,\al)}{B(\xi,\beta)}e^{\lm(t)(|\eta,\al|^s-| \xi,\beta|^s)}\left(\fr{\mathrm{J}_0(t,\eta,\al)}{\mathrm{J}_0(t,\xi,\beta)}-1\right)\\
&\qquad\qquad\times\fr{ \la \eta,\al\ra^{\sigma_1-2}}{\la \xi,\beta\ra^{\sigma_1-2}} \widehat{\tl{u}}_{0}(t,\eta-\xi,\alpha-\beta)_{<\N/8}\cdot i(\xi,\beta) \\
&\qquad \qquad \times\reallywidehat{A^{\sig_1-2}B\theta}_0(t,\xi,\beta)_{\N}  \overline{\reallywidehat{A^{\sig_1-2}B\theta}_0}(t,\eta,\alpha)\;d\xi d\eta=:\tilde{\T}^{(2),\mf{S}}_{\N}+\tilde{\T}^{(2),\mf{L}}_{\N}.
\end{align*}
Let us now focus on the short-time regime. We know, by Lemma \ref{short time}, \eqref{ratio of B}, and the bootstrap hypotheses, that
\begin{align*}
\left|\tilde{\T}^{(2),\mf{S}}_{\N}\right|\nn\les&  \sum_{\al,\beta\in\Z}\int_{\eta,\xi} {\mathds{1}}^\mf{S} \bigg( \frac{B(\eta,\al)}{B(\xi,\beta)}e^{\lm(t)(|\eta,\al|^s-| \xi,\beta|^s)}\left(\fr{\mathrm{J}_0(t,\eta,\al)}{\mathrm{J}_0(t,\xi,\beta)}-1\right) \fr{ \la \eta,\al\ra^{\sigma_1-2}}{\la \xi,\beta\ra^{\sigma_1-2}} \bigg)|\widehat{\tl{u}}_{0}(t,\eta-\xi,\alpha-\beta)_{<\N/8}|\\
&\qquad \qquad \times |\xi,\beta| |\reallywidehat{A^{\sig_1-2}B\theta_0}(t,\xi,\beta)_{\N}| |\reallywidehat{A^{\sig_1-2}B\theta_0}(t,\eta,\alpha)|\;d\xi d\eta\\
\les&\sum_{\al,\beta\in\Z}\int_{\eta,\xi} {\mathds{1}}^\mf{S} \frac{e^{\mathbf{c}\lambda(t)|\eta-\xi,\alpha-\beta|^s}}{(|\eta|+|\xi|+|\alpha|+|\beta|)^{\fr12}}|\widehat{\tl{u}}_{0}(t,\eta-\xi,\alpha-\beta)_{<\N/8}||\xi,\beta|\\& \qquad \qquad \times |\reallywidehat{A^{\sig_1-2}B\theta_0}(t,\xi,\beta)_{\N}| |\reallywidehat{A^{\sig_1-2}B\theta_0}(t,\eta,\alpha)|\;d\xi d\eta\\
\les&\sum_{\al,\beta\in\Z}\int_{\eta,\xi} {\mathds{1}}^\mf{S} e^{\mathbf{c}\lambda(t)|\eta-\xi,\alpha-\beta|^s}\fr{|\al-\beta|}{(\eta-\xi)^2+(\al-\beta)^2}|\widehat{\Pe_{\ne}^z\theta_0}(t,\eta-\xi,\al-\beta)_{<\N/8}| |\xi,\beta|^{\fr12}\\&\qquad \qquad \times |\reallywidehat{A^{\sig_1-2}B \theta_0}(t,\xi,\beta)_{\N}| |\reallywidehat{A^{\sig_1-2}B \theta_0}(t,\eta,\alpha)|\;d\xi d\eta\\
\lesssim&\|\Pe_{\ne}^z \theta_0\|_{\mathcal{G}^{\lm,\sig_6}} \||\nabla|^{\fr{s}{2}}B\theta_{{0}_{\N}}\|_{\sigma_1-2} \||\nabla|^{\fr{s}{2}}B\theta_{{0}_{\sim \N}}\|_{\sigma_1-2}\lesssim \frac{\epsilon}{\JB{t}^3} \||\nabla|^{\fr{s}{2}}B\theta_{{0}_{\N}}\|_{\sigma_1-2} \||\nabla|^{\fr{s}{2}}B\theta_{{0}_{\sim \N}}\|_{\sigma_1-2}.
\end{align*}
Let us now consider its counterpart (long-time regime) $\tilde{\T}^{(2),\mf{L}}_{\N}$. The integrand in this regime is supported on the time interval $t >\frac{1}{2}\min{(\sqrt{|\iota_1|},\sqrt{|\iota_2|})}$. 
We have $|\xi,\beta|\lesssim |\xi,\beta|^{\frac{s}{2}}|\eta,\alpha|^{\frac{s}{2}}t^{2-2s}$. Hence,
upon using Lemma~\ref{lem-ratio-J}, \eqref{ratio of B}, \eqref{eq: tlu_0} and the bootstrap hypotheses, we obtain
\begin{align*}
\left| \tilde{\T}^{(2),\mf{L}}_{\N}\right| &\lesssim \sum_{\al,\beta\in\Z}\int_{\eta,\xi} {\mathds{1}}^{\mf{L}} e^{\mathbf{c}\lm(t)|\eta-\xi,\al-\beta|^s}|\widehat{\tl{u}}_{0}(t,\eta-\xi,\alpha-\beta)_{<\N/8}||\xi,\beta|\\
\nn& \qquad \qquad \qquad\times | \reallywidehat{A^{\sig_1-2}B\theta_0}(t,\xi,\beta)_{\N}| |\reallywidehat{A^{\sig_1-2}B\theta}_0(t,\eta,\alpha)|\;d\xi d\eta\\&
\lesssim \sum_{\al,\beta\in\Z}\int_{\eta,\xi} {\mathds{1}}^{\mf{L}} e^{\mathbf{c}\lm(t)|\eta-\xi,\al-\beta|^s}|\widehat{\tl{u}}_{0}(t,\eta-\xi,\alpha-\beta)_{<\N/8}||\xi,\beta|^{\frac{s}{2}}|\eta,\alpha|^{\frac{s}{2}}t^{2-2s}\\
\nn& \qquad \qquad \qquad \times | \reallywidehat{A^{\sig_1-2}B\theta_0}(t,\xi,\beta)_{\N}|  | \reallywidehat{A^{\sig_1-2}B\theta}_0(t,\eta,\alpha)|\;d\xi d\eta\\&
\lesssim \sum_{\al,\beta\in\Z}\int_{\eta,\xi} {\mathds{1}}^{\mf{L}}e^{\mathbf{c}\lm(t)|\eta-\xi,\al-\beta|^s} \fr{|\al-\beta|}{(\eta-\xi)^2+(\al-\beta)^2}|\widehat{\Pe_{\ne}^z\theta_0}(t,\eta-\xi,\al-\beta)_{<\N/8}|\xi,\beta|^{\frac{s}{2}}|\eta,\alpha|^{\frac{s}{2}}\\
\nn& \qquad \qquad \qquad \times t^{2-2s} | \reallywidehat{A^{\sig_1-2}B\theta_0}(t,\xi,\beta)_{\N}|  | \reallywidehat{A^{\sig_1-2}B\theta}_0(t,\eta,\alpha)|\;d\xi d\eta
\\&
\lesssim t^{2-2s} \|\Pe_{\ne}^z \theta_0\|_{\mathcal{G}^{\lm,\sig_6}} \||\nabla|^{\fr{s}{2}}B\theta_{{0}_{\N}}\|_{\sigma_1-2} \||\nabla|^{\fr{s}{2}}B\theta_{{0}_{\sim \N}}\|_{\sigma_1-2}\\
&\lesssim \frac{\epsilon}{\JB{t}^{2s+1}} \||\nabla|^{\fr{s}{2}}B\theta_{{0}_{\N}}\|_{\sigma_1-2} \||\nabla|^{\fr{s}{2}}B\theta_{{0}_{\sim \N}}\|_{\sigma_1-2}.
\end{align*}
This therefore completes the estimate for $\tilde{\T}^{(2)}_{\N}$. Now, we switch our attention to estimating $\tilde{\T}^{(3)}_{\N}$.

\underline{Treatment of $\tl{\T}^{(3)}_\N$.}
By the mean value theorem, we obtain
\begin{align*}
\nn&e^{\lm(t)( |\eta,\al|^s-| \xi,\beta|^s)}\left|\fr{ \la \eta,\al\ra^{\sigma_1-2}}{\la \xi,\beta \ra^{\sigma_1-2}}-1\right|\les e^{c\lm(t)| \eta-\xi,\al-\beta|^s}\fr{|\eta-\xi,\al-\beta|}{\la \xi,\beta\ra}.
\end{align*}
Hence, combining this estimate with \eqref{ratio of B} allows us to infer that 
\begin{align*}
\left|\tilde{\T}^{(3)}_{\N}\right|
\les&\sum_{\al,\beta\in\Z}\int_{\eta,\xi} e^{c\lm(t) |\eta-\xi,\alpha-\beta|^s} \fr{|\al-\beta|}{(\eta-\xi)^2+(\al-\beta)^2}|\widehat{\Pe_{\ne}^z\theta_0}(t,\eta-\xi,\al-\beta)_{<\N/8}|\\
&\qquad \qquad\times | \reallywidehat{A^{\sig_1-2}B\theta}_0(t,\xi,\beta)_{\N}| | \reallywidehat{A^{\sig_1-2}B\theta}_0(t,\eta,\alpha)|\;d\xi d\eta\\ \lesssim & \|\Pe_{\ne}^z \theta_0\|_{\mathcal{G}^{\lm,\sig_6}} \|B\theta_{{0}_{\N}}\|_{\sigma_1-2} \|B\theta_{{0}_{\sim \N}}\|_{\sigma_1-2}\lesssim \frac{\epsilon}{\JB{t}^3} \|B\theta_{{0}_{\N}}\|_{\sigma_1-2} \|B\theta_{{0}_{\sim \N}}\|_{\sigma_1-2}.
\end{align*}
Now, we move on and estimate $\tilde{\T}^{(4)}_{\N}$.

\underline{Treatment of $\tl{\T}^{(4)}_{\N}$.}

Notice that 
\[
\left|\frac{B(\eta,\alpha)-B(\xi,\beta)}{B(\xi,\beta)}\right|\lesssim \fr{\JB{\eta-\beta,\alpha-\beta}^2}{\JB{\xi,\beta}^\fr12}.
\]
Therefore, as done before, we obtain the following inequality
\begin{equation}
\begin{aligned}
\left|\tilde{\T}^{(4)}_{\N}\right|\nn\les& \sum_{\al,\beta\in\Z}\int_{\eta,\xi} \fr{\JB{\eta-\beta,\alpha-\beta}^2}{\JB{\xi,\beta}^\fr12}|\widehat{\tl{u}}_{0}(t,\eta-\xi,\alpha-\beta)_{<\N/8}||\xi,\beta|\\
&\qquad \qquad \times  | \reallywidehat{A^{\sig_1-2}B\theta}_0(t,\xi,\beta)_{\N}| | \reallywidehat{A^{\sig_1-2}B\theta}_0(t,\eta,\alpha)|\;d\xi d\eta
\\ \lesssim & \|\Pe_{\ne}^z \theta_0\|_{\mathcal{G}^{\lm,\sig_6}} \||\nabla|^\fr{s}{2}B\theta_{{0}_{\N}}\|_{\sigma_1-2} \||\nabla|^\fr{s}{2}B\theta_{{0}_{\sim \N}}\|_{\sigma_1-2}\lesssim \frac{\epsilon}{\JB{t}^3} \||\nabla|^\fr{s}{2}B\theta_{{0}_{\N}}\|_{\sigma_1-2} \||\nabla|^\fr{s}{2}B\theta_{{0}_{\sim \N}}\|_{\sigma_1-2}.
\end{aligned}
\end{equation}

\subsubsection{Treatment of the reaction term} Recalling \eqref{paraproduct}, we write
\[
\tl{\R}_\N=\tl{\R}_\N^{(1)}+\tl{\R}_\mathrm{N}^{(2)},
\]
where
\begin{align*}
    \tilde{\R}_\mathrm{N}^{(1)}
    =&-\sum_{\al,\beta\in\Z}\int_{\eta,\xi} A^{\sig_1-2}_0(t,\eta,\al)B(\eta,\al)\widehat{\tl{u}^y_{0}}(\xi,\beta)_\mathrm{N}\widehat{\pr_y \theta}_{0}(\eta-\xi,\al-\beta)_{<\N/8}  \left(\overline{\reallywidehat{A^{\sig_1-2}B\theta}}_0\right)(t,\eta,\al) d\xi d\eta\\
    &-\sum_{\al,\beta\in\Z}\int_{\eta,\xi} A_0^{\sig_1-2}(t,\eta,\al)B(\eta,\al)\widehat{\tl{u}^z_{0}}(\xi,\beta)_\mathrm{N}\widehat{\pr_z \theta}_{0}(\eta-\xi,\al-\beta)_{<\N/8}  \left(\overline{\reallywidehat{A^{\sig_1-2}B\theta}}_0\right)(t,\eta,\al) d\xi d\eta\\
    =&\tilde{\R}_\mathrm{N}^{(1,y)}+\tilde{\R}_\mathrm{N}^{(1,z)},
\end{align*}
and
\[
\tl{\R}_\N^{(2)}=\sum_{\al,\beta\in\Z}\int_{\eta,\xi} \widehat{\tl{u}_{0}}(\xi,\beta)_\mathrm{N}\cdot\reallywidehat{\nb A^{\sig_1-2}_0B \theta}_{0}(\eta-\xi,\al-\beta)_{<N/8}  \left(\overline{\reallywidehat{A^{\sig_1-2}B\theta}}_0\right)(t,\eta,\al) d\xi d\eta.
\]
More precisely, recalling \eqref{eq: tlu_0}, we further split $\tilde{\R}_\mathrm{N}^{(1,y)}$ as follows:
\begin{align*}
    \tilde{\R}_\mathrm{N}^{(1,y)}=&\sum_{\beta\ne0}\int_{\eta,\xi} A^{\sig_1-2}_0(t,\eta,\beta)B(\eta,\beta)\fr{\beta^2}{(\xi^2+\beta^2)^2}\widehat{\theta}_0(\xi,\beta)_\mathrm{N}\widehat{\pr_y \theta}_{0}(\eta-\xi,0)_{<\N/8}  \\
    &\qquad \qquad \times\left(\overline{\reallywidehat{A^{\sig_1-2}B\theta}}_0\right)(t,\eta,\beta) d\xi d\eta\\
    &+\sum_{\al\ne\beta,\beta\ne0}\int_{\eta,\xi} A^{\sig_1-2}_0(t,\eta,\al)B(\eta,\al)\fr{\beta^2}{(\xi^2+\beta^2)^2}\widehat{\theta}_0(\xi,\beta)_\mathrm{N}\widehat{\pr_y \theta}_{0}(\eta-\xi,\al-\beta)_{<\N/8}  \\
    &\qquad \qquad \times\left(\overline{\reallywidehat{A^{\sig_1-2}B\theta}}_0\right)(t,\eta,\al) d\xi d\eta=\tilde{\R}_{\N,0}^{(1,y)}+\tilde{\R}_{\N,\ne}^{(1,y)}.
\end{align*}
On the support of the integrand of $\tilde{\R}_{\N,0}^{(1,y)}$, there holds
\[
|\eta-\xi|\le\fr{3}{16}|\xi,\beta|,\quad{\rm and}\quad |\eta,\beta|\approx|\xi,\beta|.
\]
From this and \eqref{ratio of B}, we have
\begin{align*}
    \tilde{\R}_{\N,0}^{(1,y)}=&\sum_{\beta\ne0}\int_{\eta,\xi} \fr{\beta}{\xi^2+\beta^2}\left(\reallywidehat{A^{\sig_1-2}B\theta}_0\right)(t,\xi,\beta)_\mathrm{N} \fr{A^{\sig_1-2}_0(t,\eta,\beta)B(\eta,\beta)}{A^{\sig_1-2}_0(t,\xi,\beta)B(\xi,\beta)}\fr{\eta^2+\beta^2}{\xi^2+\beta^2}\widehat{\pr_y \theta}_{0}(\eta-\xi,0)_{<\N/8}  \\
    &\qquad \qquad \times \fr{\beta}{\eta^2+\beta^2}\left(\overline{\reallywidehat{A^{\sig_1-2}B\theta}}_0\right)(t,\eta,\beta) d\xi d\eta\\
    \les& \eps\left\|\pr_z\Delta_{yz}^{-1}B\theta_{0_\mathrm{N}}\right\|_{\sig_1-2}\left\|\pr_z\Delta_{yz}^{-1}B\theta_{0_{\sim \N}}\right\|_{\sig_1-2}.
\end{align*}
Using again \eqref{ratio of B}, we find that
\begin{align*}
\tilde{\R}_{\N,\ne}^{(1,y)}\les&\sum_{\al\ne\beta,\beta\ne0}\int_{\eta,\xi} \left|\left(\reallywidehat{A^{\sig_1-2}B\theta}_0\right)(t,\xi,\beta)_\mathrm{N}\right| \fr{A^{\sig_1-2}_0(t,\eta,\alpha)B(\eta,\alpha)}{A^{\sig_1-2}_0(t,\xi,\beta)B(\xi,\beta)}\widehat{\pr_y \theta}_{0}(\eta-\xi,\al-\beta)_{<\N/8}  \\
    &\times \left|\left({\reallywidehat{A_0^{\sig_1-2}B\theta}}_0\right)(t,\eta,\beta) \right|d\xi d\eta\\
    \les&\left\|B\theta_{0_\mathrm{N}}\right\|_{\sig_1-2}\left\|B\theta_{0_{\sim \N}}\right\|_{\sig_1-2}\|\mathbb{P}^z_{\ne}\theta_0\|_{\mathcal{G}^{\lambda,\sig_6}}
    \les\fr{\eps}{t^3}\left\|B\theta_{0_\mathrm{N}}\right\|_{\sig_1-2}\left\|B\theta_{0_{\sim \N}}\right\|_{\sig_1-2}.
\end{align*}

The term $\tilde{\R}_\mathrm{N}^{(1,z)}$ can be treated in the same manner as $\tilde{\R}_{\N,\ne}^{(1,y)}$.
\begin{align*}
    \tilde{\R}_\mathrm{N}^{(1,z)}=&-\sum_{\al\ne\beta,\beta\ne0}\int_{\eta,\xi}\fr{\xi\beta}{(\xi^2+\beta^2)^2}\left(\reallywidehat{A^{\sigma_1-2}B\theta_0}\right)(t, \xi,\beta)_\mathrm{N}\fr{A^{\sig_1-2}_0(t,\eta,\al)B(\eta,\al)}{A^{\sig_1-2}_0(t,\xi,\beta)B(\xi,\beta)}  \\
    &\times \widehat{\pr_z \theta}_{0}(\eta-\xi,\al-\beta)_{<\N/8}\left(\overline{\reallywidehat{A^{\sig_1-2}B\theta}}_0\right)(t,\eta,\al) d\xi d\eta\\
    \les&\fr{\eps}{t^3}\left\|B\theta_{0_\mathrm{N}}\right\|_{\sig_1-2}\left\|B\theta_{0_{\sim \N}}\right\|_{\sig_1-2}.
\end{align*}

As for $\tl{\R}_\N^{(2)}$, noting that $|\eta-\xi,\al-\beta|\le\fr{3}{16}|\xi,\al|$ holds on the support of the integrand in $\tl{\R}_\N^{(2)}$, one easily deduces that
\[
A_0^{\sig_1-2}(t,\eta-\xi,\al-\beta)B(\eta-\xi,\al-\beta)\les e^{c\lm(t)|\xi,\beta|^s}, \quad {\rm for\ \ some}\quad c\in(0,1).
\]
Then $\tl{\R}_\N^{(2)}$ can be bounded in the same manner as $\tl{\R}_\N^{(1)}$. We omit the details and state the result:
\begin{align*}
    |\tl{\R}_\N^{(2)}|\les \eps\left\|\pr_z\Delta_{yz}^{-1}B\theta_{0_\mathrm{N}}\right\|_{\sig_1-2}\left\|\pr_z\Delta_{yz}^{-1}B\theta_{0_{\sim \N}}\right\|_{\sig_1-2}+\fr{\eps}{t^3}\left\|B\theta_{0_\mathrm{N}}\right\|_{\sig_1-2}\left\|B\theta_{0_{\sim \N}}\right\|_{\sig_1-2}.
\end{align*}

\subsubsection{Treatment of Remainder $\tilde{\mathcal{R}}$}
Let us start by writing down the term of interest here
 \begin{align*}
 \tilde{\mathcal{R}}&= 2\pi \sum_{\N\in \mathbb{D}}\sum_{\fr{\rm{N}}{8}\leq \rm{N}' \leq 8\rm{N}} \left\la \tl{u}_{{0}_{\N}}\cdot\nb A_0^{\sig_1-2}B\theta_{0_{\N'}}-A_0^{\sig_1-2}B(\tl{u}_{{0}_{\N}}\cdot\nb \theta_{0_{\N'}}) , A_0^{\sig_1-2}B\theta_0 \right\ra\\&=\tilde{\mathcal{R}}_1+\tilde{\mathcal{R}}_2.
 \end{align*}
The estimate of $
\tilde{\mathcal{R}}_1$ is easier to get, hence we avoid carrying it out here. Let us focus on $\tilde{\mathcal{R}}_2$,   which on the Fourier side reads
 \begin{align*}
 \tilde{\mathcal{R}}_2=-2\pi\sum_{\N\in \mathbb{D}}\sum_{\fr{\rm{N}}{8}\leq \rm{N}' \leq 8\rm{N}} \sum_{\substack{\al,\beta \in \mathbb{Z}\\ k,l=0}} \int_{\eta,\xi} &\overline{\reallywidehat{A^{\sigma_1-2}B\theta_0}}(t,\eta,\alpha) {A^{\sigma_1-2}_0}(t,\eta,\alpha)B(\eta,\alpha)\\
 &\times \widehat{\tilde{u}_0}(t,\xi,\beta)_\N \cdot\widehat{\nabla \theta}_{0}(t,\eta-\xi,\al-\beta)_{\N'}d\xi d\eta.
 \end{align*}
 Again, via the paraproduct decomposition, on the support of the integrand of  $\tilde{\mathcal{R}}_2$, it is true that $\fr{\N}{2}\le|\xi,\beta|\le\fr{3\N}{2}$  $\fr{\N'}{2}\le|\eta-\xi,\al-\beta|\le\fr{3\N'}{2}$ with $\fr{\N}{8}\le \N'\le8\N$.
Hence, the inequality \eqref{comp xi beta} still holds with $k=0$.
Furthermore, we know that there exists $0<c<1$ for $ 0<s\leq 1$, so that by triangle inequality, $|\eta,\alpha|^s \lesssim \left( c|\xi,\beta|^s +c|\eta-\xi,\alpha-\beta|^s \right)$.

As a result, via \eqref{ratio of B},
\[
A^{\sigma_1-2}_0
(t,\eta,\alpha)B(\eta,\alpha)\lesssim \textup{J}_0(t,\eta,\alpha)e^{c\lambda(t)|\xi,\beta|^s}e^{c\lambda(t)|\eta-\xi,\al-\beta|^s}\JB{\eta-\xi,\alpha-\beta}B(\xi,\beta)\JB{\xi,\beta}\JB{\eta-\xi,\al-\beta}^{\sigma_1-3}.
\]
But, via the definition of \textup{J} \eqref{multiplier J} and the total growth in Lemma~\ref{total growth}, we can infer that 
\[
|\textup{J}_0(t,\eta,\alpha)| \lesssim e^{\mu\sqrt{|\iota(\eta,\alpha)|}}\lesssim e^{\mu\sqrt{|\eta,\alpha|}}\lesssim e^{\mu(\sqrt{|\eta-\xi,\alpha-\beta|}+\sqrt{|\xi,\beta|})}. 
\]
Hence, via \eqref{eq: tlu_0} and the bootstrap hypotheses
\begin{align*}
|\tilde{\mathcal{R}}^2|&\lesssim \sum_{\N\in \mathbb{D}}\sum_{\fr{\rm{N}}{8}\leq \rm{N}' \leq 8\rm{N}} \sum_{\substack{\al,\beta \in \mathbb{Z}\\ k,l=0}} \int_{\eta,\xi} |\overline{\reallywidehat{A^{\sigma_1-2}B\theta_0}}(t,\eta,\alpha)|
e^{\mathbf{c}\lambda(t)|\xi,\beta|^s}B(\xi,\beta)|\xi,\beta||\widehat{\tilde{u}_0}(t,\xi,\beta)_\N| \\&
\qquad \qquad \qquad \times e^{\mathbf{c}\lambda(t)|\eta-\xi,\al-\beta|^s} |\eta-\xi,\al-\beta|^{\sigma_1-2}|\widehat{\nabla \theta}_{0}(t,\eta-\xi,\al-\beta)_{\N'}|d\xi d\eta\\&
\lesssim \norm{B\theta_0}_{\sigma_1-2}\|\mathbb{P}^z_{\ne}\theta_0\|_{\mathcal{G}^{\lambda,\sig_6}}\norm{\theta}_{\mathcal{G}^{\lambda,\sigma_3}}\lesssim\frac{\epsilon^3}{t^2}.
\end{align*}

\section{Intermediate and lower energy estimates}
In this section, we deal with the intermediate and lower energy estimates.  Here we combine the semi-group estimates for the zero mode $\mathbb{P}^z_{\ne}\theta_0$ and energy estimates for $\theta$. We refer to Figure \ref{map} for the idea of the proof. 

To begin with, we rewrite \eqref{eq-theta0'} in Fourier variable:
\be\label{eq-theta0hat}
\pr_t\widehat{\theta}_0(t,\eta,\al)+\fr{\al^2}{(\eta^2+\al^2)^2}\widehat{\theta}_0(t,\eta,\al)=-\mathcal{F}[\tl u_0\cdot\nb_{yz}\theta_0](t,\eta,\al)-\mathcal{F}[(u_{\ne}\cdot\nb_{xyz}\theta_{\ne})_0](t,\eta,\al).
\ee
Then
\begin{align}\label{exp-theta0-hat}
\widehat{\theta}_0(t,\eta,\al)\nn=&\mathbb{S}(t-10;\eta,\al)\widehat{\theta}_0(10,\eta, \al)-\int_{10}^t\mathbb{S}(t-\tau;\eta,\al) \mathcal{F}[\tl u_0\cdot\nb_{yz}\theta_0](\tau,\eta,\al) d\tau\\
\nn&-\int_{10}^t\mathbb{S}(t-\tau;\eta,\al) \mathcal{F}[(u_{\ne}\cdot\nb_{xyz}\theta_{\ne})_0](\tau,\eta,\al) d\tau\\
=&\mathcal{I}+{\rm S}_{0}+{\rm S}_{\ne},
\end{align}
where 
$\mathbb{S}(t;\eta,\al)=e^{-\fr{\al^2}{(\eta^2+\al^2)^2}t}$
is the semi-group.

\subsection{Improvement of \eqref{En0-2}}\label{sec-En02}
We first deal with the initial term $\mathcal{I}$. Since for $\al\ne0$,  $\la t\ra^{\fr32}\mathbb{S}(t;\eta,\al)\les \left(\fr{(\eta^2+\al^2)^2}{\al^2}\right)^{\fr32}\les |\eta,\al|^6$, one easily deduces that
\begin{equation}
    \left\|\sup_t{\mathds{1}}_{\al\ne0}e^{\lm(t)|\eta,\al|^s}\la t\ra^{\fr32}\la\eta,\al\ra^{\sig_2}|\mathcal{I}|\right\|_{L^2_{\eta,\al}}\les \left\|e^{\lm(10)|\nb|^s}\la \nb\ra^{\sig_2+6}\theta(10)\right\|_{L^2}.
\end{equation}

Now we turn to $S_0$. By using frequency decomposition, we split ${\rm S}_0$ into two part:
\be\label{split-S0}
\begin{aligned}
{\rm S}_{0}
=&-\int_{10}^t\mathbb{S}(t-\tau;\eta,\al)\sum_{\beta\in\Z}\int_\xi\widehat{\tl u}_0(\tau, \xi,\beta)\cdot\widehat{\nb_{yz}\theta}_0(\tau,\eta-\xi,\al-\beta){\mathds{1}}_{|\eta-\xi,\al-\beta|\le2|\xi,\beta|}d\xi d\tau\\
&-\int_{10}^t\mathbb{S}(t-\tau;\eta,\al)\sum_{\beta\in\Z}\int_\xi\widehat{\tl u}_0(\tau, \eta-\xi,\al-\beta)\cdot\widehat{\nb_{yz}\theta}_0(\tau,\xi,\beta){\mathds{1}}_{|\eta-\xi,\al-\beta|\le\fr12|\xi,\beta|}d\xi d\tau\\
=&{\rm S}_{0}^{\rm HL}+{\rm S}_{0}^{\rm LH}.
\end{aligned}
\ee
For $\al\ne0$, by \eqref{triangle1}, we have 
\begin{align*}
    &\la t\ra^{\fr32}e^{\lm(t)| \eta,\al|^s}\la \eta,\al\ra^{\sig_2}|{\rm S}_{0}^{\rm LH}|\\
\les&\fr{\al^2}{(\eta^2+\al^2)^2}\int_{10}^t\left(\fr{\al^2\la t\ra}{(\eta^2+\al^2)^2} \right)^{\fr32}e^{\lm(t)| \eta,\al|^s}\la \eta,\al\ra^{\sig_2}\mathbb{S}(t-\tau;\eta,\al)\sum_{\beta\in\Z}\int_\xi   \left( \fr{(\eta^2+\al^2)^2}{\al^2}\right)^\fr52 \\
&\times\left|\widehat{\tl u}_0(\tau, \eta-\xi,\al-\beta)\right|\left|\widehat{\nb_{yz}\theta}_0(\tau,\xi,\beta)\right|{\mathds{1}}_{|\eta-\xi,\al-\beta|\le\fr12|\xi,\beta|}d\xi d\tau\\
\les&\fr{\al^2}{(\eta^2+\al^2)^2}\int_{10}^t \left(\fr{\al^2\la t-\tau\ra}{(\eta^2+\al^2)^2} \right)^\fr32 \mathbb{S}(t-\tau;\eta,\al)d\tau \sup_{\tau}\sum_{\beta\in\Z}\int_\xi e^{c\lm(\tau)| \eta-\xi,\al-\beta|^s}\\
&\times \la \tau\ra^{\fr32}\left|\widehat{\tl u}_0(\tau, \eta-\xi,\al-\beta)\right|\left[e^{\lm(\tau)| \xi,\beta|^s}\la\xi,\beta\ra^{\sig_2+11}\left(\la |\xi|^{\fr12}, \beta\ra|\widehat{\theta}_0(\tau,\xi,\beta)|\right)\right]d\xi.
\end{align*}
The following estimate will be frequently used in the proof. It holds that for all $\al\ne0,\eta\in\mathbb{R}$, and $m\ge0$, $t>10$
\be\label{semigroup-bd}
\fr{\al^2}{(\eta^2+\al^2)^2}\int_{10}^t\left(\fr{\al^2\la t-\tau\ra}{(\eta^2+\al^2)^2} \right)^m \mathbb{S}(t-\tau;\eta,\al)d\tau\les1.
\ee
By \eqref{exp-u}, we also get that
\begin{align}\label{bd-tlu0}
\left|\widehat{\tl{u}}_0(\tau,\eta-\xi,\al-\beta)\right|\les|\widehat{\Pe^z_{\ne}\theta_0}(\tau,\eta-\xi,\al-\beta)|.
\end{align}
Then by Remark \ref{rmk: Minkowski}, the above three inequalities show that 
\begin{equation}\label{S20LH}
\begin{aligned}
\la t\ra^{\fr32}e^{\lm(t)| \eta,\al|^s}&\la \eta,\al\ra^{\sig_2}|{\rm S}_{0}^{\rm LH}|\\
\les&\sup_{\tau}\sum_{\beta\in\Z}\int_\xi
 \la \tau\ra^{\fr32}e^{c\lm(\tau) |\eta-\xi,\al-\beta|^s}\la\eta-\xi,\al-\beta\ra^2|\widehat{\Pe^z_{\ne}\theta_0}(\tau,\eta-\xi,\al-\beta)|\\
&\times\left[e^{\lm(\tau)|\xi,\beta|^s}\la\xi,\beta\ra^{\sig_2+13}\left(\la |\xi|^{\fr12}, \beta\ra|\widehat{\theta}_0(\tau,\xi,\beta)|\right)\right]d\xi \fr{1}{\la \eta,\al\ra^2}\\
\lesssim & \sup_\tau\left(\|\JB{\tau}^\fr32\mathbb{P}_{\neq}^z\theta_0(\tau)\|_{\mathcal{G}^{\lm,\sig_2}}\right)\sup_\tau\left(\|B\theta_0(\tau)\|_{\sig_1-2}\right)\fr{1}{\la\eta,\al\ra^2}\\
\lesssim & \left\|\sup_{\tau}\Big|\JB{\tau}^\fr32 e^{\lambda(\tau)|\eta,\al|^{s}}\JB{\eta,\al}^{\sigma_2}\widehat{\mathbb{P}_{\neq}^z\theta_0}(\tau,\eta,\al)\Big|\right\|_{L^2_{\eta,\al}}\sup_\tau\left(\|B\theta_0(\tau)\|_{\sig_1-2}\right)\fr{1}{\la\eta,\al\ra^2},
\end{aligned}
\end{equation}
provided $\sig_1-2\ge \sig_2+13$. It follows that
\[
\begin{aligned}
&\left\|\sup_t{\mathds{1}}_{\al\ne0}\la t\ra^{\fr32}e^{\lm(t) |\eta,\al|^s}\la \eta,\al\ra^{\sig_2}|{\rm S}_{0}^{\rm LH}|\right\|_{L^2_{\eta,\al}}\les\eps^2.
\end{aligned}
\]

Let us postpone the estimates of ${\rm S}_0^{\rm HL}$ and move on to ${\rm S}_{\ne}$, which can be treated in a similar manner to that of ${\rm S}_0^{\rm LH}$.
More precisely, we can split ${\rm S}_{\ne}$ as \eqref{split-S0}:
\begin{align}
{\rm S}_{\ne}\nn=&-\int_{10}^t\mathbb{S}(t-\tau;\eta,\al)\sum_{k,\beta}\int_{\xi} \widehat{u}_{\ne}(\tau,k,\xi,\beta)\cdot\widehat{\nb_{xyz}\theta}_{\ne}(\tau,-k,\eta-\xi,\al-\beta){\mathds{1}}_{|\eta-\xi,\al-\beta|\le2|\xi,\beta|}d\xi d\tau\\
\nn&-\int_{10}^t\mathbb{S}(t-\tau;\eta,\al)\sum_{k,\beta}\int_{\xi} \widehat{u}_{\ne}(\tau,k,\eta-\xi,\al-\beta)\cdot\widehat{\nb_{xyz}\theta}_{\ne}(\tau,-k,\xi,\beta){\mathds{1}}_{|\eta-\xi,\al-\beta|\le\fr12|\xi,\beta|}d\xi d\tau\\
\nn=&{\rm S}_{\ne}^{\rm HL}+{\rm S}_{\ne}^{\rm LH}.
\end{align}
Then using \eqref{u_ne-decay} and \eqref{semigroup-bd}, one  deduces that
\begin{align}\label{Sne2HL}
\la t \ra^{\fr32}e^{\lm(t)| \eta,\al|^s}\la \eta,\al\ra^{\sig_2}|{\rm S}_{\ne}^{\rm HL}|
\nn\les&\sup_\tau\sum_{k, \beta\in\Z}\int_\xi e^{\lm(\tau) |\xi,\beta|^s}\la \tau\ra^{\fr32}\la \xi,\beta\ra^{\sig_2+12}\left|\widehat{u}_{\ne}(\tau, k, \xi,\beta)\right|\\
\nn&\times \left(e^{\lm(\tau)|\eta- \xi,\al-\beta|^s}\left|\widehat{\theta}_{\ne}(\tau,-k,\eta-\xi,\al-\beta)\right|\right)d\xi \fr{1}{\la\eta,\al\ra^2}\\
\nn\les&\sup_\tau \sum_{k,\beta\in\Z}\int_\xi e^{\lm(\tau) |\xi,\beta|^s}\la k, \xi,\beta\ra^{\sig_2+18}\left|\widehat{\theta}_{\ne}(\tau, k, \xi,\beta)\right|\fr{1}{\la\tau\ra^{\fr32}}\\
\nn&\times \left(e^{\lm(\tau)|\eta- \xi,\al-\beta|^s}\left|\widehat{\theta}_{\ne}(\tau, -k,\eta-\xi,\al-\beta)\right|\right)d\xi \fr{1}{\la\eta,\al\ra^2}\\
\les&\sup_\tau\left[\left(\fr{1}{\la\tau\ra^{\fr32}}\|\theta(\tau)\|_{\mathcal{G}^{\lm,\sig_1}}\right)\|\theta(\tau)\|_{\mathcal{G}^{\lm,\sig_5}}\right]\fr{1}{\la\eta,\al\ra^2},
\end{align}
and
\begin{align}\label{Sne2LH}
\nn&\la t \ra^{\fr32}e^{\lm(t) |\eta,\al|^s}\la \eta,\al\ra^{\sig_2}|{\rm S}_{\ne}^{\rm LH}|\\
\nn\les&\sup_\tau\sum_{k,\beta\in\Z}\int_\xi e^{c\lm(\tau)|\eta-\xi,\al-\beta|^s}\la\eta-\xi,\al-\beta\ra^2\la \tau\ra^\fr32\left|\widehat{u}_{\ne}(\tau, k, \eta-\xi,\al-\beta)\right|\\
\nn&\times \left(e^{\lm(\tau)|\xi,\beta|^s}\la \xi,\beta\ra^{\sig_2+12}\left|\widehat{\nb_{xyz}\theta}_{\ne}(\tau,-k,\xi,\beta)\right|\right)d\xi \fr{1}{\la\eta,\al\ra^2}\\
\nn\les&\sup_\tau \sum_{k,\beta\in\Z}\int_\xi e^{c\lm(\tau) |\eta-\xi,\al-\beta|^s}\la k, \eta-\xi,\al-\beta\ra^8\left|\widehat{\theta}_{\ne}(\tau, k, \eta-\xi,\al-\beta)\right|\\
\nn&\times \fr{1}{\la\tau\ra^{\fr32}}\left(e^{\lm(\tau)| \xi,\beta|^s}\la \xi,\beta\ra^{\sig_2+12}\left|\widehat{\nb_{xyz}\theta}_{\ne}(\tau,-k,\xi,\beta)\right|\right)d\xi\fr{1}{\la\eta,\al\ra^2}\\
\les&\sup_\tau\left[\left(\fr{1}{\la\tau\ra^{\fr32}}\|\theta(\tau)\|_{\mathcal{G}^{\lm,\sig_1}}\right)\|\theta(\tau)\|_{\mathcal{G}^{\lm,\sig_5}}\right]\fr{1}{\la\eta,\al\ra^2}.
\end{align}
Consequently,
\begin{align}
\left\|\sup_t{\mathds{1}}_{\al\ne0}\la t \ra^{\fr32}e^{\lm(t)| \eta,\al|^s}\la \eta,\al\ra^{\sig_2}{\rm S}_{\ne}\right\|_{L^2_{\eta,\al}}\les\sup_t\left[\|\theta(t)\|_{\mathcal{G}^{\lm,\sig_5}}\left(\fr{1}{\la t \ra^{\fr32}} \|\theta(t)\|_{\mathcal{G}^{\lm,\sig_1}}\right)\right]\les\eps^2,
\end{align}
provided we choose $\sig_2+18\le \sig_1$.

Next, we turn to ${\rm S}_{0}^{\rm HL}$, which can be divided into two parts based on $\theta_0$ is at zero frequency or not with respect to $z$ variable:
\[
{\rm S}_{0}^{\rm HL}={\rm S}_{0\ne z}^{\rm HL}+{\rm S}_{00z}^{\rm HL},
\]
with
\[
\begin{aligned}
{\rm S}_{0\ne z}^{\rm HL}=&-\int_{10}^t\mathbb{S}(t-\tau;\eta,\al)\sum_{\beta\ne0,\beta\ne\al}\int_\xi\widehat{\tl u}_0(\tau, \xi,\beta)\reallywidehat{\nb_{yz}\mathbb{P}^z_{\ne}\theta_0}(\tau,\eta-\xi,\al-\beta){\mathds{1}}_{|\eta-\xi,\al-\beta|\le2|\xi,\beta|}d\xi d\tau,\\
{\rm S}_{00z}^{\rm HL}=&i\int_{10}^t\mathbb{S}(t-\tau;\eta,\al)\int_{\xi}\fr{\al^2}{(\xi^2+\al^2)^2}\widehat{\Pe^{z}_{\ne}\theta}_0(\tau,\xi,\al)(\eta-\xi)\widehat{\theta}_0(\tau,\eta-\xi,0){\mathds{1}}_{|\eta-\xi|\le2|\xi,\al|}d\xi d\tau.
\end{aligned}
\]
The interaction with non-zero modes in $z$ is easy to control, due to the time decay of $\mathbb{P}^z_{\ne}\theta_0$ in the lower frequencies. In fact,  using again \eqref{semigroup-bd} and Remark \ref{rmk: Minkowski}, we arrive at
\begin{align}
 \nn   &\la t\ra^{\fr32}e^{\lm(t)|\eta,\al|^s}\la\eta,\al\ra^{\sig_2}|{\rm S}^{\rm HL}_{0\ne z}|\\
\nn    \les&\fr{\al^2}{(\eta^2+\al^2)^2}\int_{10}^t\left(\fr{\al^2\la t-\tau\ra}{(\eta^2+\al^2)^2}\right)^{\fr32}\mathbb{S}(t-\tau; \eta,\al)\\
\nn    &\times\sum_{\beta\ne0,\beta\ne\al} \int_{\xi}e^{\lm(\tau)|\xi,\beta|^s}\la\xi,\beta\ra^{\sig_2+12}\la|\xi|^\fr12,\beta\ra|\widehat{\mathbb{P}_{\ne}^z\theta_0}(\tau,\xi,\beta)|\\
\nn    &\times\la\tau\ra^{\fr32}e^{c\lm(\tau)|\eta-\xi,\al-\beta|^s}\la \eta-\xi,\al-\beta\ra^2|\reallywidehat{\nb_{yz}\mathbb{P}^z_{\ne}\theta_0}(\tau,\eta-\xi,\al-\beta)|d\xi \fr{1}{\la\eta,\al\ra^2}\\
\label{e-S0nezHL}    \les&\sup_\tau \|B\theta_0\|_{\sig_1-2}\left\|\sup_\tau \la\tau\ra^{\fr32}e^{\lm(\tau)|\eta,\al|^s}\la \eta,\al\ra^{\sig_2}|\widehat{\mathbb{P}_{\ne}^z\theta_0}(\tau,\eta,\al)| \right\|_{L^2_{\eta,\al}}\fr{1}{\la\eta,\al\ra^2}.
\end{align}
Thus,
\[
\left\|\sup_{t}{\mathds{1}}_{\al\ne0} \la t\ra^{\fr32}e^{\lm(t)|\eta,\al|^s}\la\eta,\al\ra^{\sig_2} {\rm S}^{\rm HL}_{0\ne z}\right\|_{L^2_{\eta,\al}}\les\eps^2,
\]
provided we choose $\sig_1-2\ge \sig_2+12$.

To bound ${\rm S}^{\rm HL}_{00z}$, noting that $\fr{|\eta,\al|}{|\xi,\al|}\les\la \eta-\xi\ra$ for $\al\ne0$, we have
\be\label{Isig2}
\begin{aligned}
&\la t\ra^\fr32 e^{\lm(t)|\eta,\al|^s}\la\eta,\al\ra^{\sig_2}|{\rm S}_{00z}^{\rm HL}|\\
\les&\int_{\fr{t}{2}}^t\fr{\al^2}{(\eta^2+\al^2)^2}\mathbb{S}(t-\tau;\eta,\al)\int_{\xi}\la\tau\ra^\fr32e^{\lm(\tau)| \xi,\al|^s}\la \xi,\al\ra^{\sig_2}|\widehat{\Pe^z_{\ne}\theta_0}(\tau, \xi,\al)|\\
&\qquad \times e^{c\lambda(\tau)| \eta-\xi|^s}\la\eta-\xi\ra^7 |\widehat{\theta}_0(\tau,\eta-\xi,0)| \fr{1}{\la\eta-\xi\ra^2}d\xi d\tau\\
&+\int_{10}^\fr{t}{2}\fr{\al^2}{(\eta^2+\al^2)^2}\la t-\tau\ra^\fr32 e^{(\lm(t)-\lm(\tau))| \eta,\al|^s}\mathbb{S}(t-\tau;\eta,\al)\int_{\xi}e^{\lm(\tau) |\xi,\al|^s}\la \xi,\al\ra^{\sig_2}|\widehat{\Pe^z_{\ne}\theta_0}(\tau, \xi,\al)|\\
&\qquad \qquad  \times e^{c\lm(\tau)| \eta-\xi|^s}\la\eta-\xi\ra^7 |\widehat{\theta}_0(\tau,\eta-\xi,0)|\fr{1}{\la\eta-\xi\ra^2}d\xi d\tau={\rm I}_1^{\sig_2}+{\rm I}_2^{\sig_2}.
\end{aligned}
\ee
In view of \eqref{semigroup-bd} and the bootstrap assumption \eqref{En-7}, one can bound ${\rm I}_1^{\sig_2}$ as follows:
\[
\begin{aligned}
|{\rm I}_1^{\sig_2}|\les &\int_{\xi}\sup_{\tau}\la\tau\ra^\fr32e^{\lm(\tau)| \xi,\al|^s}\la \xi,\al\ra^{\sig_2}|\widehat{\Pe^z_{\ne}\theta_0}(\tau, \xi,\al)|\fr{1}{\la\eta-\xi\ra^2}d\xi\sup_t\sup_\eta \left| e^{\lm(t)|\eta|^s}\widehat{\mathbb{P}^z_0\theta}_0(t,\eta)\right|\\
\les&\eps\int_{\xi}\sup_{\tau}\la\tau\ra^\fr32e^{\lm(\tau)| \xi,\al|^s}\la \xi,\al\ra^{\sig_2}|\widehat{\Pe^z_{\ne}\theta}_0(\tau, \xi,\al)|\fr{1}{\la\eta-\xi\ra^2}d\xi,
\end{aligned}
\]
and hence,
\begin{equation}\label{bd-I1}
\left\|\sup_t {\rm I}_1^{\sig_2}\right\|_{L^2_{\eta,\al}}\les\eps\left\|\sup_\tau \la\tau\ra^{\fr32}e^{\lm(\tau)| \eta,\al|^s}\la \eta,\al\ra^{\sig_2}|\widehat{\Pe^z_{\ne}\theta}_0(\tau, \eta,\al)|\right\|_{L^2_{\eta,\al}}\les \eps^2.
\end{equation}
As for ${\rm I}_2^{\sig_2}$, by the definition of $\lm(t)$,
\begin{equation}\label{gap1}
\lm(\tau)-\lm(t)=\fr{\tl\dl}{(1+\tau)^a}-\fr{\tl \dl}{(1+t)^a}\gt\fr{t-\tau}{(1+ \tau)^a(1+t)}, \quad \mathrm{for \ \ all}\quad \tau<t.
\end{equation}
Combining this with the restriction $\tau\le\fr{t}{2}$, noting that $\al\ne0$, we are led to 
\begin{align}
\la t-\tau\ra^{\fr32} e^{(\lm(t)-\lm(\tau))| \eta,\al|^s}=\nn&\left(\fr{\la t-\tau\ra\al^2}{(\eta^2+\al^2)^2}\right)^{\fr32}\left(\fr{(\eta^2+\al^2)^2}{\al^2}\right)^{\fr32}e^{-(\lm(\tau)-\lm(t))| \eta,\al|^s}\\
\nn\les&\left(\fr{\la t-\tau\ra\al^2}{(\eta^2+\al^2)^2}\right)^{\fr32}\la\eta,\al\ra^{6}\fr{\la\tau\ra^{\Gamma a}}{\la\eta,\al\ra^{\Gamma s}}\\
\nn=&\left(\fr{\la t-\tau\ra\al^2}{(\eta^2+\al^2)^2}\right)^{\fr32}\la\eta,\al\ra^{6}\fr{1}{\la\eta,\al\ra^{\Gamma s}}\left(\fr{\la\tau\ra\al^2}{(\eta^2+\al^2)^2}\right)^{\Gamma a}\left(\fr{(\eta^2+\al^2)^2}{\al^2}\right)^{\Gamma a}\\
\nn\les&\left(\fr{\la t-\tau\ra\al^2}{(\eta^2+\al^2)^2}\right)^{\fr32+\Gamma a}\fr{\la\eta,\al\ra^{4\Gamma a+6}}{\la\eta,\al\ra^{\Gamma s}}\les\left(\fr{\la t-\tau\ra\al^2}{(\eta^2+\al^2)^2}\right)^{3+\Gamma a},
\end{align}
provided we choose 
\be
\Gamma s\ge 4\Gamma a+6, \quad{\rm namely,}\quad \Gamma\ge\fr{12}{s-4a}\quad {\rm with}\quad s>4a.
\ee
Then
\be
\begin{aligned}
\int_{10}^\fr{t}{2}&\fr{\al^2}{(\eta^2+\al^2)^2}\la t-\tau\ra^{\fr32} e^{(\lm(t)-\lm(\tau))| \eta,\al|^s}\mathbb{S}(t-\tau;\eta,\al)d\tau\\
\les&\int_{10}^\fr{t}{2}\fr{\al^2}{(\eta^2+\al^2)^2}\left(\fr{\la t-\tau\ra\al^2}{(\eta^2+\al^2)^2}\right)^{\fr32+\Gamma a}e^{-\fr{\al^2}{(\eta^2+\al^2)^2}(t-\tau)}d\tau\\
\les&\int_{10}^t(1+ \tau)^{\fr32+\Gamma a}e^{-\tau}d\tau\les1.
\end{aligned}
\ee
Consequently, similar to the estimate of ${\rm I}_1^{\sig_2}$, we have
\[
\begin{aligned}
{\rm I}_2^{\sig_2}\les&\sup_\tau\sup_{\eta}\left| e^{\lm(\tau)|\eta|^s}\widehat{\mathbb{P}^z_0\theta}_0(\tau,\eta)\right|\sup_{\tau}\int_{\xi}e^{\lm(\tau)|\xi,\al|^s}\la\xi,\al\ra^{\sig_2}|\widehat{\mathbb{P}^z_{\ne}\theta_0}(\tau,\xi,\al)|\fr{1}{\la \eta-\xi\ra^2}d\xi\\
\les&\eps\int_{\xi}\sup_{\tau}\Big(e^{\lm(\tau)|\xi,\al|^s}\la\xi,\al\ra^{\sig_2}|\widehat{\mathbb{P}^z_{\ne}\theta}_0(\tau,\xi,\al)|\Big)\fr{1}{\la \eta-\xi\ra^2}d\xi,
\end{aligned}
\]
and thus
\begin{equation}\label{e-I_2}
\left\|\sup_t{\rm I}_2^{\sig_2} \right\|_{L^2_{\eta,\al}}\les\eps
\left\|\sup_\tau \la\tau\ra^{\fr32}e^{\lm(\tau)| \eta,\al|^s}\la \eta,\al\ra^{\sig_2}|\widehat{\Pe^z_{\ne}\theta_0}(\tau, \eta,\al)|\right\|_{L^2_{\eta,\al}}\les\eps^2.
\end{equation}

\subsection{Improvement of \eqref{En-3}}\label{sec-im-En-3}
Similar to \eqref{e-theta}, we have
\be\label{e-theta-sig3}
\begin{aligned}
&\fr12\fr{d}{dt}\|\theta(t)\|_{\mathcal{G}^{\lm, \sig_3}}^2
+\left\|\sqrt{-\Dl_{xz}}\Delta^{-1}_L \theta\right\|_{\mathcal{G}^{\lm, \sig_3}}^2+{\rm CK}^{\sig_3}_{\lambda,\theta}(t)\\
=&-\left\la e^{\lm(t)|\nb|^s}\la\nb\ra^{\sig_3}(u_0\cdot\nb \theta), e^{\lm(t)|\nb|^s}\la\nb\ra^{\sig_3}\theta\right\ra-\left\la e^{\lm(t)|\nb|^s}\la\nb\ra^{\sig_3}(u_{\ne}\cdot\nb \theta), e^{\lm(t)|\nb|^s}\la\nb\ra^{\sig_3}\theta\right\ra\\
=&{\bf NL}_{\sig_3;0}+{\bf NL}_{\sig_3;\ne}.
\end{aligned}
\ee
By using frequency decomposition, we write
\[
{\bf NL}_{\sig_3;0}={\bf NL}_{\sig_3;0}^{\rm HL}+{\bf NL}_{\sig_3;0}^{\rm LH},\qquad {\bf NL}_{\sig_3;\ne}={\bf NL}_{\sig_3;\ne}^{\rm HL}+{\bf NL}_{\sig_3;\ne}^{\rm LH},
\]
where
\begin{align*}
    {\bf NL}_{\sig_3;0}^{\rm HL}=&-\sum_{k,\al,\beta\in\Z}\int_{\eta,\xi}e^{\lm(t)|k, \eta,\al|^s}\la k, \eta,\al\ra^{\sig_3}\widehat{u}_0(t,\xi, \beta)\cdot\widehat{\nb_{xyz}\theta}_k(t,\eta-\xi,\al-\beta)\\
&\qquad \qquad \times  e^{\lm(t)| k,\eta,\al|^s}\la k,\eta,\al\ra^{\sig_3}\overline{\widehat{\theta}}_k(t,\eta,\al){\mathds{1}}_{|k,\eta-\xi,\al-\beta|<\fr12|\xi,\beta|}d\xi d\eta,\\
{\bf NL}_{\sig_3;0}^{\rm LH}=&-\sum_{k,\al,\beta\in\Z}\int_{\eta,\xi}e^{\lm(t)|k, \eta,\al|^s}\la k, \eta,\al\ra^{\sig_3}\widehat{u}_0(t,\eta-\xi, \al-\beta)\cdot\widehat{\nb_{xyz}\theta}_k(t,\xi,\beta)\\
&\qquad \qquad \times  e^{\lm(t)| k,\eta,\al|^s}\la k,\eta,\al\ra^{\sig_3}\overline{\widehat{\theta}}_k(t,\eta,\al){\mathds{1}}_{|\eta-\xi,\al-\beta|\le 2|k,\xi,\beta|}d\xi d\eta,\\
{\bf NL}_{\sig_3;\ne}^{\rm HL}=&-\sum_{k,\al,\beta\in\Z,l\ne0}\int_{\eta,\xi}e^{\lm(t)|k, \eta,\al|^s}\la k, \eta,\al\ra^{\sig_3}\widehat{u}_{\ne}(t,l,\xi, \beta)\cdot\widehat{\nb_{xyz}\theta}_{k-l}(t,\eta-\xi,\al-\beta)\\
&\qquad \qquad\times  e^{\lm(t)| k,\eta,\al|^s}\la k,\eta,\al\ra^{\sig_3}\overline{\widehat{\theta}}_k(t,\eta,\al){\mathds{1}}_{|k-l,\eta-\xi,\al-\beta|<\fr12|l,\xi,\beta|}d\xi d\eta,\\
{\bf NL}_{\sig_3;\ne}^{\rm LH}=&-\sum_{k\ne l,\al,\beta\in\Z}\int_{\eta,\xi}e^{\lm(t)|k, \eta,\al|^s}\la k, \eta,\al\ra^{\sig_3}\widehat{u}_{\ne}(t,k-l,\eta-\xi, \al-\beta)\cdot\widehat{\nb_{xyz}\theta}_l(t,\xi,\beta)\\
&\qquad \qquad\times  e^{\lm(t)| k,\eta,\al|^s}\la k,\eta,\al\ra^{\sig_3}\overline{\widehat{\theta}}_k(t,\eta,\al){\mathds{1}}_{|k-l,\eta-\xi,\al-\beta|\le 2|l,\xi,\beta|}d\xi d\eta.
\end{align*}

Thanks to \eqref{u_ne-decay},
\be\label{e-NLHL3}
\begin{aligned}
{\bf NL}_{\sig_3;\ne}^{\rm HL}\les&\sum_{\substack{k,\al,\beta\in\Z\\l\ne0}}\int_{\eta,\xi}e^{\lm(t)|l, \xi,\beta|^s}\la l, \xi,\beta\ra^{\sig_3+6}\fr{1}{\la t\ra^3}|\widehat{\theta}_{\ne}(t,l,\xi, \beta)| e^{c\lm(t)|k-l,\eta-\xi,\al-\beta|^s} \\
&\qquad \qquad\times|\widehat{\nb_{xyz}\theta}_{k-l}(t,\eta-\xi,\al-\beta)| e^{\lm(t)| k,\eta,\al|^s}\la k,\eta,\al\ra^{\sig_3}|\widehat{\theta}_k(t,\eta,\al)|d\xi d\eta\\
\les&\fr{1}{\la t\ra^{\fr32}}\|t^{-\fr32}\theta \|_{\sig_1}\|\theta\|_{\mathcal{G}^{\lm,\sig_5}}\|\theta\|_{\mathcal{G}^{\lm,\sig_3}}\les\fr{1}{\la t\ra^{\fr32}}\eps^2\|\theta\|_{\mathcal{G}^{\lm,\sig_3}},
\end{aligned}
\ee
and
\be\label{e-NLLH3}
\begin{aligned}
{\bf NL}_{\sig_3;\ne}^{\rm LH}\les&\sum_{k\ne l,\al,\beta\in\Z}\int_{\eta,\xi}e^{c\lm(t)|k-l, \eta-\xi,\al-\beta|^s}\la k-l, \eta-\xi,\al-\beta\ra^{6}\fr{1}{\la t\ra^3}|\widehat{\theta}_{k-l}(t,\eta-\xi, \al-\beta)| \\
&\qquad \quad\times e^{\lm(t)|l,\xi,\beta|^s}|l,\xi,\beta|^{\sig_3}\widehat{\nb_{xyz}\theta}_{l}(t,\xi,\beta)| \times  e^{\lm(t)| k,\eta,\al|^s}\la k,\eta,\al\ra^{\sig_3}|\widehat{\theta}_k(t,\eta,\al)|d\xi d\eta\\
\les&\fr{1}{\la t\ra^{\fr32}}\|t^{-\fr32}\theta \|_{\sig_1}\|\theta\|_{\mathcal{G}^{\lm,\sig_5}}\|\theta\|_{\mathcal{G}^{\lm,\sig_3}}\les\fr{1}{\la t\ra^{\fr32}}\eps^2\|\theta\|_{\mathcal{G}^{\lm,\sig_3}}.
\end{aligned}
\ee
Recalling \eqref{exp-u}, it is easy to see that
\begin{align}\label{bd-u0hat}
|\widehat{u}_0(t,\xi, \beta)|
\les \fr{\la t\ra \beta^2+|\xi||\beta|}{(\xi^2+\beta^2)^2}|\widehat{\Pe^z_{\ne}\theta_0}(t,\xi, \beta)|
\les \fr{\la t\ra }{\xi^2+\beta^2}|\widehat{\Pe^z_{\ne}\theta_0}(t,\xi, \beta)|.
\end{align}
It follows that
\begin{align}
{\bf NL}_{\sig_3;0}^{\rm HL}\nn\les&\sum_{k,\al,\beta\in\Z}\int_{\eta,\xi}e^{\lm(t)| \xi,\beta|^s}\la \xi,\beta\ra^{\sig_3}\fr{\la t\ra \beta^2+|\xi||\beta|}{(\xi^2+\beta^2)^2}|\widehat{\Pe^z_{\ne}\theta_0}(t,\xi, \beta)| e^{c\lm(t)| k,\eta-\xi,\al-\beta|^s}\\
\nn&\qquad \qquad \times\left|\widehat{\nb_{xyz}\theta}_k(t,\eta-\xi,\al-\beta) \right|\left|e^{\lm(t)| k,\eta,\al|^s}\la k,\eta,\al\ra^{\sig_3}\widehat{\theta}_k(t,\eta,\al)\right|d\xi d\eta\\
\nn\les&\sum_{k,\al,\beta\in\Z}\int_{\eta,\xi}e^{\lm(t)|\xi,\beta|^s}\la \xi,\beta\ra^{\sig_3-2}\la t\ra |\widehat{\Pe^z_{\ne}\theta_0}(t,\xi, \beta)|e^{c\lm(t) |k,\eta-\xi,\al-\beta|^s}\\
\nn&\qquad \qquad \times \left|\widehat{\nb_{xyz}\theta}_k(t,\eta-\xi,\al-\beta) \right|\left|e^{\lm(t)|k,\eta,\al|^s}\la k,\eta,\al\ra^{\sig_3}\widehat{\theta}_k(t,\eta,\al)\right|d\xi d\eta\\
\nn\les&\la t\ra \|\Pe^z_{\ne}\theta_0(t) \|_{\mathcal{G}^{\lm,\sig_3-2}}\|\theta(t)\|_{\mathcal{G}^{\lm,\sig_5}}\|\theta(t)\|_{\mathcal{G}^{\lm,\sig_3}}\\
\nn\les&\fr{1}{\la t\ra^{\fr12}}\left(\la t\ra^{\fr32}\|\Pe^z_{\ne}\theta_0(t)\|_{\mathcal{G}^{\lm,\sig_2}}\right)\|\theta(t)\|_{\mathcal{G}^{\lm,\sig_5}}\|\theta(t)\|_{\mathcal{G}^{\lm,\sig_3}}\les\fr{\eps^2}{\la t\ra^{\fr12}}\|\theta(t)\|_{\mathcal{G}^{\lm,\sig_3}},
\end{align}
and
\begin{align}
{\bf NL}_{\sig_3;0}^{\rm LH}
\nn\les&\sum_{k,\al,\beta\in\Z}\int_{\eta,\xi}e^{c\lm(t)| \eta-\xi,\al-\beta|^s}\la t\ra |\Pe^z_{\ne}\widehat{\theta}_0(t,\eta-\xi, \al-\beta)|e^{\lm(t) |k,\xi,\beta|^s}\la k,\xi,\beta\ra^{\sig_3}\\
\nn&\times \left|\widehat{\nb_{xyz}\theta}_k(t,\xi,\beta) \right|\left|e^{\lm(t)|k,\eta,\al|^s}\la k,\eta,\al\ra^{\sig_3}\widehat{\theta}_k(t,\eta,\al)\right|d\xi d\eta\\
\nn\les&\la t\ra \|\Pe^z_{\ne}\theta_0(t) \|_{\mathcal{G}^{\lm,\sig_6}}\|\theta(t)\|_{\sig_1}\|\theta(t)\|_{\mathcal{G}^{\lm,\sig_3}}\\
\nn\les&\fr{1}{\la t\ra^{\fr12}}\left(\la t\ra^{3}\|\Pe^z_{\ne}\theta_0(t)\|_{\mathcal{G}^{\lm,\sig_6}}\right)\|t^{-\fr32}\theta(t)\|_{\sig_1}\|\theta(t)\|_{\mathcal{G}^{\lm,\sig_3}}\les\fr{\eps^2}{\la t\ra^{\fr12}}\|\theta(t)\|_{\mathcal{G}^{\lm,\sig_3}}.
\end{align}
We then conclude that 
\begin{align*}
    \fr{d}{dt}\|\theta(t)\|_{\mathcal{G}^{\lm,\sig_3}}\lesssim \eps^2\fr{1}{\la t\ra^{\fr12}},
\end{align*}
which gives us for $t\in[10, T^*]$, 
\begin{align*}
    \|\theta(t)\|_{\mathcal{G}^{\lm,\sig_3}}\leq \|\theta(10)\|_{\mathcal{G}^{\lm,\sig_3}}+C\eps^2 \sqrt{t}\leq 2C\eps\sqrt{t}. 
\end{align*}

\subsection{Improvement of \eqref{En0-4}}\label{sec-im-En0-4} We first focus on ${\rm S}_{0}^{\rm LH}$ as in Section \ref{sec-En02}. Indeed,  similar to \eqref{S20LH}, we have
\begin{equation}\label{S40LH}
\begin{aligned}
\la t\ra^{\fr52}&e^{\lm(t)| \eta,\al|^s}\la \eta,\al\ra^{\sig_4}|{\rm S}_{0}^{\rm LH}|\\
\lesssim & \left\|\sup_{\tau}\Big|\JB{\tau}^{\fr52} e^{\lambda(\tau)|\eta,\al|^{s}}\JB{\eta,\al}^{\sigma_4}\widehat{\mathbb{P}_{\neq}^z\theta_0}(\tau,\eta,\al)\Big|\right\|_{L^2_{\eta,\al}}\sup_\tau\left(\|B\theta_0(\tau)\|_{\sig_1-2}\right)\fr{1}{\la\eta,\al\ra^2},
\end{aligned}
\end{equation}
provided $\sig_1-2\ge \sig_4+17$, and hence 
\[
\begin{aligned}
&\left\|\sup_t{\mathds{1}}_{\al\ne0}\la t\ra^{\fr52} e^{\lm(t) |\eta,\al|^s}\la \eta,\al\ra^{\sig_4}|{\rm S}_{0}^{\rm LH}|\right\|_{L^2_{\eta,\al}}\les\eps^2.
\end{aligned}
\]

As for ${\rm S}_{\ne}$, similar to \eqref{Sne2HL} and \eqref{Sne2LH}, if $\al\ne0$, we have
\begin{align}\label{Sne4HL}
\la t \ra^{\fr52}e^{\lm(t)|\eta,\al|^s}\la \eta,\al\ra^{\sig_4}|{\rm S}_{\ne}^{\rm HL}|
\nn\les&\sup_\tau \sum_{k,\beta\in\Z}\int_\xi e^{\lm(\tau)| \xi,\beta|^s}\la k, \xi,\beta\ra^{\sig_4+22}\left|\widehat{\theta}_{\ne}(\tau, k, \xi,\beta)\right|\fr{1}{\la\tau\ra^{\fr12}}\\
\nn&\qquad \qquad\times \left(e^{\lm(\tau)|\eta- \xi,\al-\beta|^s}\left|\widehat{\theta}_{\ne}(\tau, -k,\eta-\xi,\al-\beta)\right|\right)d\xi \fr{1}{\la\eta,\al\ra^2}\\
\les&\sup_\tau\left[\left(\fr{1}{\la\tau\ra^{\fr12}}\|\theta(\tau)\|_{\mathcal{G}^{\lm,\sig_3}}\right)\|\theta(\tau)\|_{\mathcal{G}^{\lm,\sig_5}}\right]\fr{1}{\la\eta,\al\ra^2},
\end{align}
and
\begin{align}\label{Sne4LH}
\la t \ra^{\fr52}&e^{\lm(t)|\eta,\al|^s}\la \eta,\al\ra^{\sig_4}|{\rm S}_{\ne}^{\rm LH}|\\
\nn\les&\sup_\tau \sum_{k,\beta\in\Z}\int_\xi e^{c\lm(\tau)|  \eta-\xi,\al-\beta|^s}\la k, \eta-\xi,\al-\beta\ra^8\left|\widehat{\theta}_{\ne}(\tau, k, \eta-\xi,\al-\beta)\right|\\
\nn&\qquad \qquad\times \fr{1}{\la\tau\ra^{\fr12}}\left(e^{\lm(\tau)| \xi,\beta|^s}\la \xi,\beta\ra^{\sig_4+16}\left|\widehat{\nb_{xyz}\theta}_{\ne}(\tau,-k,\xi,\beta)\right|\right)d\xi\fr{1}{\la\eta,\al\ra^2}\\
\les&\sup_\tau\left[\left(\fr{1}{\la\tau\ra^{\fr12}}\|\theta(\tau)\|_{\mathcal{G}^{\lm,\sig_3}}\right)\|\theta(\tau)\|_{\mathcal{G}^{\lm,\sig_5}}\right]\fr{1}{\la\eta,\al\ra^2}.
\end{align}
Thus,
\begin{align}
\left\|\sup_t{\mathds{1}}_{\al\ne0}\la t \ra^{\fr52}e^{\lm(t)|\eta,\al|^s}\la \eta,\al\ra^{\sig_4}{\rm S}_{\ne}\right\|_{L^2_{\eta,\al}}\les\sup_t\left[\|\theta(t)\|_{\mathcal{G}^{\lm,\sig_5}}\left(\fr{1}{\la t \ra^{\fr12}} \|\theta(t)\|_{\mathcal{G}^{\lm,\sig_3}}\right)\right]\les\eps^2,
\end{align}
provided we choose $\sig_4+22\le \sig_3$.

Now we turn to ${\rm S}_{0}^{\rm HL}$. First consider ${\rm S}_{0\ne z}^{\rm HL}$. Similar to \eqref{e-S0nezHL}, we are led to
\be\label{e-S0nezHL-4}
\begin{aligned}
    \la t\ra^{\fr52}&e^{\lm(t)|\eta,\al|^s}\la\eta,\al\ra^{\sig_4}|{\rm S}^{\rm HL}_{0\ne z}|\\
    \les&\sup_{\tau}\sum_{\beta\ne0,\beta\ne\al} \int_{\xi}e^{\lm(\tau)|\xi,\beta|^s}\la\xi,\beta\ra^{\sig_4+16}\la|\xi|^\fr12,\beta\ra|\widehat{\mathbb{P}_{\ne}^z\theta_0}(\tau,\xi,\beta)|\\
    &\qquad \qquad \times\la\tau\ra^{\fr52}e^{c\lm(\tau)|\eta-\xi,\al-\beta|^s}\la \eta-\xi,\al-\beta\ra^2|\reallywidehat{\nb_{yz}\mathbb{P}^z_{\ne}\theta_0}(\tau,\eta-\xi,\al-\beta)|d\xi \fr{1}{\la\eta,\al\ra^2}\\
    \les&\sup_\tau \|B\theta_0\|_{\sig_1-2}\left\|\sup_\tau \la\tau\ra^{\fr52}e^{\lm(\tau)|\eta,\al|^s}\la \eta,\al\ra^{\sig_4}|\widehat{\mathbb{P}_{\ne}^z\theta_0}(\tau,\eta,\al)| \right\|_{L^2_{\eta,\al}}\fr{1}{\la\eta,\al\ra^2},
\end{aligned}
\ee
and hence,
\[
\left\|\sup_{t}{\mathds{1}}_{\al\ne0}\la t\ra^{\fr52}e^{\lm(t)|\eta,\al|^s}\la\eta,\al\ra^{\sig_4} {\rm S}^{\rm HL}_{0\ne z}\right\|_{L^2_{\eta,\al}}\les\eps^2,
\]
provide we choose $\sig_1-2\ge \sig_4+16$.

To bound ${\rm S}_{00z}^{\rm HL}$, similar to \eqref{Isig2}, we have the following split: 
\be\label{Isig4}
\begin{aligned}
\la t\ra^\fr52 &e^{\lm(t)|\eta,\al|^s}\la\eta,\al\ra^{\sig_4}|{\rm S}_{00z}^{\rm HL}|\\
\les&\int_{\fr{t}{2}}^t\fr{\al^2}{(\eta^2+\al^2)^2}\mathbb{S}(t-\tau;\eta,\al)\int_{\xi}\la\tau\ra^\fr52e^{\lm(\tau)| \xi,\al|^s}\la \xi,\al\ra^{\sig_4}|\widehat{\Pe^z_{\ne}\theta_0}(\tau, \xi,\al)|\\
&\qquad \times e^{c\lambda(\tau)| \eta-\xi|^s}\la\eta-\xi\ra^7 |\widehat{\theta}_0(\tau,\eta-\xi,0)| \fr{1}{\la\eta-\xi\ra^2}d\xi d\tau\\
&+\int_{10}^\fr{t}{2}\fr{\al^2}{(\eta^2+\al^2)^2}\left(\fr{\al^2\la t-\tau\ra}{(\eta^2+\al^2)}\right)^\fr52 \int_{\xi}e^{\lm(\tau)| \xi,\al|^s}\la \xi,\al\ra^{\sig_4+12}|\widehat{\Pe^z_{\ne}\theta_0}(\tau, \xi,\al)|\\
&\qquad \times e^{c\lm(\tau)| \eta-\xi|^s}\la\eta-\xi\ra^7 |\widehat{\theta}_0(\tau,\eta-\xi,0)|d\xi \fr{1}{\la\eta\ra^2} d\tau={\rm I}_1^{\sig_4}+{\rm I}_2^{\sig_4}.
\end{aligned}
\ee
Note that  ${\rm I}_2^{\sig_4}$  above is defined in a different way than ${\rm I}_2^{\sig_2}$. However, ${\rm I}_1^{\sig_4}$ can be treated in the same manner as ${\rm I}_1^{\sig_2}$:
\begin{equation}\label{bd-I1-4}
\left\|\sup_t {\rm I}_1^{\sig_4}\right\|_{L^2_{\eta,\al}}\les\eps\left\|\sup_\tau \la\tau\ra^{\fr52}e^{\lm(\tau)| \eta,\al|^s}\la \eta,\al\ra^{\sig_4}|\widehat{\Pe^z_{\ne}\theta_0}(\tau, \eta,\al)|\right\|_{L^2_{\eta,\al}}\les \eps^2.
\end{equation}
Now ${\rm I}_2^{\sig_4}$  can be treated without resorting to \eqref{gap1} and \eqref{En-7}. In fact, we will take a strategy similar to that used in the treatment of ${\rm S}_0^{\rm LH}$ and ${\rm S}_{\ne}$. More precisely, using again \eqref{semigroup-bd} and Cauchy-Schwarz in $\xi$ variable, we find that
\[
\begin{aligned}
|{\rm I}^{\sig_4}_2|\les&\sup_\tau \left\|e^{\lm(\tau)| \xi,\al|^s}\la \xi,\al\ra^{\sig_4+12}|\widehat{\Pe^z_{\ne}\theta_0}(\tau, \xi,\al)|\right\|_{L^2_\xi}\\
&\times \sup_\tau\left\|e^{c\lm(\tau)| \eta|^s}\la\eta\ra^7 |\widehat{\theta}_0(\tau,\eta,0)|\right\|_{L^2_\eta} \fr{1}{\la\eta\ra^2}\\
\les& \left\|\sup_\tau e^{\lm(\tau)| \xi,\al|^s}\la \xi,\al\ra^{\sig_2}|\widehat{\Pe^z_{\ne}\theta_0}(\tau, \xi,\al)|\right\|_{L^2_\xi} \sup_\tau\left\| \theta(\tau)\right\|_{\mathcal{G}^{\lm,\sig_5}} \fr{1}{\la\eta\ra^2},
\end{aligned}
\]
and thus,
\begin{equation}\label{e-I4}
\left\|\sup_t{\rm I}_2^{\sig_4} \right\|_{L^2_{\eta,\al}}\les\eps
\left\|\sup_\tau \la\tau\ra^{\fr32}e^{\lm(\tau)| \eta,\al|^s}\la \eta,\al\ra^{\sig_2}|\widehat{\Pe^z_{\ne}\theta_0}(\tau, \eta,\al)|\right\|_{L^2_{\eta,\al}}\les\eps^2,
\end{equation}
provide we choose $\sig_2\ge\sig_4+12$.

\subsection{Improvement of\eqref{En-5}}\label{sec-im-En-5}
Note that \eqref{e-theta-sig3} still holds with $\sig_3$ replaced by $\sig_5$. Thus we will continue the notation from  Section \ref{sec-im-En-3}. ${\bf NL}_{\sig_5;\ne}$ can be treated in the same way as ${\bf NL}_{\sig_3;\ne}$. Indeed, similar to \eqref{e-NLHL3} and \eqref{e-NLLH3}, we have
\begin{align*}
{\bf NL}_{\sig_5;\ne}^{\rm HL}\les&\sum_{k,\al,\beta\in\Z,l\ne0}\int_{\eta,\xi}e^{\lm(t)|l, \xi,\beta|^s}\la l, \xi,\beta\ra^{\sig_5+6}\fr{1}{\la t\ra^3}|\widehat{\theta}_{\ne}(t,l,\xi, \beta)| \\
&\times e^{c\lm(t)|k-l,\eta-\xi,\al-\beta|^s}|\widehat{\nb_{xyz}\theta}_{k-l}(t,\eta-\xi,\al-\beta)|\\
&\times  e^{\lm(t)| k,\eta,\al|^s}\la k,\eta,\al\ra^{\sig_5}|\widehat{\theta}_k(t,\eta,\al)|d\xi d\eta\\
\les&\fr{1}{\la t\ra^{\fr32}}\|t^{-\fr12}\theta \|_{\sig_3}\|\theta\|_{\mathcal{G}^{\lm,\sig_5}}^2\les\fr{\eps^3}{\la t\ra^{\fr52}},
\end{align*}
and
\be
\begin{aligned}
{\bf NL}_{\sig_5;\ne}^{\rm LH}\les&\sum_{k\ne l,\al,\beta\in\Z}\int_{\eta,\xi}e^{c\lm(t)|k-l, \eta-\xi,\al-\beta|^s}\la k-l, \eta-\xi,\al-\beta\ra^{6}\fr{1}{\la t\ra^3}|\widehat{\theta}_{k-l}(t,\eta-\xi, \al-\beta)| \\
&\times e^{\lm(t)|l,\xi,\beta|^s}|l,\xi,\beta|^{\sig_5}\widehat{\nb_{xyz}\theta}_{l}(t,\xi,\beta)|\\
&\times  e^{\lm(t)| k,\eta,\al|^s}\la k,\eta,\al\ra^{\sig_5}|\widehat{\theta}_k(t,\eta,\al)|d\xi d\eta\\
\les&\fr{1}{\la t\ra^{\fr32}}\|t^{-\fr12}\theta \|_{\sig_3}\|\theta\|_{\mathcal{G}^{\lm,\sig_5}}^2\les\fr{\eps^3}{\la t\ra^{\fr52}}.
\end{aligned}
\ee

For ${\bf NL}_{\sig_5;0}$, using again \eqref{bd-u0hat}, we find that
\begin{align}
{\bf NL}_{\sig_5;0}^{\rm HL}
\nn\les&\sum_{k,\al,\beta\in\Z}\int_{\eta,\xi}e^{\lm(t)| \xi,\beta|^s}\la \xi,\beta\ra^{\sig_5-2}\la t\ra |\Pe^z_{\ne}\widehat{\theta}_0(t,\xi, \beta)|\\
\nn&\times e^{c\lm(t)| k,\eta-\xi,\al-\beta|^s}\left|\widehat{\nb_{xyz}\theta}_k(t,\eta-\xi,\al-\beta) \right|\left|e^{\lm(t)| k,\eta,\al|^s}\la k,\eta,\al\ra^{\sig_5}\widehat{\theta}_k(t,\eta,\al)\right|d\xi d\eta\\
\nn\les&\fr{1}{\la t\ra^{\fr32}}\left(\la t\ra^{\fr52}\|\Pe^z_{\ne}\theta_0(t)\|_{\mathcal{G}^{\lm,\sig_4}}\right)\|\theta(t)\|_{\mathcal{G}^{\lm,\sig_5}}^2\les\fr{\eps^3}{\la t\ra^{\fr32}},
\end{align}
and
\begin{align}
{\bf NL}_{\sig_5;0}^{\rm HL}
\nn\les&\sum_{k,\al,\beta\in\Z}\int_{\eta,\xi}e^{c\lm(t)| k,\eta-\xi,\al-\beta|^s}\la t\ra |\Pe^z_{\ne}\widehat{\theta}_0(t,\eta-\xi, \al-\beta)|\\
\nn&\times e^{\lm(t)| \xi,\beta|^s}\la \xi,\beta\ra^{\sig_5}\left|\widehat{\nb_{xyz}\theta}_k(t,\xi,\beta) \right|\left|e^{\lm(t)| k,\eta,\al|^s}\la k,\eta,\al\ra^{\sig_5}\widehat{\theta}_k(t,\eta,\al)\right|d\xi d\eta\\
\nn\les&\fr{1}{\la t\ra^{\fr32}}\left(\la t\ra^{3}\|\Pe^z_{\ne}\theta_0(t)\|_{\mathcal{G}^{\lm,\sig_6}}\right)\left(\la t\ra^{-\fr12}\|\theta(t)\|_{\mathcal{G}^{\lm,\sig_3}}\right)\|\theta(t)\|_{\mathcal{G}^{\lm,\sig_5}}\les\fr{\eps^3}{\la t\ra^{\fr32}}.
\end{align}

\subsection{Improvement of \eqref{En-6}}\label{sec:im-En6}
Similar to \eqref{S20LH} and \eqref{S40LH}, we arrive at
\begin{equation}\label{S60LH}
\begin{aligned}
&\la t\ra^{3}e^{\lm(t)|\eta,\al|^s}\la \eta,\al\ra^{\sig_6}|{\rm S}_{0}^{\rm LH}|\\
\lesssim & \left\|\sup_{\tau}\Big|\JB{\tau}^3 e^{\lambda(\tau)|\eta,\al|^{s}}\JB{\eta,\al}^{\sigma_6}\mathbb{P}_{\neq}^z\widehat{\theta}_0(\tau,\eta,\al)\Big|\right\|_{L^2_{\eta,\al}}\sup_\tau\left(\|B\theta_0(\tau)\|_{\sig_1-2}\right)\fr{1}{\la\eta,\al\ra^2},
\end{aligned}
\end{equation}
provided $\sig_1-2\ge \sig_6+19$, and hence 
\[
\begin{aligned}
&\left\|\sup_t{\mathds{1}}_{\al\ne0}\la t\ra^3 e^{\lm(t)|\eta,\al|^s}\la \eta,\al\ra^{\sig_6}|{\rm S}_{0}^{\rm LH}|\right\|_{L^2_{\eta,\al}}\les\eps^2.
\end{aligned}
\]

For ${\rm S}_{\ne}$, similar to \eqref{Sne4HL} and \eqref{Sne4LH}, if $\al\ne0$, we have
\begin{align}\label{Sne6HL}
\la t \ra^{3}e^{\lm(t)|\eta,\al|^s}\la \eta,\al\ra^{\sig_6}|{\rm S}_{\ne}^{\rm HL}|
\nn\les&\sup_\tau \sum_{k,\beta\in\Z}\int_\xi e^{\lm(\tau)| \xi,\beta|^s}\la k, \xi,\beta\ra^{\sig_6+24}\left|\widehat{\theta}_{\ne}(\tau, k, \xi,\beta)\right|\\
\nn&\times \left(e^{\lm(\tau)|\eta- \xi,\al-\beta|^s}\left|\widehat{\theta}_{\ne}(\tau, -k,\eta-\xi,\al-\beta)\right|\right)d\xi \fr{1}{\la\eta,\al\ra^2}\\
\les&\sup_\tau\|\theta(\tau)\|_{\mathcal{G}^{\lm,\sig_5}}^2\fr{1}{\la\eta,\al\ra^2},
\end{align}
and
\begin{align}\label{Sne6LH}
\la t \ra^{3}&e^{\lm(t)|\eta,\al|^s}\la \eta,\al\ra^{\sig_6}|{\rm S}_{\ne}^{\rm LH}|\\
\nn\les&\sup_\tau \sum_{k,\beta\in\Z}\int_\xi e^{c\lm(\tau)|  \eta-\xi,\al-\beta|^s}\la k, \eta-\xi,\al-\beta\ra^8\left|\widehat{\theta}_{\ne}(\tau, k, \eta-\xi,\al-\beta)\right|\\
\nn&\times \left(e^{\lm(\tau)| \xi,\beta|^s}\la \xi,\beta\ra^{\sig_6+18}\left|\widehat{\nb_{xyz}\theta}_{\ne}(\tau,-k,\xi,\beta)\right|\right)d\xi\fr{1}{\la\eta,\al\ra^2}\\
\les&\sup_\tau\|\theta(\tau)\|_{\mathcal{G}^{\lm,\sig_5}}^2\fr{1}{\la\eta,\al\ra^2}.
\end{align}
Accordingly,
\begin{align}
\left\|\sup_t{\mathds{1}}_{\al\ne0}\la t \ra^{3}e^{\lm(t)|\eta,\al|^s}\la \eta,\al\ra^{\sig_6}{\rm S}_{\ne}\right\|_{L^2_{\eta,\al}}\les\sup_t\|\theta(t)\|_{\mathcal{G}^{\lm,\sig_5}}^2\les\eps^2,
\end{align}
provided we choose $\sig_6+24\le \sig_5$.

New we consider ${\rm S}_{0\ne z}^{\rm HL}$. Similar to \eqref{e-S0nezHL} and \eqref{e-S0nezHL-4}, one deduces that
\be\label{e-S0nezHL-6}
\begin{aligned}
    &\la t\ra^{3}e^{\lm(t)|\eta,\al|^s}\la\eta,\al\ra^{\sig_6}|{\rm S}^{\rm HL}_{0\ne z}|\\
    \les&\sup_{\tau}\sum_{\beta\ne0,\beta\ne\al} \int_{\xi}e^{\lm(\tau)|\xi,\beta|^s}\la\xi,\beta\ra^{\sig_6+18}\la|\xi|^\fr12,\beta\ra|\widehat{\mathbb{P}_{\ne}^z\theta_0}(\tau,\xi,\beta)|\\
    &\times\la\tau\ra^{3}e^{c\lm(\tau)|\eta-\xi,\al-\beta|^s}\la \eta-\xi,\al-\beta\ra^2|\widehat{\nb_{yz}\mathbb{P}^z_{\ne}\theta_0}(\tau,\eta-\xi,\al-\beta)|d\xi \fr{1}{\la\eta,\al\ra^2}\\
    \les&\sup_\tau \|B\theta_0\|_{\sig_1-2}\left\|\sup_\tau \la\tau\ra^{3}e^{\lm(\tau)|\eta,\al|^s}\la \eta,\al\ra^{\sig_6}|\widehat{\mathbb{P}_{\ne}^z\theta_0}(\tau,\eta,\al)| \right\|_{L^2_{\eta,\al}}\fr{1}{\la\eta,\al\ra^2},
\end{aligned}
\ee
and hence,
\[
\left\|\sup_{t}{\mathds{1}}_{\al\ne0}\la t\ra^{3}e^{\lm(t)|\eta,\al|^s}\la\eta,\al\ra^{\sig_6} {\rm S}^{\rm HL}_{0\ne z}\right\|_{L^2_{\eta,\al}}\les\eps^2,
\]
provide we choose $\sig_1-2\ge \sig_6+18$.

We are left to treat ${\rm S}_{00z}^{\rm HL}$. It suffices to bound ${\rm I}^{\sig_6}_1$ and ${\rm I}^{\sig_6}_2$ defined analogously to those in \eqref{Isig4}. We can repeat what we did for ${\rm I}^{\sig_2}_1$ and ${\rm I}^{\sig_4}_1$ to get
\be\label{bd-I1-6}
\left\|\sup_t {\rm I}_1^{\sig_6}\right\|_{L^2_{\eta,\al}}\les\eps\left\|\sup_\tau \la\tau\ra^{3}e^{\lm(\tau)| \eta,\al|^s}\la \eta,\al\ra^{\sig_6}|\Pe^z_{\ne}\widehat{\theta}_0(\tau, \eta,\al)|\right\|_{L^2_{\eta,\al}}\les \eps^2.
\ee
Finally, for ${\rm I}^{\sig_6}_2$, similar to the treatment of ${\rm I}^{\sig_4}_2$, we have 
\[
\begin{aligned}
|{\rm I}^{\sig_6}_2|\les&\sup_\tau \left\|e^{\lm(\tau)| \xi,\al|^s}\la \xi,\al\ra^{\sig_6+14}|\Pe^z_{\ne}\widehat{\theta}_0(\tau, \xi,\al)|\right\|_{L^2_\xi}\\
&\times \sup_\tau\left\|e^{c\lm(\tau)| \eta|^s}\la\eta\ra^7 |\widehat{\theta}_0(\tau,\eta,0)|\right\|_{L^2_\eta} \fr{1}{\la\eta\ra^2}\\
\les& \left\|\sup_\tau e^{\lm(\tau)| \xi,\al|^s}\la \xi,\al\ra^{\sig_4}|\Pe^z_{\ne}\widehat{\theta}_0(\tau, \xi,\al)|\right\|_{L^2_\xi} \sup_\tau\left\| \theta(\tau)\right\|_{\mathcal{G}^{\lm,\sig_5}} \fr{1}{\la\eta\ra^2},
\end{aligned}
\]
and thus,
\begin{equation}\label{e-I6}
\left\|\sup_t{\rm I}_2^{\sig_6} \right\|_{L^2_{\eta,\al}}\les\eps
\left\|\sup_\tau \la\tau\ra^{\fr52}e^{\lm(\tau)| \eta,\al|^s}\la \eta,\al\ra^{\sig_4}|\Pe^z_{\ne}\widehat{\theta}_0(\tau, \eta,\al)|\right\|_{L^2_{\eta,\al}}\les\eps^2,
\end{equation}
provide we choose $\sig_4\ge\sig_6+14$.

\subsection{Improvement of \eqref{En-7}}\label{sec:im-En7}
Taking $\al=0$ in \eqref{eq-theta0hat}, we get the equation for $\mathbb{P}^z_0\theta_0$:
\begin{equation}\label{eq: theta_00}
    \pr_t\widehat{\mathbb{P}^z_0\theta}_0(t,\eta)={\rm S}_{00}+{\rm S}_{\ne\ne},
\end{equation}
with
\be\label{S00}
\begin{aligned}
{\rm S}_{00}=&-\mathcal{F}[\tl u_0\cdot\nb_{yz}\theta_0](t,\eta,0)\\
=&i\sum_{\beta\ne0}\int_{\xi}\fr{\beta^2}{((\eta-\xi)^2+\beta^2)^2}\widehat{\theta}_0(t,\eta-\xi,-\beta)\xi\widehat{\theta}_0(t,\xi,\beta)d\xi\\
&+i\sum_{\beta\ne0}\int_{\xi}\fr{(\eta-\xi)\beta}{((\eta-\xi)^2+\beta^2)^2}\widehat{\theta}_0(t,\eta-\xi,-\beta)\beta\widehat{\theta}_0(t,\xi,\beta)d\xi,
\end{aligned}
\ee
and
\be\label{Snene}
\begin{aligned}
{\rm S}_{\ne\ne}=-\mathcal{F}[(u_{\ne}\cdot\nb_{xyz}\theta_{\ne})_0](t,\eta,0)=&-\sum_{k\ne0,\beta\in\Z }\int_{\xi}\widehat{u}_{\ne}(-k,\eta-\xi,-\beta)\cdot\widehat{\nb_{xyz}\theta}_{\ne}(k,\xi,\beta)d\xi.
\end{aligned}
\ee
Thus, we have 
\be\label{exp-theta00-hat}
\widehat{\Pe_0^z\theta}_0(t,\eta)=\widehat{\Pe_0^z\theta}_0(10,\eta)+\int_{10}^t\big({\rm S}_{00}+{\rm S}_{\ne\ne}\big)d\tau. 
\ee
It is easy to see that
\begin{align*}
\sup_{t}\sup_{\eta}\left|e^{\lm(t)|\eta|^s}\la\eta\ra^{\sig_7}\widehat{\Pe_0^z\theta}_0(10,\eta)\right|\le\sup_{t}\left|e^{\lm(10)|\eta|^s}\la\eta\ra^{\sig_7}\widehat{\Pe_0^z\theta}_0(10,\eta)\right|{   \le\epsilon}.
\end{align*}
The rest two terms in \eqref{exp-theta00-hat} can be treated by using the following lemma.
\begin{lem}\label{lem: est-so}
 For any $K>1$, the following two estimates hold: 
 \be
\begin{aligned}\label{ineq1}
&\left\|e^{\lm(t)|\eta|^s}\la\eta\ra^{\sig_7}\sum_{k, \beta\in\Z}\int_{10}^t \fr{1}{\JB{\tau}^{K}}\int_{\xi}\widehat{f}(\tau,-k,\eta-\xi,-\beta)\widehat{g}(\tau,k,\xi,\beta) d\xi d\tau \right\|_{L^{\infty}_tL^{\infty}_{\eta}}\\
\lesssim &\Big(\sup_{t\in [10,T^{*}]}\|f\|_{\mathcal{G}^{\lambda(t),\sigma_7}}\Big)\Big(\sup_{t\in [10,T^{*}]}\|g\|_{\mathcal{G}^{\lambda(t),\sigma_7}}\Big),
\end{aligned}
\ee
and
\be
\begin{aligned}\label{ineq2}
&\left\|e^{\lm(t)|\eta|^s}\la\eta\ra^{\sig_7}\sum_{ \beta\in\Z}\int_{10}^t \fr{1}{\JB{\tau}^{K}}\int_{\xi}\widehat{f}_0(\tau,\eta-\xi,-\beta)\widehat{g}_0(\tau,\xi,\beta) d\xi d\tau \right\|_{L^{\infty}_tL^{\infty}_{\eta}}\\
\lesssim &\left\| \sup_{t\in [10,T^{*}]}\Big|e^{\lambda(t)|\eta,\al|^s}\JB{\eta,\al}^{\sigma_7}\widehat{f}_0(t,\eta,\al)\Big|\right\|_{L^2_{\eta,\al}}\left\| \sup_{t\in [10,T^{*}]}\Big|e^{\lambda(t)|\eta,\al|^s}\JB{\eta,\al}^{\sigma_7}\widehat{g}_0(t,\eta,\al)\Big|\right\|_{L^2_{\eta,\al}}. 
\end{aligned}
\ee
\end{lem}
\begin{proof}
By using triangle inequality and Cauchy-Schwarz, one easily deduces that for all $t$ and $\eta$,
\begin{align*}
&e^{\lm(t)|\eta|^s}\la\eta\ra^{\sig_7}\left|\sum_{k, \beta\in\Z}\int_{10}^t \fr{1}{\JB{\tau}^{K}}\int_{\xi}\widehat{f}(\tau,-k,\eta-\xi,-\beta)\widehat{g}(\tau,k,\xi,\beta) d\xi d\tau \right|\\
\les&\int_{10}^t \fr{1}{\JB{\tau}^{K}}d\tau\sup_\tau\sum_{k, \beta\in\Z}\int_{\xi}e^{\lm(\tau)|\eta-\xi|^s}\la\eta-\xi\ra^{\sig_7}|\widehat{f}(\tau,-k,\eta-\xi,-\beta)|e^{\lm(\tau)|\xi|^s}\la\xi\ra^{\sig_7}|\widehat{g}(\tau,k,\xi,\beta)| d\xi\\
\les&\sup_\tau \left(\left\|e^{\lm(\tau)|\eta|^s}\la\eta\ra^{\sig_7}\widehat{f}(\tau,k,\eta,\al)\right\|_{L^2_{k,\eta,\al}}\left\|e^{\lm(\tau)|\eta|^s}\la\eta\ra^{\sig_7}\widehat{g}(\tau,k,\eta,\al)\right\|_{L^2_{k,\eta,\al}}\right),
\end{align*}
then \eqref{ineq1} follows immediately. Combining the strategy in the proof of \eqref{ineq1} above with Remark \ref{rmk: Minkowski}, we get \eqref{ineq2}.
\end{proof}
Now we are in a position to estimate the last two terms \eqref{exp-theta00-hat}. Recalling \eqref{u_ne-decay} and \eqref{Snene}, we have
\begin{align*}
&\left|e^{\lm(t)|\eta|^s}\la\eta\ra^{\sig_7}\int_{10}^t{\rm S}_{\ne\ne} d\tau\right|\\
\les& \sum_{k\ne0,\beta\in\Z}\int_{10}^t \fr{1}{\la\tau\ra^3}\int_{\xi}e^{\lm(\tau)|\eta-\xi|^s}\la \eta-\xi \ra^{\sig_7}|\widehat{\JB{\nabla}^6\theta}_{\ne}(\tau,-k,\eta-\xi,-\beta)| \\
&\quad\quad \times e^{\lm(\tau)|\xi|^s}\la k,\xi,\beta \ra^{\sig_7}|\widehat{\nabla \theta}_{\ne}(\tau,k,\xi,\beta)|d\xi d\tau. 
\end{align*}
Then applying Lemma \ref{lem: est-so} by taking $K=3$, $f=\JB{\nabla}^6\theta_{\ne}$, $g=\nabla  \theta_{\ne}$, and the bootstrap hypotheses \eqref{En-5}, we get that 
\begin{align*}
    \left\|e^{\lm(t)|\eta|^s}\la\eta\ra^{\sig_7}\int_{10}^t{\rm S}_{\ne\ne} d\tau\right\|_{L^{\infty}_tL^{\infty}_{\eta}}\lesssim \epsilon^2,
\end{align*}
provided we choose $\sig_5\ge\sig_7+6$.

Finally, recalling the definition of $S_{00}$ in \eqref{S00}, we infer that 
\begin{align*}
    &\left|e^{\lm(t)|\eta|^s}\la\eta\ra^{\sig_7}\int_{10}^t{\rm S}_{00} d\tau\right|\\
    \lesssim& \sum_{\beta\neq 0}\int_{10}^t \fr{1}{\JB{\tau}^6}\int_{\xi} \big|e^{\lm(\tau)|\eta-\xi|^s}\la\xi-\eta,\beta\ra^{\sig_7}\JB{\tau}^3\widehat{\JB{\nabla}^{-1}\theta}(\tau,\eta-\xi,-\beta)\big|\\
    &\quad\times \big|e^{\lm(\tau)|\xi|^s}\la\xi,\beta\ra^{\sig_7}\JB{\tau}^3\widehat{\JB{\nabla}\theta}(\tau,\xi,\beta)\big|d\xi d\tau. 
\end{align*}
Then applying Lemma \ref{lem: est-so} by taking $K=6$, $f_0=\JB{\tau}^3\JB{\nabla}^{-1}\mathbb{P}_{\neq}^z\theta_0$, $g_0=\JB{\tau}^3\JB{\nabla}\mathbb{P}_{\neq}^z\theta_0$, and the bootstrap hypotheses \eqref{En-6}, we get that 
\begin{align*}
    \left\|e^{\lm(t)|\eta|^s}\la\eta\ra^{\sig_7}\int_{10}^t{\rm S}_{00} d\tau\right\|_{L^{\infty}_tL^{\infty}_{\eta}}\lesssim \epsilon^2,
\end{align*}
provide we choose $\sig_6\ge\sig_7+1$.

\appendix
\section{Various estimates}

The following two lemmas below ultimately predict the growth of high frequencies which signal the loss of Gevrey-2 regularity. 
\begin{lemma}\label{total growth}
    Suppose that $|\iota|>1$. Then there exists $\mu=4(1+2C_*)$ such that 
    \begin{equation}\label{growth of Theta}
          \dfrac{w_k(2\iota,\iota)}{w_k(0,\iota)}=\dfrac{1}{w_k(0,\iota)}=\dfrac{1}{w_k\big(t_{E(|\iota|^{\frac{1}{2}}),\iota},\iota\big)}\lesssim e^{\fr{\mu}{2}|\iota|^{\frac{1}{2}}},
    \end{equation}
    where $C_*$ is the constant in the definition of $w_{\res}$ and $w_{\nr}$. 
\end{lemma}

\begin{lemma}[\cite{BM1}]\label{exponent ineq}
    Suppose that $0<s<1$ and $x,y \geq 0$.\\
        (1) If $x+y>0$, then 
           $ |x^s-y^s|\lesssim_s \dfrac{1}{x^{1-s}-y^{1-s}}|x-y|.$\\
        (2) If $|x-y|\leq \frac{x}{{\rm K}}$ for some ${\rm K}>1$, then 
        \begin{equation}\label{triangle1}
            |x^s-y^s|\leq \frac{s}{({\rm K}-1)^{1-s}}|x-y|^s.
        \end{equation}
        (3) More generally, it holds that
           $ |x+y|^s \leq \Big(\dfrac{x}{x+y}\Big)^{1-s} (x^s+y^s).$
            In particular, if $y\le x\le {\rm K}y$ for some $\rm K>0$, then
        \begin{align}\label{triangle2}
|x+y|^s\le\Big(\fr{\rm K}{1+\rm K}\Big)^{1-s}(x^s+y^s).
        \end{align}
\end{lemma}
\begin{lemma}[\cite{BM1}] \label{Scenarios}
    Let $\iota_1$ and $\iota_2$ be such that there exists a number ${\rm K}\geq 1$ such that ${\rm K}^{-1}|\iota_2|\leq |\iota_1|\leq {\rm K} |\iota_2|$ and let $k$ and $n$ be such that $t\in \textup{I}_{k,\iota_1} \cap \textup{I}_{n,\iota_2}$, then at least of the following conclusions holds:
 \begin{enumerate}
  \item $k=n$(almost the same interval). 
  \item $\big|t-\frac{\iota_1}{k}\big|\gtrsim_{\rm K} \frac{|\iota_1|}{k^2}$ and $\big|t-\frac{\iota_2}{n}\big|\gtrsim_{\rm K}\frac{|\iota_2|}{n^2}$(away from resonance),
  \item $\big|\iota_1-\iota_2 \big|\gtrsim_{\rm K} \frac{|\iota_1|}{|n|}$(well-separated).
\end{enumerate}
\end{lemma}

\begin{lemma}[\cite{BM1}]\label{lem-r-wNR}
For all $t,\iota_1,\iota_2$, it holds that
\[
\frac{w_{\rm{NR}}(t,\iota_1)}{w_{\rm{NR}}(t,\iota_2)} \lesssim e^{\mu |\iota_1-\iota_2|^{\fr12}}.
\]
\end{lemma}

\section{Local well-posedness}
\begin{lemma}\label{local wp}
Let $\fr{1}{2}<s\leq 1$. There exists $\epsilon_0$ such that for any $0<\epsilon<\epsilon_0$, if
$\norm{\theta_{\textup{in}}}_{\mathcal{G}^{\lambda_{\textup{in}}};s}\leq \epsilon$ and $\sup\limits_{\eta\in\mathbb{R}}\left|e^{\lambda_{\mathrm{in}}|\eta|^s}\widehat{\theta_{\mathrm{in}}}(0,\eta, 0)\right|<\epsilon$, 
then there exists a constant $C=C(\lambda_{\textup{in}},\sigma_1,s)$ such that 
\[
\sup_{t\in[0,10]}\norm{\theta(t)}_{\mathcal{G}^{\lambda_1(t),\sigma_0}}+\sup_{t\in [0,10]}\sup_{\eta\in \mathbb{R}}\left|e^{\lambda_1(t)|\eta|^s}\JB{\eta}^{\sigma_1}\widehat{\mathbb{P}^z_{0}\theta_0}(t,\eta)\right|\leq C\epsilon,
\]  
where $\sigma_0:=\sigma_1+10$ and $\lambda_1(t)=\lambda_1(0)-\fr{\lambda_{\textup{in}}t}{200}$, with $\lambda_1(0)=\fr{19}{20} \lambda_{\textup{in}}$.
\end{lemma}
\begin{proof}
    We recall that under the linear change of coordinate introduced in \eqref{linear change}, our density satisfies
     \begin{align}\label{evolution of theta}
      \partial_t\theta-\Delta_{xz}\Delta_{L}^{-2}\theta+u\cdot\nabla_{xyz}\theta=0,
      \end{align}
      and with the equation \eqref{exp-u} holding.
     Hence, we can write
      \begin{align*}
      \frac{1}{2}\frac{d}{dt}\norm{\theta(t)}^2_{\mathcal{G}^{\lambda_1,\sigma_0;s}}&=\int e^{\lambda_1|\nabla|^s}\JB{\nabla}^{\sigma_0}\theta e^{\lambda_1|\nabla|^s}\JB{\nabla}^{\sigma_0}\theta_t\;dxdydz+\fr{d\lambda_1}{dt}\norm{|\nabla|^{\fr{s}{2}}\theta(t)}^2_{\mathcal{G}^{\lambda_1,\sigma_0;s}}\\&
      :=I_1 +I_2. 
      \end{align*}
      Via \eqref{evolution of theta}, we obtain
      \begin{align*}
      I_1&=\int e^{\lambda_1|\nabla|^s}\JB{\nabla}^{\sigma_0}\theta e^{\lambda_1|\nabla|^s}\JB{\nabla}^{\sigma_0}\left(\Delta_{xz}\Delta_{L}^{-2}\theta-u\cdot\nabla_{xyz}\theta\right)\;dxdydz\\&
      = \int e^{\lambda_1|\nabla|^s}\JB{\nabla}^{\sigma_0}\theta e^{\lambda_1|\nabla|^s}\JB{\nabla}^{\sigma_0}\left(\Delta_{xz}\Delta_{L}^{-2}\theta\right)\; dxdydz \\&
      \qquad +\int \nabla \cdot u \left| e^{\lambda_1|\nabla|^s}\JB{\nabla}^{\sigma_0} \theta \right|^2 \;dxdydz\\&
      \qquad -\int e^{\lambda_1|\nabla|^s}\JB{\nabla}^{\sigma_0}\theta \left(e^{\lambda_1|\nabla|^s}\JB{\nabla}^{\sigma_0}(u\cdot\nabla\theta)-u\cdot \nabla  e^{\lambda_1|\nabla|^s}\JB{\nabla}^{\sigma_0}\theta\right) \; dxdydz\\&
      =I_1^{(1)} +I_1^{(2)}+I_1^{(3)}
      \end{align*}
Note that  $I_1^{(1)}\leq 0$. During $t\in[0,10]$, the second term  $I_1^{(2)}$ can be estimated  as follows
\[
|I_1^{(2)}|\lesssim \norm{\theta(t)}^3_{\mathcal{G}^{\lambda_1,\sigma_0;s}}
\]

The third term $I_1^{(3)}$ can be decomposed as before into three parts via the paraproduct decomposition as performed in \eqref{paraproduct decomposition nonzero mode} for instance. This yields
\[
I_1^{(3)}=\frac{1}{2\pi} \sum_{\N\geq 8} \mathfrak{T}_\mathrm{N} + \frac{1}{2\pi} \sum_{\N\geq 8} \mathfrak{R}_\mathrm{N} +\mathcal{R},
\]
where 
\begin{align*}
\sum_{\N\geq 8}\mathfrak{T}_{\N} &\lesssim \sum_{\N\geq 8} \sum_{\alpha,\beta,k,l\in\mathbb{Z}} \int e^{\lambda_1|k,\eta,\al|^s}\JB{k,\eta,\al}^{\sigma_0}|\widehat{\theta}(t,k,\eta,\al)|\left|\fr{e^{\lambda_1|k,\eta,\al|^s}\JB{k,\eta,\al}^{\sigma_0}}{e^{\lambda_1|l,\xi,\beta|^s}\JB{l,\xi,\beta}^{\sigma_0}}-1\right|\\&
\qquad \qquad \qquad \times |\widehat{u}(t,k-l,\eta-\xi,\al-\beta)_{<\N/8}||l,\xi,\beta|e^{\lambda_1|l,\xi,\beta|^s}\JB{l,\xi,\beta}^{\sigma_0} |\widehat{\theta}(t,l,\xi,\beta)_\mathrm{N}|d\xi d\eta 
\end{align*}
On the support of the integrand above, the following inequality holds
\[
\frac{3}{16}|l,\xi,\beta|\leq|k,\eta,\al|\leq \frac{19}{16}|l,\xi,\beta|.
\]
Also, via \eqref{triangle1} and \eqref{triangle2}, we obtain
\[
||k,\eta,\alpha|^s-|l,\xi,\beta|^s|\lesssim\frac{|k-l,\eta-\xi,\alpha-\beta|}{|l,\xi,\beta|^{1-s}},\qquad |k,\eta,\alpha|^s\leq c_1|l,\xi,\beta|^s+c_2|k-l,\eta-\xi,\alpha-\beta|^s.
\]
Exploiting the above inequalities, we may infer that
\begin{align*}
\left|\fr{e^{\lambda_1|k,\eta,\al|^s}\JB{k,\eta,\al}^{\sigma_0}}{e^{\lambda_1|l,\xi,\beta|^s}\JB{l,\xi,\beta}^{\sigma_0}}-1\right|&\lesssim e^{c\lambda_1|k-l,\eta-\xi,\alpha-\beta|^s}\frac{|k-l,\eta-\xi,\alpha-\beta|}{|l,\xi,\beta|^{1-s}}\qquad
\\&\qquad +\frac{|k-l,\eta-\xi,\alpha-\beta|}{|l,\xi,\beta|}.
\end{align*}
During the time  $t\in[0,10]$, $
|\widehat{u}(t,l,\xi,\beta)_\mathrm{N}|\lesssim |\widehat{\theta}(t,l,\xi,\beta)_\mathrm{N}|$ which leads us to infer 
\[
\sum_{\N\geq 8}\left|\mathfrak{T}_\mathrm{N}\right|\lesssim \sum_{\N\geq 8}\norm{|\nabla|^{\fr{s}{2}}\theta_{\sim \N}}^2_{{\mathcal{G}^{\lambda_1,\sigma_0}}}\norm{\theta}_{\mathcal{G}^{\lambda_{1},\sigma_0}}.
\]
Now, we proceed and analyze $\mathfrak{R}_\mathrm{N}$.
\begin{align*}
\mathfrak{R}_\mathrm{N}&=-2\pi\int e^{\lambda_1|\nabla|^s}\JB{\nabla}^{\sigma_0}\theta \left(e^{\lambda_1|\nabla|^s}\JB{\nabla}^{\sigma_0}(u_\mathrm{N}\cdot\nabla\theta_{<\N/8})-u_\mathrm{N}\cdot \nabla  e^{\lambda_1|\nabla|^s}\JB{\nabla}^{\sigma_0}\theta_{<\N/8}\right) \; dxdydz\\&
=\mathfrak{R}^{(1)}_\mathrm{N}+\mathfrak{R}^{(2)}_\mathrm{N}.
\end{align*}
It is sufficient to estimate $\mathfrak{R}^{(1)}_\mathrm{N}$. Observe first of all, due to the fact that $t\in[0,10]$ and $\fr{\N}{2}\leq|l,\xi,\beta|\leq \fr{3\N}{2}$ and \eqref{exp-u}, we can infer that 
\[
|\widehat{u}(t,l,\xi,\beta)_\mathrm{N}|\lesssim |\widehat{\theta}(t,l,\xi,\beta)_\mathrm{N}|.
\]
and thus
\begin{align*}
\sum_{\N\geq 8}\left|\mathfrak{R}^{(1)}_\mathrm{N}\right|&\lesssim\sum_{\N\geq 8}\sum_{k,l,\alpha,\beta}\int e^{\lambda_1|k,\eta,\alpha|^s}\JB{k,\eta,\alpha}^{\sigma_0}|\widehat{\theta}(t,k,\eta,\alpha)|e^{\lambda_1|k,\eta,\alpha|^s}\JB{k,\eta,\alpha}^{\sigma_0}|\widehat{\theta}(t,l,\xi,\beta)_\mathrm{N}|\\& \qquad \qquad \times|\widehat{\nabla\theta}(t,k-l,\eta-\xi,\alpha-\beta)_{<\N/8}|d\xi d\eta\\&
\lesssim \sum_{\N\geq 8}\norm{\theta}_{{{\mathcal{G}^{\lambda_{1},\sigma_0}}}} \norm{\theta_\mathrm{N}}_{{{\mathcal{G}^{\lambda_{1},\sigma_0}}}} \norm{\theta_{\sim \N}}_{{\mathcal{G}^{\lambda_1,\sigma_0}}}\lesssim \norm{\theta}^3_{{\mathcal{G}^{\lambda_1,\sigma_0}}}.
\end{align*}
One can treat the term $\mathcal{R}$ in the same manner we treat the reaction term $\mathfrak{R}_\mathrm{N}$. Hence, we avoid redoing it here. 
Gathering all estimates  back and using the fact that $I_1^{(1)}\leq 0$, we arrive at
 \begin{align*}
      \frac{d}{dt}\norm{\theta(t)}^2_{\mathcal{G}^{\lambda_1,\sigma_0}} &\lesssim \fr{d\lambda_1}{dt} \norm{|\nabla|^{\fr{s}{2}}\theta}^2_{\mathcal{G}^{\lambda_1,\sigma_0}} + \norm{|\nabla|^{\fr{s}{2}}\theta}^2_{{\mathcal{G}^{\lambda_1,\sigma_0}}}\norm{\theta}_{\mathcal{G}^{\lambda_{1},\sigma_0}}+\norm{\theta(t)}^3_{\mathcal{G}^{\lambda_1,\sigma_0}}
      \lesssim \norm{\theta(t)}^3_{\mathcal{G}^{\lambda_1,\sigma_0}}.
\end{align*}
Thus,
\[
\sup_{t\in[0,10]}\norm{\theta(t)}^2_{\mathcal{G}^{\lambda_1,\sigma_0}}\lesssim \epsilon^2+\sup_{t\in[0,10]}
\int_0^t \norm{\theta(t)}^3_{\mathcal{G}^{\lambda_1,\sigma_0}},
\]
from which we can infer that $\sup_{t\in[0,10]}\norm{\theta(t)}^2_{\mathcal{G}^{\lambda_1,\sigma_0}}\leq C \epsilon^2.$ 

To estimate $\widehat{\mathbb{P}_0^z\theta_0}$, we use \eqref{eq: theta_00}. By using the Young inequality, we obtain that 
\begin{align*}
    \|e^{\lambda_1(t)|\eta|^s}\JB{\eta}^{\sigma_1}\mathrm{S}_{00}\|_{L^{\infty}_{t,\eta}}+\|e^{\lambda_1(t)|\eta|^s}\JB{\eta}^{\sigma_1}\mathrm{S}_{\neq\neq}\|_{L^{\infty}_{t,\eta}}\lesssim \|\theta\|_{\mathcal{G}^{\lambda_1(t),\sigma_1}}^2,
\end{align*}
which together with the initial condition completes the proof of the lemma. 
\end{proof}

\bigbreak
\noindent{\bf Acknowledgments: }
RZ is partially supported by NSF of China under  Grants 12222105. Part of this work was done when RZ was visiting NYU Abu Dhabi. He appreciates the hospitality of NYUAD.

\noindent{\bf Conflict of interest:}  We confirm that we do not have any conflict of interest. 

\noindent{\bf Data availibility:} The manuscript has no associated data.

\bibliographystyle{abbrv.bst} 
\bibliography{references.bib}

\end{document}